\documentclass[12pt]{article}
\usepackage[a4paper, margin=1in]{geometry}
\usepackage[utf8]{inputenc}
\usepackage{amsmath}
\usepackage{amssymb}
\usepackage{amsthm}
\usepackage{mathtools}
\usepackage{enumerate}
\usepackage{bbm}
\usepackage{amsmath}
\usepackage{bm}
\usepackage{hyperref}
\usepackage{dsfont}
\usepackage{color}
\usepackage[capitalise]{cleveref}
\usepackage{todonotes}
\usepackage{bm}
\usepackage{lineno}

\usepackage{float}  
\usepackage{authblk}

\usepackage{physics}
\newcommand{\factorial}[1]{#1!}

\newcommand{\econst}{\varepsilon}  
\newcommand{\bconst}{\beta}         
\newcommand{\Gind}{G}

\newcommand{\nk}{\frac{n}{2k-1}}    

\newcommand{\nkterm}[2]{\frac{#1}{2k-1}n + #2\beta n} 
\newcommand{\nknegterm}[2]{\frac{#1}{2k-1}n - (#2)\beta n}

\newtheorem{theorem}{Theorem}[section]
\newtheorem{problem}{Problem}
\newtheorem{corollary}[theorem]{Corollary}
\newtheorem{lemma}[theorem]{Lemma}

\newtheorem{definition}[theorem]{Definition}
\newtheorem{fact}[theorem]{Fact}
\newtheorem{claim}[theorem]{Claim}

\newcommand{\eps}{\varepsilon}

\newcommand{\de}{\delta}

\newcommand{\size}[1]{\lvert{#1}\rvert}

\newcommand{\N}{\mathbb{N}}

\title{Perfect tilings with the generalised triangle in $k$-graphs}

\author[1]{Weichan Liu \thanks{Supported by the Postdoctoral Fellowship Program of CPSF under Grant Number GZC20252020.}}
\author[2]{Xiangxiang Nie\thanks{Corresponding author.}} 
\author[1]{Donglei Yang}
\author[1]{Lin-Peng Zhang}
\affil[1]{\small School of Mathematics, Shandong University, Jinan, China.\\

{\tt  wcliu@sdu.edu.cn}\\{\tt dlyang@sdu.edu.cn}\\ {\tt lpzhangmath@163.com}}
\affil[2]{\small Data Science Institute, Shandong University, Jinan, China.\\

{\tt xiangxiangnie@sdu.edu.cn} }

\date{}

\begin{document}
%\linenumbers
\maketitle
\begin{abstract}
Denote by $T_k$ the generalised triangle, a $k$-uniform hypergraph on vertex set $\{1,2,\dots,2k-1\}$ with three edges $\{1,\dots,k-1,k\}$,$\{1,\dots,k-1,k+1\}$ and $\{k,k+1,\dots,2k-1\}$.
Recently, Bowtell, Kathapurkar, Morrison and Mycroft [arXiv: 2505.05606] established the exact minimum codegree threshold for perfect $T_3$-tilings in $3$-graphs. 
In this paper, we extend their result to all $k \geq 3$, determining the optimal minimum codegree threshold for perfect $T_k$-tilings in $k$-graphs. Our proof uses the lattice-based absorption method, as is usual, but develops a unified and effective approach to build transferrals for all uniformities, which is of independent interest.  
Additionally, we establish an asymptotically tight minimum codegree threshold for a rainbow variant of the problem.

\end{abstract}

\section{Introduction}
For an integer $k\ge 2$, a \textit{$k$-uniform hypergraph} (for short \textit{$k$-graph}) $H=(V,E)$ consists of a \textit{vertex set} $V$ and a collection $E$ of $k$-subsets (called \textit{edges}) of $V$. Given a $k$-graph $H = (V, E)$ and any $(k-1)$-subset $S \subseteq V$, we define the \emph{codegree} of $S$ as:
\[
d_H(S) := \left\lvert 
\left\{ v \in V \setminus S \mid S \cup \{v\} \in E(H) 
\right\} 
\right\rvert.
\]
The \emph{minimum codegree} of $H$ is
$\delta_{k-1}(H) := \min_{\substack{S \in \binom{V}{k-1} }} d_H(S).$
For convenience, we write $\delta(H)$ and $d(S)$ for  $\delta_{k-1}(H)$ and $d_{H} (S)$ respectively. When $k=2$, that is, when $H$ is a simple graph, the minimum codegree reduces to the minimum degree $\delta(H)$ as usual.

Let $H$ and $F$ be a $k$-graph on $n$ vertices and a $k$-graph on $s$ vertices, respectively. An \emph{$F$-tiling} in $H$ is a family of vertex-disjoint copies of $F$ in $H$. We call an $F$-tiling in $H$ \emph{perfect} if it covers all the vertices of $H$.
Note that for the case when $F$ is a single edge, a perfect $F$-tiling is exactly a perfect matching. 

Recall that the decision problem asking whether there exists a perfect matching in $3$-partite $3$-graphs is one of the first $21$ NP-Complete problems. Over the last few decades, the question of determining the minimum codegree condition for the existence of a perfect 
$F$-tiling constitutes one of the most fundamental objectives in extremal graph theory.
\begin{problem}\label{problem}
Let $F$ be a $k$-graph on $s$ vertices. What is the minimum codegree condition that guarantees a perfect $F$-tiling in a $k$-graph $H$ on $n$ vertices, where $n \in s\mathbb{N}$?
\end{problem}

The investigation of Problem~\ref{problem} in graphs originated in Dirac's seminal theorem on Hamiltonian cycles~\cite{Dirac}, which directly yields the sharp minimum degree condition $\delta(H) \geq \frac{n}{2}$ ensuring perfect $K_2$-tilings in graphs $H$. Subsequent breakthroughs establish sharp minimum degree thresholds for clique tilings: Corr\'{a}di and Hajnal \cite{CH} proved that $\delta(H) \geq \frac{2n}{3}$ suffices for perfect $K_3$-tilings, which was later extended to arbitrary cliques by Hajnal and Szemerédi \cite{HSz} who showed $\delta(H) \geq \left(1 - \frac{1}{t}\right)n$ guarantees perfect $K_t$-tilings.  For general graphs $F$, a significant breakthrough was achieved by Alon and Yuster~\cite{AY} who showed that $\delta(H) \geq \left(1 - \frac{1}{\chi(F)}+o(1)\right)n$ ensures perfect $F$-tilings, which was subsequently reinforced by Koml\'{o}s~\cite{Komlos}, Koml\'{o}s, S\'{a}rk\"{o}zy and Szemer\'{e}di~\cite{KSSz}. Notably, K\"{u}hn and Osthus~\cite{KO} characterize the minimum degree threshold up to an additive constant, establishing the current state of the art.

For graphs, comprehensive solutions to Problem~\ref{problem} are nowadays well-established. However, in contrast, the progress for $k$-graphs with $k \geq 3$ regarding Problem~\ref{problem} remains exceedingly scarce. 
Currently, substantial findings exist only in the particular case when $F$ is a $k$-partite $k$-graph. For perfect matchings (where $F$ is a single edge), R\"{o}dl, Ruci\'{n}ski and Szemer\'{edi}~\cite{RRS} established the minimum codegree threshold $\delta(H) \geq n/2 - c$ with $c \in \{1/2, 1, 3/2, 2\}$ determined by the parities of $k$ and $n$. This improved on a sequence of increasingly stronger approximations established earlier by R\"{o}dl, Ruci\'{n}ski and Szemer\'{edi}~\cite{RRS2}, K\"{u}hn and Osthus~\cite{ko_match}, and R\"{o}dl, Rucinski, Schacht and Szemer\'{edi}~\cite{rrss}.
For an arbitrary $k$-partite $k$-graph $F$, Mycroft~\cite{Mycroft} established a sufficient minimum codegree condition for the existence of a perfect $F$-tiling and proved that this condition is asymptotically best possible for a broad class of $k$-partite $k$-graphs including all complete $k$-partite $k$-graphs. 
In particular, Mycroft provided an asymptotic solution to a question of R\"{o}dl and Ruci\'{n}ski~\cite{RR} on a minimum codegree condition for the existence of a perfect $F$-tiling when $F$ is a loose cycle.
When $F$ is the 4-vertex 3-graph with two edges, this case was resolved earlier by K\"{u}hn and Osthus~\cite{ko_loose}. Subsequently, Czygrinow, DeBiasio, and Nagle~\cite{cdn} established the exact minimum codegree threshold for large $n$.
For the $(k+1)$-vertex $k$-graph with two edges, Gao, Han, and Zhao~\cite{GHZ} 
determined the exact minimum codegree threshold. 
Moreover, they strengthened the sublinear error term in Mycroft's general result 
and established a sharp minimum codegree threshold for $F$ being a loose cycle.

The theory for non-$k$-partite $k$-graphs remains significantly less developed. Consider the case $F = K_4^3$, Lo and Markstr\"{o}m~\cite{LM} established the asymptotically optimal minimum codegree threshold $\delta(H) \geq 3n/4 + o(n)$ for perfect $K_4^3$-tilings. Keevash and Mycroft~\cite{keevash15}  determined the exact threshold for large $n$, proving $\delta(H) \geq 3n/4 - 2$ when $8 \mid n$ and $\delta(H) \geq 3n/4 - 1$ otherwise. These advancements superseded earlier sufficient conditions by Czygrinow and Nagle~\cite{cn} and Pikhurko~\cite{pikhurko}.
For the $4$-vertex $3$-graph $F = K_4^{3-}$, Lo and Markstr\"{o}m~\cite{lm2} first established an asymptotically tight bound $\delta(H) \geq n/2 + o(n)$, which was later strengthened by Han, Lo, Treglown and Zhao~\cite{HLTZ} who established the exact threshold $\delta(H) \geq n/2 - 1$.
For higher uniformities, Han, Lo and Sanhueza-Matamala~\cite{hlsm} characterized asymptotically optimal minimum codegree thresholds for $F$-tilings where $k \geq 4$ is even and $F$ represents a long tight cycle satisfying specific divisibility constraints.

Denote by $T_k$ the generalised triangle, a $k$-uniform hypergraph on vertex set 
$\{1,2,\dots,\\2k-1\}$ with three edges
$\{1,\dots,k-1,k\}$,$\{1,\dots,k-1,k+1\}$ and $\{k,k+1,\dots,2k-1\}$.
In the case \(k=3\), the generalised triangle \(T_3\) plays a fundamental role in extremal graph theory, motivating diverse research directions. Bowtell, Kathapurkar, Morrison and Mycroft~\cite{BKMM} provide a systematic survey of results concerning this fundamental object. We refer readers to their comprehensive survey for further exploration. In \cite{BKMM}, they established the exact minimum codegree threshold for perfect $T_3$-tilings in $3$-graphs. 
In this paper, we extend their result to every integer $k \geq 3$, determining the optimal minimum codegree threshold for perfect $T_k$-tilings in $k$-graphs. This advancement, as an analogue of
the Corr\'{a}di--Hajnal theorem in $k$-graphs, provides strong support for resolving perfect $F$-tilings in the case when $F$ is a non-$k$-partite $k$-graph.

\begin{theorem}[Main result]\label{thm:main}
Let $k\ge 3$ be an integer. There exists $n_0 \in \mathbb{N}$ such that for every $n \geq n_0$ with $(2k-1) \mid n$, every $k$-graph $H$ on $n$ vertices with $\delta(H) \geq \frac{2n}{2k-1}$ admits a perfect $T_k$-tiling.
\end{theorem}

To demonstrate the optimality of the minimum codegree condition in Theorem \ref{thm:main}, we introduce the following extremal construction. 
Let \(n\in (2k-1)\N\). Let $A$ and $B$ be two disjoint vertex sets of size 
$\frac{2n}{2k-1} - 1$ and $\frac{(2k-3)n}{2k-1} + 1$ respectively.
Define a \(k\)-graph \(H_\mathrm{ext}\) on \(A \cup B\) where a \(k\)-set is an edge if and only if it intersects \(A\). 
This yields
\[
\delta_{k-1}(H_\mathrm{ext}) = |A| = \frac{2n}{2k-1} - 1.
\]
Critically, every edge of \(H_\mathrm{ext}\) contains at least one vertex from \(A\), 
while no vertex of \(T_k\) lies in all its edges. 
Consequently, every copy of \(T_k\) in \(H_\mathrm{ext}\) must contain at least two vertices from \(A\). 
It follows that any \(T_k\)-tiling $\mathcal{T}$ satisfies
\[
|\mathcal{T}| \leq \frac{|A|}{2} = \frac{n}{2k-1} - \frac{1}{2} < \frac{n}{2k-1},
\]
proving that \(H_\mathrm{ext}\) admits no perfect \(T_k\)-tiling.

\subsection{Rainbow tilings}

We also investigate minimum codegree conditions for rainbow $T_k$-tilings in families of $k$-graphs. Let $V$ be a set of $n$ vertices and $\mathcal{H} = \{H_1, \dots, H_{3n/(2k-1)}\}$ be a family of $k$-graphs with common vertex set $V$. A \emph{perfect rainbow $T_k$-tiling} in $\mathcal{H}$ is a perfect $T_k$-tiling in the multiset union $\bigcup_{i=1}^{3n/(2k-1)} H_i$ consisting of edges $e_1, \dots, e_{3n/(2k-1)}$ where $e_i \in E(H_i)$ for each $i$. 

In Section \ref{sec:rainbow}, we combine Theorem~\ref{thm:main} with a theorem given by Lang~\cite{Lang} to establish the following rainbow version.

\begin{theorem}\label{thm:rainbow}
Let $k\ge 3$ be an integer. For every $\varepsilon > 0$, there exists $n_0 \in \mathbb{N}$ such that the following holds for every $n \geq n_0$ with $(2k-1) \mid n$. Let $V$ be a set of $n$ vertices and $\mathcal{H} = \{H_1, \dots, H_{3n/(2k-1)}\}$ be a family of $k$-graphs with common vertex set $V$. If $\delta(H_i) \geq \left( \frac{2}{2k-1} + \varepsilon \right)n$ for each $i\in [\frac{3n}{2k-1}]$, then $\mathcal{H}$ contains a perfect rainbow $T_k$-tiling.
\end{theorem}

The minimum codegree bound in Theorem~\ref{thm:rainbow} is asymptotically tight. This is seen by setting $H_1 = \cdots = H_{3n/(2k-1)} = H_{\mathrm{ext}}$, where $H_{\mathrm{ext}}$ is the extremal construction after Theorem~\ref{thm:main}. 

\subsection{Proof strategy}
 By modifying and extending the methods in~\cite{BKMM}, our proof of Theorem~\ref{thm:main} distinguishes an extremal case (when $H$ approximates the configuration $H_\mathrm{ext}$) from a non-extremal case. 
The \emph{density} of a $k$-graph $H$ on $n$ vertices is defined as $d(H) := e(H) / \binom{n}{k}$, where $e(H)$ denotes the number of edges in $H$. For a set $S \subseteq V(H)$, the \emph{subgraph of $H$ induced by $S$}, denoted by $H[S]$, is the $k$-graph with vertex set $S$ consisting of all edges $e \in E(H)$ such that $e \subseteq S$.

\begin{definition}\rm
Let $H$ be a $k$-graph on $n$ vertices. For $\gamma > 0$, we say $H$ is \textit{$\gamma$-extremal} if there exists $S \subseteq V(H)$ with
$|S| = \lfloor \frac{2k-3}{2k-1} n \rfloor$
such that $H[S]$ has density at most $\gamma$.
\end{definition}

The following lemma establishes that Theorem~\ref{thm:main} holds in the extremal case. The proof, presented in Section~\ref{sec:extremal}, utilizes an approach specifically designed for the generalised triangle $T_k$; the pivotal step involves finding a perfect matching in an appropriately constructed auxiliary graph.
\begin{lemma}[Extremal case] \label{lem:extremal}
Let $k$ be an integer with $k\geq 3$, there exist $\gamma > 0$ and $n_0 \in \mathbb{N}$ for which the following holds. Let $H$ be a $k$-graph on $n \geq n_0$ vertices with $\delta(H) \geq \frac{2n}{2k-1}$ and $(2k-1) \mid n$. If $H$ is $\gamma$-extremal, then $H$ contains a perfect $T_k$-tiling.
\end{lemma}

It remains to prove Theorem~\ref{thm:main} in the non-extremal case.
Our proof employs the absorption method, an indispensable technique for spanning structures in discrete mathematics popularized by R\"{o}dl, Ruci\'{n}ski, and Szemer\'{e}di~\cite{rodl09}. 
It comprises two main ingredients. The first is the following absorbing lemma, guaranteeing the existence of a small $T_k$-tiling that can incorporate any sufficiently small vertex set (which is called an \emph{absorbing} set, see Definition~\ref{def:absorbing set simple}). Our first task is to prove the existence of an absorbing set in any $k$-graph with an arbitraily small linear codegree (see Lemma~\ref{lem:absorbing}), which takes up the bulk of our work. In constrast, the previous construction of absorbing set from Bowtell, Kathapurkar, Morrison and Mycroft~\cite{BKMM} in $3$-graphs heavily relies on a much stronger codegree condition $\de_2(H)\ge \frac{1}{3}n+o(n)$.

We will use the following notion of absorbing set as used in the graph case for example in~\cite{HMWY,NP}.
\begin{definition} \label{def:absorbing set simple}\rm
Let $H$ be an $n$-vertex $k$-graph and $F$ be an $s$-vertex $k$-graph.
A subset $A\subseteq V(H)$ is an $(F,\xi)$-\emph{absorbing set} for some constant $\xi>0$ if for any subset $U\subseteq V(H)\setminus A$ of size at most $\xi n$ such that $s$ divides $|A\cup U|$, $H[A\cup U]$ contains a perfect $F$-tiling.
\end{definition}

Here, we have the following key lemma which ensures the existence of an absorbing set.
\begin{lemma}[Absorbing lemma] \label{lem:absorbing} 
For any $\gamma,\alpha>0$, there exists $\xi>0$ such that the following holds for all sufficiently large $n$. If $H$ is an $n$-vertex $k$-graph with $\delta(H)\ge \alpha n$, then $H$ contains a $(T_k,\xi)$-absorbing set of size at most $\gamma n$. 
\end{lemma}

The second component for the non-extremal case is the following almost cover lemma: under a slightly weaker codegree condition than that in Theorem~\ref{thm:main}, any non-extremal $k$-graph $H$ admits a $T_k$-tiling covering almost all vertices. We shall prove the following lemma in Section~\ref{sec:tiling}.
\begin{lemma}[Almost cover lemma] \label{lem:tiling}
Suppose that $1/n \ll \alpha, \eta \ll \gamma$, and that $H$ is a $k$-graph on $n$ vertices. If $\delta(H) \geq \frac{2n}{2k-1} - \alpha n$,
then either $H$ admits a $T_k$-tiling covering at least $(1-\eta) n$ vertices of $H$, or $H$ is $\gamma$-extremal.
\end{lemma}

Combining Lemmas~\ref{lem:extremal},~\ref{lem:absorbing} and~\ref{lem:tiling}, we are now ready to prove Theorem~\ref{thm:main}.

\begin{proof}[\bf Proof of Theorem~\ref{thm:main}]
Choose constants $1/n\ll\gamma',\xi\ll \alpha\ll \gamma\ll 1/k$ such that $n$ is divisible by $2k-1$, and $H$ is a $k$-graph on $n$ vertices with $\delta(H) \geq 2n/(2k-1)$. By Lemma~\ref{lem:absorbing}, there exists a $(T_k,\xi)$-absorbing set $A \subseteq V(H)$. 
Let $V' := V(H) \setminus A$, $H' = H[V']$ and $n' = |V'|$, so $n-\gamma' n \leq n' \leq n$, and also $\delta(H') \geq 2n/(2k-1) - \gamma' n \geq 2n'/(2k-1) - \alpha n'$. By Lemma~\ref{lem:tiling} (applied with $\xi$ and $\gamma/2$ in place of $\eta$ and $\gamma$ respectively) it follows that either $H'$ admits a $T_k$-tiling covering at least $(1-\xi)n'$ vertices of $H'$, or $H'$ is $\gamma/2$-extremal.

Suppose first that $H'$ admits a $T_k$-tiling $\mathcal{T}_k$ which covers at least $(1-\xi)n'$ vertices of $H'$. Then $S := V' \setminus V(\mathcal{T}_k)$ has $|S| \leq \xi n' \leq \xi n$, and moreover $(2k-1) \mid |A\cup S|$ since both $n$ and $|V(\mathcal{T}_k)|$ are divisible by $2k-1$. By our choice of $A$ it follows that $H[A \cup S]$ admits a perfect $T_k$-tiling $\mathcal{T}'_k$, whereupon $\mathcal{T}_k \cup \mathcal{T}'_k$ is a perfect $T_k$-tiling in $H$, as required.

Now suppose instead that $H'$ is $\gamma/2$-extremal, meaning that there exists a set $S' \subseteq V(H')$ with $|S'| = (2k-3)n'/(2k-1)$ such that $d(H'[S']) \leq \gamma/2$.
By adding arbitrary $(2k-3)(n-n')/(2k-1) \leq \gamma' n$ vertices of $V(H) \setminus S'$ to $S'$, we obtain a set $S \subseteq V(H)$ with $|S| = (2k-3)n/(2k-1)$ such that $d(H[S]) \leq ((\gamma/2) \binom{n'}{k} + \gamma' n \binom{n}{k-1})/\binom{n}{k} \leq \gamma$.
This set $S$ witnesses that $H$ is $\gamma$-extremal, and so $H$ admits a perfect $T_k$-tiling by Lemma~\ref{lem:extremal}, completing the proof.
\end{proof}

\noindent \textbf{Notation}. 
For $n \in \mathbb{N}$, define $[n] := \{1, 2, \dots, n\}$. For simplicity, we often abbreviate the set $\{u_1,\ldots, u_l\}$ as $u_1\ldots u_l$. For a set $X$, write $\binom{X}{m}$ for all $m$-subsets of $X$. The expression $x \ll y$ means that for any $y > 0$, there exists $x_0 = x_0(y) > 0$ such that all relevant statements hold for $0 < x < x_0$. This notation extends naturally to arbitrary lengths (eg. $\alpha \ll \beta \ll \gamma$).

Let $H$ be a $k$-graph. Denote by $\nu(H)$ and $e(H)$ the number of vertices and edges in $H$, respectively. We call $H$ a \textit{$k$-partite} $k$-graph if its vertex set $V(H)$ admits a partition into $k$ parts such that every edge contains exactly one vertex from each part.
A vertex subset $S\subseteq V(H)$ of size $2k-1$ \textit{supports $T_k$} if $H[S]$ contains a copy of $T_k$.
A \textit{tight $2$-path} in $H$ consists of two edges sharing $k-1$ vertices. 

\section{Absorbing} \label{sec:absorbing}
In this section, we prove the absorbing lemma (Lemma~\ref{lem:absorbing}). For this, we follow the approach of Han, Morris, Wang and Yang~\cite{HMWY} and Nenadov and Pehova~\cite{NP} (the only difference is we consider hypergraphs here). 
At first, we define the following key notion of \emph{absorber} (following the notation in~\cite{HMWY,NP} in the graph setting) which we will use as ``building blocks'' for absorbing structures.

\begin{definition}\label{def:absorber}
\rm Let $H$ be a $k$-graph on $n$ vertices and $F$ be a $k$-graph on $s$ vertices.
  For any $S\in \binom{V(H)}{s}$ and an integer $t$, we say a subset $A_S\subset V(H)\setminus S$ is an $(F,t)$-\emph{absorber} for $S$ if $|A_S|\le s^2t$ and both $H[A_S]$ and $H[A_S\cup S]$ contain a perfect $F$-tiling.  
\end{definition} 

\subsection{Main tools}
We make use of a construction which guarantees a $(F,\xi)$-absorbing set provided that \emph{every $s$-set of vertices has linearly many vertex-disjoint $(F,t)$-absorbers}. The key idea that makes this possible stems back to Montgomery~\cite{Montgo} and has since found many applications in absorption arguments. Here, we shall follow the approach (and notation) of Nenadov and Pehova \cite{NP}. The following lemma is a hypergraph version of Lemma 2.2 in \cite{NP}. We prove it in Appendix \ref{lem3.2} for completeness.\footnote{To be  exact, the lemma of Nenadov and Pehova \cite[Lemma 2.2]{NP} differs slightly from ours here as they define absorbers to have at most $kt$ vertices whereas we define the upper bound to be $k^2t$. This minor adjustment has no bearing on the statement of this lemma or its proof.}

\begin{lemma}\label{bip-temp}
Let $s,t\in \mathbb{N}$, $F$ be a $k$-graph on $s$ vertices and  $\gamma >0$. Then there exists $\xi = \xi(s,t,\gamma)>0$ such that  the following holds. If $H$ is an $n$-vertex $k$-graph such that for
every $S\in \binom{V(H)}{s}$ there is a family of at least
$\gamma n$ vertex-disjoint $(F,t)$-absorbers, 
then $H$ contains a $(F,\xi)$-absorbing set of size at most $\gamma n$.
\end{lemma}

In order to apply Lemma \ref{bip-temp}, what remains is to prove the existence of linearly many vertex-disjoint absorbers. To achieve this, we adopt the lattice-based absorbing method developed in \cite{han16,han17,keevash15}, which we now discuss. 

We will use the following notation introduced by Keevash and Mycroft
\cite{keevash15}.
Let $H$ be an $n$-vertex $k$-graph and $\mathcal{P} = \{V_1,\ldots,V_r\}$ be a partition of $V(H)$. For any subset $S\subseteq V(H)$, the \emph{index vector} of $S$ with respect to $\mathcal{P}$, denoted by $\mathbf{i}_{\mathcal{P}}(S)$, is the vector in $\mathbb{Z}^r$ whose $i$th coordinate is the size of the intersections of $S$ with $V_i$ for each $i\in[r]$. For $j \in [r]$, let $\mathbf{u}_j\in \mathbb{Z}^r$ be the $j$th unit vector, i.e., $\mathbf{u}_j$ has $j$th coordinate  $1$ and
 all other coordinates $0$. 
 A \emph{transferral} is a vector of the form $\mathbf{u}_i-\mathbf{u}_j$ for some
distinct $i\neq j\in[r]$. A vector $\mathbf{i}\in \mathbb{Z}^r$ is an $s$-vector if all its coordinates are non-negative and their sum equals
$s$. Given $\beta > 0$ and an $s$-vertex graph $F$, an $s$-vector $\mathbf{v}$ is called $(F,\beta)$-\emph{robust with respect to} $\mathcal{P}$ if for any set $W$ of at most $\beta n$ vertices, there is a copy of $F$ in $V(H)\setminus W$ whose vertex set has index vector $\mathbf{v}$. Let $I^{\beta}(\mathcal{P})$ be the set of all $(F,\beta)$-robust $s$-vectors and $L^{\beta}
(\mathcal{P})$ be the lattice (i.e., the additive subgroup) generated by $I^{\beta}(\mathcal{P})$.\medskip

We also need the notion of $F$-reachability introduced by Han, Morris, Wang and Yang~\cite{HMWY}. Let $m,t$ be positive integers. Then we say that two vertices $u, v\in V(H)$ are $(F, m, t)$-\emph{reachable} (in $H$) if for any set $W\subseteq V(H)$ of at most $m$ vertices, there is a set $S\subseteq V(H)\setminus W$ of size at most $st-1$ such that both $H[S\cup \{u\}]$ and $H[S\cup \{v\}]$ have perfect $F$-tilings, where we call such $S$ an $F$-\emph{connector} for $u$ and $v$. 
Moreover, a set $U\subseteq V(H)$ is $(F,m,t)$-\emph{closed} if every two vertices in $U$ are $(F,m,t)$-reachable.   \medskip

The following result provides a sufficient condition on a given partition $\{V_1, V_2,\ldots,V_r\}$ to ensure that every $S\in \binom{V(G)}{s}$ of a certain type has linearly many vertex-disjoint absorbers. The following lemma is a hypergraph version of Lemma 3.10 in \cite{HMWY}. We prove it in Appendix \ref{lem3.3} for completeness.

\begin{lemma}\label{close}
Given $k,s, t\in \mathbb{N}$ with $k,s\ge 3, t\ge 1$ and  $\beta >0$, the following holds for any $s$-vertex $k$-graph $F$ and sufficiently large $n$. Let $H$ be an $n$-vertex $k$-graph with a partition $\mathcal{P}= \{V_1, V_2,\ldots,V_r\}$ for some integer $r\ge 1$ such that each $V_i$ is $(F,\beta n, t)$-closed with $|V_i|\ge \beta n$ for each $i\in [r]$. If $S\in \binom{V(H)}{s}$ such that $\emph{\bf{i}}_{\mathcal{P}}(S)$ is $(F,\beta)$-robust, then $S$ has at least $\frac{\beta}{s^3t}n$
%$\lfloor\frac{\be n-k}{k^2t}\rfloor$
vertex-disjoint $(F,t)$-absorbers.
\end{lemma}

To apply Lemma~\ref{close}, our main task is to show that
$V(H)$ is $(T_k,\beta n, t)$-closed (which is achieved until Lemma~\ref{closed}). In doing this, we first show that 
$V(H)$ admits a partition into constantly many closed parts. This partition serves as a basis for merging parts to obtain our target partition-$\{V(H)\}$. The following lemma is a hypergraph version of Lemma 4.1 in \cite{HMWY}. We prove it in Appendix \ref{lem3.4} for completeness.

\begin{lemma}[Partition lemma]\label{partition}
For any constant $\delta > 0$ and integers $s,k \ge 3$, there exist $\beta>0$ 
and an integer $t>0$
such that the following holds for sufficiently large $n$. Let $H$ be an $n$-vertex $k$-graph and $F$ be an $s$-vertex $k$-graph. If every vertex in $V(H)$ is $(F,\delta n, 1)$-reachable to at least $\delta n$ other vertices, then there is a partition $\mathcal{P}= \{V_1, V_2,\ldots,V_r\}$ of $V(H)$ for some integer $r\le \lceil{\tfrac{1}{\delta}\rceil}$ such that for each $i\in[r]$, $V_i$ is $(F,\beta n, t)$-closed and $|V_i|\ge  \tfrac{\delta }{2}n$.
\end{lemma}

We then proceed with a sequence of merging processes, which follows the strategy of transferrals in \cite{han17}. The following lemma builds a sufficient condition that allows us to merge two distinct parts into a closed one.
As the proof of Lemma \ref{transferral} is very similar to that of [\cite{HMWY}, Lemma 4.4], we defer it to Appendix \ref{fulu:transferral}.

\begin{lemma}[Transferral]\label{transferral}
Given any positive integers $s, t,k$  with $s,k\ge 3$ and  constant $\beta>0$, the following holds for sufficiently large $n$. Let $F$ be an $s$-vertex $k$-graph and $H$ be an $n$-vertex $k$-graph with a partition $\mathcal{P}=\{V_1, V_2,\ldots,V_r\}$ of $V(H)$ such that each $V_i$ is $(F,\beta n, t)$-closed. For distinct $i, j \in[r]$, if $\mathbf{u}_i-\mathbf{u}_j\in L^{\beta}(\mathcal{P})$, then there exists a constant $C:=C(F,r)$ such that $V_i\cup V_j$ is $\left(F,\tfrac{\beta}{2}n, Cst\right)$-closed.
\end{lemma}

To satisfy the required condition in Lemma~\ref{partition} when $F=T_k$, we prove the following result.
\begin{lemma}\label{patition1}
Suppose that $1/n\ll \delta \ll \alpha$ and $k\ge 3$ is an integer. If $H$ is a $k$-graph on $n$ vertices with $\delta(H)\ge \alpha n$, then every vertex in $V(H)$ is $(T_k,\delta n,1)$-reachable to at least $\delta n$ vertices of $V(H)$. 
\end{lemma}
\begin{proof}
Fix a vertex $z \in V(H)$ and a subset $W \subset V(H)$ with $|W| \leq \delta n$. For any $(k-2)$-set $U \subset V(H) \setminus (W \cup \{z\})$, we have $|N(U \cup \{z\})|\ge \alpha n$ as $\de(H)\ge \alpha n$. Since $\delta \ll \alpha$, we arbitrarily select $\lceil\frac{4}{\alpha}\rceil$ vertices $\{w_1, \dots, w_{\lceil\frac{4}{\alpha}\rceil}\} $ from $ N(U \cup \{z\}) \setminus W$.

By the minimum codegree condition and Inclusion-Exclusion Principle, there exist distinct $i, j \in \{1,2,\ldots,\lceil\frac{4}{\alpha}\rceil\}$ such that 
\[
|N(U \cup \{w_i\}) \cap N(U \cup \{w_j\})| \geq \frac{n}{2\lceil\frac{4}{\alpha}\rceil^2}\ge 2\delta n.
\]
We may then choose any vertex $v \in \left(N(U \cup \{w_i\}) \cap N(U \cup \{w_j\})\right) \setminus W$, noting there are at least $2\delta n - \delta n = \delta n$ choices. 

Next, select $k-3$ vertices $\{y_1, \dots, y_{k-3}\}$ from $V(H) \setminus (U \cup W\cup\{z, w_i, w_j\})$ and choose $u \in N(\{w_i, w_j, y_1,\dots, y_{k-3}\}) \setminus W$. Then both $(2k-1)$-sets 
\[
\{z\} \cup U \cup \{w_i, w_j\} \cup \{y_1, \dots, y_{k-3}\} \cup \{u\}
\]
and 
\[
\{v\} \cup U \cup \{w_i, w_j\} \cup \{y_1, \dots, y_{k-3}\} \cup \{u\}
\]
induce two copies of $T_k$. Hence every vertex $z$ is $(T_k, \delta n, 1)$-reachable to at least $\delta n$ vertices $v$ as above.
\end{proof}

\subsection{Merging}
Given Lemma~\ref{partition}, the application of Lemma~\ref{transferral} 
in the merging process reduces to verifying the existence of a transferral. The key difference from  Bowtell, Kathapurkar, Morrison and Mycroft's work~\cite{BKMM} is that $\delta(H) \geq \alpha n$ enables us to merge constantly many parts into a closed one. 

\begin{lemma}\label{closed} 
Let $k\ge 3$, $r,t$ be positive integers and $1/n\ll\beta\ll\delta,\alpha$. Let $H$ be a $k$-graph on $n$ vertices such that $\delta(H) \ge \alpha  n$ and $\mathcal{P}=\{V_1, V_2,\ldots,V_r\}$ be a partition of $V(H)$, where $V_i$ is $(T_k,\beta n, t)$-closed and $|V_i|\ge\delta n$ for each $i\in[r]$. Then $V(H)$ is~$(T_k, \tfrac{\beta}{2^{r-1}}n, (2Ck-C)^{r-1}t)$-closed, where $C=C(T_k,r)$ is a constant.
\end{lemma}

Combined with the blowup arguments as in Lemma \ref{lem:rooted-blow-ups}, our construction of transferrals with respect to $T_k$ is reformulated as depicting how the edges in $H$ are statistically distributed, and this is handled by Lemma~\ref{distinctpart} and~Lemma \ref{kmorethan4}.%加上分类的原则

\begin{lemma}\label{distinctpart}
Let $r\ge 2, k\ge 3,t$ be positive integers, $1/n\ll\beta\ll\gamma,\zeta\ll\delta,\alpha$. Let $H$ be a $k$-graph on $n$ vertices such that $\delta(H) \ge \alpha  n$ and $\mathcal{P}=\{V_1, V_2,\ldots,V_r\}$ be a partition of $V(H)$, where $V_i$ is $(T_k,\beta n, t)$-closed and $|V_i|\ge\delta n$ for each $i\in[r]$.  For some distinct $i,j\in [r]$, if there exists a $(k-1)$-vector $\mathbf{v}$ such that  for at least $\gamma n^{k-1}$ many $(k-1)$-sets $S$  with $\mathbf{i}_{\mathcal{P}}(S)=\mathbf{v}$, we have $|N(S)\cap V_i|\ge \zeta n$ and  $|N(S)\cap V_j|\ge \zeta n$, then  $V_i\cup V_j$ is $\left(T_k,\tfrac{\beta}{2}n, (2Ck-C)t\right)$-closed, where $C=C(T_k,r)$ is a constant.
\end{lemma}

\begin{lemma}\label{kmorethan4}
Let $r\ge 2, k\ge3, t$ be positive integers,  $1/n\ll\beta\ll\gamma,\zeta\ll\delta, \alpha$. Let $H$ be a $k$-graph on $n$ vertices such that $\delta(H) \ge \alpha  n$ and $\mathcal{P}=\{V_1, V_2,\ldots,V_r\}$ be a partition of $V(H)$, where $V_i$ is $(T_k,\beta n, t)$-closed and $|V_i|\ge\delta n$ for each $i\in[r]$. 
If for any $(k-1)$-vector $\mathbf{v}$, there is $i\in[r]$ such that for all but at most $\gamma n^{k-1}$ such $(k-1)$-sets $S$ with $\mathbf{i}_{\mathcal{P}}(S)=\mathbf{v}$, we have $|N(S)\cap V_i|\ge |N(S)|-\zeta n$,  
then there are $j,l\in [r]$ such that $V_j\cup V_l$ is $\left(T_k,\tfrac{\beta}{2}n, (2Ck-C)t\right)$-closed, where $C=C(T_k,r)$ is a constant.
\end{lemma}

Now we are ready to give a proof of Lemma~\ref{lem:absorbing}.
\begin{proof}[{\bf Proof of Lemma~\ref{lem:absorbing}}]
Choose constants $1/n\ll \beta'\ll\delta'\ll \alpha$ and let $H$ be an $n$-vertex  $k$-graph with $\de(H)\ge \alpha n$.
By Lemma~\ref{patition1} applied with $\delta=2\delta'$, 
every vertex in $V(H)$ is $(T_k,2\delta' n,1)$-reachable to at least $2\delta' n$ vertices of $V(H)$. Lemma~\ref{partition}, 
implies a partition $\mathcal{P}= \{V_1, V_2,\ldots,V_r\}$ of $V(H)$ for some integer $r\le \lceil{\tfrac{1}{2\delta'}\rceil}$ such that for each $i\in[r]$, $V_i$ is $(T_k,\beta' n, t)$-closed and $|V_i|\ge  \delta' n$. By Lemma~\ref{closed} applied with $\beta=\beta'$, $\delta=\delta'$, there exists a constant $C=C(T_k,r)$ such that $V(H)$ is~$(T_k, \tfrac{\beta'}{2^{r-1}}n, (2Ck-C)^{r-1}t)$-closed.
Now note that for any set $W$ of at most $\tfrac{\beta'}{2^{r-1}}n$ vertices, there is a copy of $T_k$ in $V(H)\setminus W$. Indeed, this follows by the minimum codegree condition of $H$. Therefore, by Lemma \ref{close} applied with $r=1$, $\beta=\tfrac{\beta'}{2^{r-1}}$, $s=2k-1$ and $F=T_k$, every $S\in \binom{V(H)}{2k-1}$ has at least $\gamma n$ vertex-disjoint $(T_k,t)$-absorbers for $\gamma=\tfrac{\beta}{(2k-1)^3t}$. Furthermore, applying Lemma \ref{bip-temp}, we have that $H$ contains a $(T_k,\xi)$-absorbing set of size at most $\gamma n$ as desired.        
\end{proof}

\subsubsection{Proofs of Lemma~\ref{closed}-\ref{kmorethan4}}
We denote by $F[t]$ the \emph{$t$-blowup of $F$} obtained from $F$ by replacing every vertex of $F$ by a vertex set of size $t$ and every hyperedge of $F$ by a copy of a complete $k$-partite $k$-graph.
Lemma~\ref{lem:rooted-blow-ups} is obtained from \cite[Lemma 4.3]{blowup}.
\begin{lemma}[\cite{blowup}, Lemma 4.3]\label{lem:rooted-blow-ups}
Let $1/n \ll 1/s,1/k$ and $\mu>0$. Let~$H$ be an $n$-vertex $k$-graph, $\mathcal{F}$ be a finite family of $s$-vertex $k$-graphs.
Suppose that there are at least $\mu n^{s}$ subsets $S \subseteq V(H)$ of size $s$ such that $H[S] \in \mathcal{F}$.
Then there is a family $\mathcal{U} = \{U_i\}_{i \in [s]}$ of pairwise disjoint $2$-sets such that the $s$-partite $k$-graph induced by $\mathcal{U}$ contains a $2$-blowup of some $F\in \mathcal{F}$.
\end{lemma}

Given a $k$-vector $\mathbf{x} = (x_1, \dots, x_r)$, let $\mathcal{P} = \{V_1, V_2, \ldots, V_r\}$ be a partition of the vertex set of a $k$-graph $H$ such that $E(H) \subseteq \{ S \subseteq V(H) : \mathbf{i}_{\mathcal{P}}(S) = \mathbf{x} \}$.
We define the \emph{$\mathbf{x}$-density} of $H$ as $d^{\mathbf{x}}(H) := e(H) / \prod_{i=1}^r \binom{|V_i|}{x_i}$.
We say that $H$ is \emph{$(\epsilon, \mathbf{x})$-complete} if $d^{\mathbf{x}}(H) \ge 1 - \epsilon$. When $\mathbf{x}$ is clear from context, we simply say that $H$ is \emph{$\epsilon$-complete}.
Let $c : E(H) \rightarrow [r]$ be an $r$-edge coloring of $H$. We say that $H$ is \emph{$\zeta$-monochromatic with respect to $c$} if there exists some $i \in [r]$ such that $|c^{-1}(i)| / e(H) \ge 1 - \zeta$.
We will use the following stability result, whose proof is deferred to the end of this section.

Let $H$ be an edge-colored $k$-graph. 
A tight 2-path in $H$ is \textit{rainbow} if its two edges have different colors.

\begin{lemma}\label{sametpart}
Let $k,r$ be  positive integers and $1/n \ll \epsilon_k,\, \xi_k \ll \zeta_k,\, \delta$.  
Let $\mathbf{x} = (x_1, \dots, x_r)$ be a $k$-vector, and $\mathcal{P} = \{V_1, V_2, \dots, V_r\}$ be a partition of the vertex set of an $n$-vertex $k$-graph $H$ such that $E(H) \subseteq \{ S \subseteq V(H) : \mathbf{i}_{\mathcal{P}}(S) = \mathbf{x} \}$, where each part satisfies $|V_j| \ge \delta n$ for all $j \in [r]$.  
Suppose $H$ is an $(\epsilon_k, \mathbf{x})$-complete $r$-edge-colored $k$-graph on $n$ vertices, which contains fewer than $\xi_k n^{k+1}$ rainbow tight $2$-paths.  
Then $H$ is $\zeta_k$-monochromatic.
\end{lemma}

\begin{proof}[{\bf Proof of Lemma~\ref{closed}}]
Choose constants $\gamma,\xi_{k-1},\epsilon_{k-1},\zeta,\zeta_{k-1}$ satisfying $\beta \ll \xi_{k-1}\ll \gamma,\zeta\ll \epsilon_{k-1}\ll \zeta_{k-1} \ll \delta,\alpha$ and let $H$ be a $k$-graph on $n$ vertices such that $\delta(H) \ge \alpha  n$ and $\mathcal{P}=\{V_1, V_2,\ldots,V_r\}$ be a partition of $V(H)$, where $V_i$ is $(T_k,\beta n, t)$-closed for each $i\in[r]$. We may assume that $r\ge 2$, or else we are done. 
Lemma~\ref{distinctpart} implies that whenever there exists a $(k-1)$-vector $\mathbf{v}$ for which at least $\gamma n^{k-1}$ many $(k-1)$-sets $S$ with $\mathbf{i}_{\mathcal{P}}(S) = \mathbf{v}$ satisfy $|N(S) \cap V_i| \geq \zeta n$ and $|N(S) \cap V_j| \geq \zeta n$ for some pair $i \neq j$, then there exists a constant $C_r=C_r(T_k,r)$ such that  $V_i \cup V_j$ becomes $\left(T_k, \tfrac{\beta}{2}n, (2C_rk-C_r)t\right)$-closed.

Assume that there exists no such a $(k-1)$-vector with the required property in Lemma~\ref{distinctpart}. 
Let $\mathcal{S}^{i}_{\mathbf{v}}$ denote the collection of all $(k-1)$-sets $S$ with index vector $\mathbf{v}=(x_1,\ldots,x_r)$ satisfying $|N(S)\cap V_i| \geq |N(S)| - (r-1)\zeta n$, and define $\mathcal{S}_{\mathbf{v}} = \bigcup_{i=1}^r \mathcal{S}^i_{\mathbf{v}}$. We construct a $(k-1)$-graph $H^{k-1}$ on  $V(H)$ with $E(H^{k-1}) = \mathcal{S}_{\mathbf{v}}$. 

The hypergraph $H^{k-1}$ satisfies 
\[
e(H^{k-1}) \geq \prod_{i=1}^r \binom{|V_i|}{x_i} - \gamma n^{k-1} > (1-\epsilon_{k-1})\prod_{j=1}^r \binom{|V_j|}{x_j},
\]
which establishes that $H^{k-1}$ is $(\epsilon_{k-1},\mathbf{v})$-complete. 

We further define an edge coloring $\phi_{\mathbf{v}}: \mathcal{S}_{\mathbf{v}} \to [r]$ on $H^{k-1}$  where each $(k-1)$-set $S$ is assigned color $\phi_{\mathbf{v}}(S) = i$ if and only if  $S\in \mathcal{S}^i_{\mathbf{v}}$.
Let $R$ the number of rainbow tight $2$-paths in the resulting edge-colored $H^{k-1}$.
We split the argument into two separate cases.

{\bf Case 1.} $R \geq \xi_{k-1} n^{k}$. There exist two colors, without loss of generality, assume colors $1$ and $2$, such that $H^{k-1}$ contains at least $\xi_{k-1} n^{k}/\binom{r}{2}$  rainbow tight 2-paths using these two colors. 

Recall $\textbf{v}=(x_1,\ldots,x_r)$ and 
assume $x_q\neq0$ for some $q\in[r]$.  Let $\mathcal{T}$ be a maximal family of pairwise vertex-disjoint pairs $\{T, T'\}$ of copies of $T_k$ (i.e., $V(T\cup T')\cap V(\widetilde{T}\cup \widetilde{T}')=\emptyset$ for any $(T, T'),(\widetilde{T}, \widetilde{T}')\in \mathcal{T}$)  such that  $\mathbf{i}_{\mathcal{P}}(T)-\mathbf{i}_{\mathcal{P}}(T')=\mathbf{u}_1-\mathbf{u}_2$ and $\mathbf{i}_{\mathcal{P}}(T)\in\{\sum_{l=1}^{r}x_l\mathbf{u}_l+(k-2)\mathbf{u}_1+\mathbf{u}_p+\mathbf{u}_q: p\in [r]\}$.

We claim that $|\mathcal{T}|\ge 2r\beta n$.
Otherwise, $|V(\mathcal{T})|< (8k-4)r\beta n$.   
 Observing that there are at most $|V(\mathcal{T})| \cdot n^{k-1} < (8k-4)r\beta n^{k}$ rainbow tight 2-paths intersecting $V(\mathcal{T})$, it follows that $H^{k-1} \setminus V(\mathcal{T})$ contains at least $\xi_{k-1} n^{k}/\binom{r}{2}  - (8k-4)r\beta  n^{k} $ rainbow tight $2$-paths whose edges are colored 1 and 2. 
There exists $i\in [r]$ such that at least $(\xi_{k-1} n^{k}/\binom{r}{2}  - (8k-4)r\beta  n^{k})/r$ rainbow tight 2-paths  $P=e_1\cup e_2$ in $H^{k-1}$  with the symmetric difference of $e_1$ and $e_2$ in $ V_i$ and $\phi_{\mathbf{v}}(e_1)=1$, $\phi_{\mathbf{v}}(e_2)=2$.
Let $\mathcal{F}$ be the family of a $2$-path $P'=(e_1\cup\{v_1\})\cup (e_2\cup\{v_2\})$ in $H$ with the symmetric difference of $e_1$ and $e_2$ in $ V_i$ with $v_1\in N(e_1)\cap (V_1\setminus V(\mathcal{T}))$ and $v_2\in N(e_2)\cap (V_2\setminus V(\mathcal{T}))$.  Note that $|N(e_i)\cap (V_i\setminus V(\mathcal{T}))|\ge\alpha n- (r-1)\zeta n-(8k-4)r\beta n$ for $i=\{1,2\}$.
Then $|\mathcal{F}|\ge (\xi_{k-1}  n^k/\binom{r}{2}  - (8k-4)r\beta n^k)(\alpha n-(r-1)\zeta n-(8k-4)r\beta n)^2/r$.
 Applying Lemma~\ref{lem:rooted-blow-ups}, we obtain a 2-blowup $F^*$ of some $F \in \mathcal{F}$ with vertex set $V(F^*) = \{a,a',b,b',u,u',v,v',u_1,\ldots,u_{k-2},u'_1,\ldots,u'_{k-2}\}$, where 
 $a,a'\in V_1\setminus V(\mathcal{T})$, $b,b'\in V_2\setminus V(\mathcal{T})$,
 $u,u',v,v' \in V_i\setminus V(\mathcal{T})$, $u_1,u'_1\in V_q\setminus V(\mathcal{T})$, $\{u,u_1,\dots,u_{k-2},a\} \in E(H)$, $\{v,u_1,\dots,u_{k-2},b\}\in E(H)$.

Take $w_1,\ldots,w_{k-3}\in V_1\setminus (V(\mathcal{T})\cup V(F^*))$ and $z\in N(w_1\ldots w_{k-3}u_1u'_1)\setminus (V(\mathcal{T}\cup V(F^*))$.  Now, we get two copies of $T_k$
$$T=\{au_1\ldots u_{k-2}u,au'_1u_2\ldots u_{k-2}u,u_1u'_1w_1\ldots w_{k-3}z\},$$ 
$$T'=\{b u_1 \ldots u_{k-2}v, bu'_1u_2\ldots u_{k-2}v,u_1u'_1w_1\ldots w_{k-3}z\}.$$
Then, $V(T\cup T')\cap V(\mathcal{T})=\emptyset$ and $\mathbf{i}_{\mathcal{P}}(T)-\mathbf{i}_{\mathcal{P}}(T')=\mathbf{u}_1-\mathbf{u}_2$, which contradicts the maximality of $\mathcal{T}$. Hence, $|\mathcal{T}|\ge 2r\beta n$.
By the pigeonhole principle, there is $p\in [r]$ such that $$\mathbf{s}:=\mathbf{i}_{\mathcal{P}}(T_k)=\sum_{l=1}^{r}x_l\mathbf{u}_l+(k-2)\mathbf{u}_1+\mathbf{u}_p+\mathbf{u}_q,$$  
$$\mathbf{t}:=\mathbf{i}_{\mathcal{P}}(T'_k)=\sum_{l=1}^{r}x_l\mathbf{u}_l+(k-3)\mathbf{u}_1+\mathbf{u}_2+\mathbf{u}_p+\mathbf{u}_q$$ for at least $2\beta n$ pairs of $\mathcal{T}$.
Therefore, $\mathbf{s},\mathbf{t}\in I^{\beta}(\mathcal{P})$, and then Lemma \ref{transferral} implies $V_1\cup V_2$ is $\left(T_k,\tfrac{\beta}{2}n, (2C_rk-C_r)t\right)$-closed.

{\bf Case 2.} $R < \xi_{k-1} n^{k}$. By Lemma~\ref{sametpart}, $H^{k-1}$ becomes $\zeta_{k-1}$-monochromatic with respect to $\phi_{\mathbf{v}}$. 
Hence, we have 
\begin{equation*}
\begin{aligned}
|\phi_{\mathbf{v}}^{-1}(i)| &\ge (1-\zeta_{k-1})e(H^{k-1})\\
&\ge (1-\zeta_{k-1})(1-\epsilon_{k-1})\prod_{j\in [r]}\binom{|V_j|}{x_j}.\\
\end{aligned}
\end{equation*}
Therefore, 
for all but at most $(\zeta_{k-1}+\epsilon_{k-1})\prod_{j\in [r]}\binom{|V_j|}{x_j}\le(\zeta_{k-1}+\epsilon_{k-1}) n^{k-1}$ many $(k-1)$-sets $S$ with index vector $\mathbf{v}$ satisfy $|N(S) \cap V_i| > |N(S)| - (r-1)\zeta n$. 
Since $1/n\ll \beta\ll\zeta\ll\epsilon_{k-1}\ll \zeta_{k-1}\ll \delta, \alpha$, applying Lemma~\ref{kmorethan4} with $\gamma=\zeta_{k-1}+\epsilon_{k-1}$, we obtain there exist $j,l \in [r]$ such that $V_j \cup V_l$ is $\left(T_k, \tfrac{\beta}{2}n, (2C_rk-C_r)t\right)$-closed.

In all cases, we can construct a refined partition $V(H)=(V_1,\ldots,V_{r-1})$  such that $\delta(H) \ge \alpha  n$, where $V_i$ is $\left(T_k,\tfrac{\beta}{2}n, (2C_rk-C_r)t\right)$-closed and $|V_i|\ge\delta n$ for each $i\in[r-1]$. 
Iterating this argument at most $r-1$ times with the current partition yields that $V(H)$ is $\left(T_k, \tfrac{\beta}{2^{r-1}}n, (2Ck-C)^{r-1}t\right)$-closed, where $C=\max\{C_2,\cdots,C_r\}$.
\end{proof}

\begin{proof}[{\bf Proof of Lemma~\ref{distinctpart}}]
By the assumption of Lemma~\ref{distinctpart}, we have $\beta\ll \gamma,\zeta\ll \delta,\alpha$.
Consider a $(k-1)$-vector $\mathbf{v}=(x_1,\ldots,x_r)$ as in the assumption. There are at least $\gamma n^{k-1}$ $(k-1)$-sets $S$ with $\mathbf{i}_{\mathcal{P}}(S)=\mathbf{v}$ for which $|N(S)\cap V_i|\ge \zeta n$ and $|N(S)\cap V_j|\ge \zeta n$ hold simultaneously for some $i\neq j\in [r]$. 
Fix such a $(k-1)$-set $S=\{v_1,\ldots,v_{k-1}\}$. Assume $v_{k-1}\in V_q$ for some $q\in [r]$. Let $\mathcal{T}$ be a maximal family of pairwise vertex-disjoint pairs $\{T, T'\}$ of copies of $T_k$  such that $\mathbf{i}_{\mathcal{P}}(T)-\mathbf{i}_{\mathcal{P}}(T')=\mathbf{u}_i-\mathbf{u}_j$ with $\mathbf{i}_{\mathcal{P}}(T)\in\{\sum_{l=1}^{r}x_l\mathbf{u}_l+(k-2)\mathbf{u}_i+\mathbf{u}_p+\mathbf{u}_q: p\in [r]\}$.

We claim that $|\mathcal{T}| \geq 2r\beta n$. Suppose for contradiction that $|\mathcal{T}| < 2r\beta n$, which implies $|V(\mathcal{T})| < (8k-4)r\beta n$.  
There are at most $|V(\mathcal{T})| \cdot n^{k} < (8k-4)r\beta n^{k+1}$ tight 2-paths intersecting $V(\mathcal{T})$. It follows from $\beta \ll\gamma,\zeta$ that $V(H)\setminus V(\mathcal{T})$ contains at least
\[
\gamma n^{k-1} \cdot \zeta n \cdot \zeta n - (8k-4)r\beta n^{k+1} \geq \frac{\gamma}{2}\zeta^2 n^{k+1}
\]
tight 2-paths. 
Let $\mathcal{F}$ be the family of a tight 2-path in $H\setminus V(\mathcal{T})$. 
Applying Lemma~\ref{lem:rooted-blow-ups} with $\mu=\frac{\gamma}{2}\zeta^2$ and $s=k+1$, we obtain a 2-blowup $F^*$ of some $F \in \mathcal{F}$ with vertex set $V(F^*) = \{a,a',b,b',u_1,\ldots,u_{k-1},u'_1,\ldots,u'_{k-1}\}$, where $a,a' \in V_i\setminus V(\mathcal{T})$, $b,b' \in V_j\setminus V(\mathcal{T})$, and both $\{u_1,\ldots,u_{k-1}\}$ and $\{u'_1,\ldots,u'_{k-1}\}$ have index vector $\mathbf{v}$. 

Take $w_1,\ldots,w_{k-3}\in V_i\setminus (V(\mathcal{T})\cup V(F^*))$ and $z\in N(w_1\ldots w_{k-3}u_{k-1}u'_{k-1})\setminus (V(\mathcal{T})\cup V(F^*))$.  
Now, we get  two copies of $T_k$
$$T=\{au_1\ldots u_{k-1},au_1\ldots u_{k-2}u'_{k-1},u_{k-1}u'_{k-1}w_1\ldots w_{k-3}z\},$$ 
$$T'=\{b u'_1 \ldots u'_{k-2}u_{k-1},  bu'_1\ldots u'_{k-1},u_{k-1}u'_{k-1}w_1\ldots w_{k-3}z\}.$$ 
Then, $V(T\cup T')\cap V(\mathcal{T})=\emptyset$ and $\mathbf{i}_{\mathcal{P}}(T)-\mathbf{i}_{\mathcal{P}}(T')=\mathbf{u}_i-\mathbf{u}_j$, which contradicts the maximality of $\mathcal{T}$.

By the pigeonhole principle, there is $p\in [r]$ such that there are at least $2\beta n$ pairs $(T,T')$ from $\mathcal{T}$ satisfying
$$\mathbf{s}:=\mathbf{i}_{\mathcal{P}}(T)=\sum_{l=1}^{r}x_l\mathbf{u}_l+(k-2)\mathbf{u}_i+\mathbf{u}_p+\mathbf{u}_q,$$  
$$\mathbf{t}:=\mathbf{i}_{\mathcal{P}}(T')=\sum_{l=1}^{r}x_l\mathbf{u}_l+(k-3)\mathbf{u}_i+\mathbf{u}_j+\mathbf{u}_p+\mathbf{u}_q.$$ Therefore, $\mathbf{s},\mathbf{t}\in I^{\beta}(\mathcal{P})$ and $\mathbf{s}-\mathbf{t}=\mathbf{u}_i-\mathbf{u}_j$. Consequently, Lemma~\ref{transferral} implies that there exists a constant $C=C(T_k,r)$ such that  $V_i \cup V_j$ is $\left(T_k, \tfrac{\beta}{2}n, (2Ck-C)t\right)$-closed. 
\end{proof}

\begin{proof}[\bf Proof of Lemma~\ref{kmorethan4}] 
By the assumption in Lemma~\ref{kmorethan4}, we choose $1/n\ll \beta\ll \gamma,\zeta\ll \delta,\alpha$.
Consider a $(k-1)$-vector $\mathbf{v}=(x_1,\cdots,x_r)$. Let $\mathcal{X}$ be the family of all $(k-1)$-sets $S$ with index vector $\mathbf{v}$ such that $|N(S)\cap V_i|\ge |N(S)|-\zeta n$ for some $i\in[r]$. We further assume $x_j\ge 1$ for some $j\in [r]$.

\begin{claim}\label{fact}
$|N(S')\cap V_j|\ge |N(S')|-\zeta n$ holds for all but at most $\gamma n^{k-1}$ many $(k-1)$-sets $S'$ with index vector
$\mathbf{v}'=\mathbf{v}+\mathbf{u}_i-\mathbf{u}_j$. 
\end{claim}
\begin{proof}
Indeed, there are at least 
\begin{equation*}
\begin{aligned}
|\mathcal{X}|\cdot (|N(S)|-\zeta n)&\ge
(\prod_{l\in [r]}\binom{|V_l|}{x_l}-\gamma n^{k-1})(|N(S)|-\zeta n)\\
&\ge (\prod_{l\in [r]}\binom{|V_l|}{x_l}-\gamma n^{k-1})(\alpha n-\zeta n)\\
&\ge \frac{1}{4k^k} \prod_{l\in [r]}|V_l|^{x_l}\alpha n
\end{aligned}
\end{equation*}
edges with index vector $\mathbf{v}+\mathbf{u}_i$ of $H$. 
Let $\mathcal{Y}$ be the family of all $(k-1)$-sets $S'$ with index vector $\mathbf{v}'$ and $|N(S')\cap V_j|\ge 2\zeta n$.
Then there are at most
\[
|\mathcal{Y}||V_j|+\binom{|V_i|}{x_i+1}\binom{|V_j|}{x_j-1}\prod_{l\in [r]\setminus\{i,j\}}\binom{|V_l|}{x_l}\cdot 2\zeta n.
\]
edges with index vector $\mathbf{v}+\mathbf{u}_i$.
Hence,
\[
|\mathcal{Y}||V_j|+\binom{|V_i|}{x_i+1}\binom{|V_j|}{x_j-1}\prod_{l\in [r]\setminus\{i,j\}}\binom{|V_l|}{x_l}\cdot 2\zeta n 
\ge \frac{1}{4k^k} \prod_{l\in [r]}|V_l|^{x_l}\alpha n.
\]
Thus,
\begin{equation*}
\begin{aligned}
|\mathcal{Y}|
&\ge \frac{1}{|V_j|}\left(\frac{1}{4k^k} \prod_{l\in [r]}|V_l|^{x_l}\alpha n-\binom{|V_i|}{x_i+1}\binom{|V_j|}{x_j-1}\prod_{l\in [r]\setminus\{i,j\}}\binom{|V_l|}{x_l}\cdot 2\zeta n\right)\\
&\ge \frac{1}{4k^k} |V_j|^{x_j-1}\prod_{l\in [r]\setminus\{j\}}|V_l|^{x_l}\alpha n-\binom{|V_i|}{x_i+1}\binom{|V_j|}{x_j-1}\prod_{l\in [r]\setminus\{i,j\}}\binom{|V_l|}{x_l}\cdot 2\zeta \frac{n}{|V_j|}\\
&\ge \frac{\alpha}{10k^k} \binom{|V_i|}{x_i+1}\binom{|V_j|}{x_j-1}\prod_{l\in[r]\setminus \{i,j\}}\binom{|V_l|}{x_l},
\end{aligned}
\end{equation*}
since $\frac{n}{|V_j|}\le \frac{1}\delta$ and $\zeta\ll\delta,\alpha, 1/k$.
By assumption in Lemma \ref{kmorethan4}, there are at least 
\[
\binom{|V_i|}{x_i+1}\binom{|V_j|}{x_j-1}\prod_{l\in [r]\setminus\{i,j\}}\binom{|V_l|}{x_l}-\gamma n^{k-1}
\]
$(k-1)$-sets $S'$ with index vector $\textbf{v}'$ and $|N(S')\cap V_l|\ge |N(S')|-\zeta n$ for some $l\in [r]$.
As $\gamma\ll \frac{\alpha}{10k^k}$, 
we have $l=j$. Hence, $|N(S')\cap V_j|\ge |N(S')|-\zeta n$ holds for all but at most $\gamma n^{k-1}$ $(k-1)$-sets $S'$ with index vector
$\mathbf{v}'$. 
\end{proof}
\noindent Next, we proceed with three cases based on the value of $k$.\medskip

{\bf Case 1.} $k\ge 5$. Choose a $(k-1)$-vector $\mathbf{v}=(x_1,x_2,x_3,\ldots,x_r)$ such that $x_1,x_2\ge 2$. Assume that for all but at most $\gamma n^{k-1}$ such $(k-1)$-sets $S$ with $\mathbf{i}_{\mathcal{P}}(S)=\mathbf{v}$, we have $|N(S)\cap V_i|\ge |N(S)|-\zeta n$.
Let $\mathbf{a}=\mathbf{v}+\mathbf{u}_i-\mathbf{u}_1$ and $\mathbf{b}=\mathbf{v}+\mathbf{u}_i-\mathbf{u}_2$.
Claim~\ref{fact} guarantees that for all but at most $\gamma n^{k-1}$ many $(k-1)$-sets $A$ with index vector $\mathbf{a}$, at least $|N(A)| - \zeta n$ neighbors lie in $V_1$; we denote this family of $(k-1)$-sets by $\mathcal{A}$. Similarly, for all but at most $\gamma n^{k-1}$ many $(k-1)$-sets $B$ with index vector $\mathbf{b}$, at least $|N(B)| - \zeta n$ neighbors lie in $V_2$, and we denote this family by $\mathcal{B}$.
Let $\mathcal{T}$ be a maximal family of pairwise vertex-disjoint pairs $\{T, T'\}$ of copies of $T_k$  such that $\mathbf{i}_{\mathcal{P}}(T)-\mathbf{i}_{\mathcal{P}}(T')=\mathbf{u}_1-\mathbf{u}_2$ with
$\mathbf{i}_{\mathcal{P}}(T)\in\{\mathbf{v}+(k-4)\mathbf{u}_1+2\mathbf{u}_2+\mathbf{u}_i+\mathbf{u}_p: p\in [r]\}$.

We claim that $|\mathcal{T}|\ge 2r\beta n$.
Suppose to the contrary that $|\mathcal{T}| < 2r\beta n$, which implies $|V(\mathcal{T})| < (8k-2)r\beta n$.

There are at most $|V(\mathcal{T})|n^{k-2}<(8k-2)r\beta n^{k-1}$ many $(k-1)$-sets intersecting $V(\mathcal{T})$.
 Consequently, there are at least
\begin{equation*}
\begin{aligned}
&\prod_{\substack{l\in[r]\setminus \{q,i\}}}\binom{|V_l|}{x_l}\binom{|V_q|}{x_q-1}\binom{|V_i|}{x_i+1} - (8k-2)r\beta n^{k-1} \\
&\ge \prod_{\substack{l\in[r]\setminus \{q,i\}}}\binom{\delta n}{x_l}\binom{\delta n}{x_q-1}\binom{\delta n}{x_i+1} - (8k-2)r\beta n^{k-1}  \\
&\ge \gamma n^{k-1}
\end{aligned}
\end{equation*}
$(k-1)$-sets from $\mathcal{A}\setminus V(\mathcal{T})$ when $q=1$ (respectively $\mathcal{B}\setminus V(\mathcal{T})$ when $q=2$), where the inequality follows from the choice $\beta \ll \gamma \ll \delta$.

Choose two $(k-1)$-sets $A\in \mathcal{A}\setminus V(\mathcal{T})$, $B\in \mathcal{B}\setminus V(\mathcal{T})$. Then by definition we can choose $a,b\in N(A)\cap (V_1\setminus V(\mathcal{T}))$, $c,d\in N(B)\cap (V_2\setminus V(\mathcal{T}))$, $v_5,\ldots,v_{k-1}\in V_1\setminus V(\mathcal{T})$.  Let $z\in N(abcdv_5\ldots v_{k-1})\setminus V(\mathcal{T})$.
Now, we get two copies of $T_k$
$$T=\{\{a\}\cup A,\{b\}\cup A,abcdv_5\ldots v_{k-1}z\},$$ 
$$T'=\{\{c\}\cup B,\{d\}\cup B, abcdv_5\ldots v_{k-1}z\}.$$ Then, 
$V(T\cup T')\cap V(\mathcal{T})=\emptyset$ and $\mathbf{i}_{\mathcal{P}}(T)-\mathbf{i}_{\mathcal{P}}(T')=\mathbf{u}_1-\mathbf{u}_2$, which contradicts the maximality of $\mathcal{T}$. Thus $|\mathcal{T}| \geq 2r\beta n$.

By the pigeonhole principle, there exists $p \in [r]$ for which the index vectors
\[
\mathbf{s} := \mathbf{v} + (k-4)\mathbf{u}_1 + 2\mathbf{u}_2 + \mathbf{u}_i + \mathbf{u}_p,
\]
\[
\mathbf{t} := \mathbf{v} + (k-3)\mathbf{u}_1 + \mathbf{u}_2 + \mathbf{u}_i + \mathbf{u}_p
\]
both belong to $I^{\beta}(\mathcal{P})$. Lemma~\ref{transferral} then yields that $V_1 \cup V_2$ is $\left(T_k, \tfrac{\beta}{2}n, (2Ck-C)t\right)$-closed.\medskip

{\bf Case 2.} $k=4$. Let $\mathbf{c}=\mathbf{u}_1+2\mathbf{u}_2$ and $\mathbf{d}=2\mathbf{u}_1+\mathbf{u}_2$. 
By the assumption in Lemma~\ref{kmorethan4},
there exist $i,j\in [r]$ such that for all but at most $\gamma n^3$ triples $C$ with index vector $\mathbf{c}$, we have $|N(C) \cap V_i| \geq |N(C)| - \zeta n$ (denoting this subcollection by $\mathcal{C}$), and similarly for all but at most $\gamma n^3$ triples $D$ with index vector $\mathbf{d}$, we have $|N(D) \cap V_j| \geq |N(D)| - \zeta n$ (denoted by $\mathcal{D}$).

By Claim \ref{fact}, we can obtain that 
if $i=1$, then $j=2$, vice versa. 
Let $\mathbf{s}=\mathbf{c}+\mathbf{u}_i+\mathbf{u}_j+\mathbf{u}_1+\mathbf{u}_2$ and $\mathbf{t}=\mathbf{d}+\mathbf{u}_i+\mathbf{u}_j+\mathbf{u}_1+\mathbf{u}_2$. 
Let $\mathcal{T}$ be a maximal family of pairwise vertex-disjoint pairs $\{T, T'\}$ of copies of $T_4$  such that $\mathbf{i}_{\mathcal{P}}(T)=\mathbf{s}$ and $\mathbf{i}_{\mathcal{P}}(T')=\mathbf{t}$.

We claim that $|\mathcal{T}|\ge 2\beta n$.
If $|\mathcal{T}|< 2\beta n$, then $|V(\mathcal{T})|< 28\beta n$.  
By Claim~\ref{fact}, we may take $v_1\in V_i\setminus V(\mathcal{T}), v_2,v_3\in V_2\setminus V(\mathcal{T})$ such that $|N(v_1v_2v_3)\cap (V_1\setminus V(\mathcal{T}))|>(\alpha-28\beta-\zeta)n$. 
We can pick $u_1,u_2\in V_1\cap (N(v_1v_2v_3)\setminus V(\mathcal{T}))$ and $w\in V_2\setminus V(\mathcal{T})$ such that $\{u_1,u_2,w\}\in \mathcal{D}\setminus V(\mathcal{T})$, which follows from 
\[
\binom{(\alpha-28\beta-\zeta)n}{2}\binom{\delta n-28\beta n}{1}>\gamma n^3.
\]
Take $z\in N(u_1u_2w)\cap (V_j\setminus V(\mathcal{T}))$.
Let 
$T=\{u_1v_1v_2v_3,u_2v_1v_2v_3,u_1u_2wz\}$.
Then $\mathbf{i}_{\mathcal{P}}(T)=\mathbf{s}$.
The illustration of 
$T$ in Figure \ref{fig:T_4(1)} uses monochromatic lines connecting four vertices to denote a hyperedge.

\begin{figure}[H]
    \centering

\tikzset{every picture/.style={line width=0.75pt}}       

\begin{tikzpicture}[x=0.75pt,y=0.75pt,yscale=-1,xscale=1]

\draw   (100,40) -- (170,40) -- (170,160) -- (100,160) -- cycle ;

\draw   (191,40) -- (261,40) -- (261,160) -- (191,160) -- cycle ;

\draw   (294,40) -- (364,40) -- (364,160) -- (294,160) -- cycle ;

\draw   (387,40) -- (457,40) -- (457,160) -- (387,160) -- cycle ;

\draw  [color={rgb, 255:red, 0; green, 0; blue, 0 }  ,draw opacity=1 ][fill={rgb, 255:red, 0; green, 0; blue, 0 }  ,fill opacity=1 ] (129.59,65.3) .. controls (129.55,62.93) and (131.43,60.96) .. (133.81,60.92) .. controls (136.19,60.87) and (138.15,62.76) .. (138.2,65.13) .. controls (138.24,67.51) and (136.36,69.47) .. (133.98,69.52) .. controls (131.6,69.57) and (129.64,67.68) .. (129.59,65.3) -- cycle ;

\draw  [color={rgb, 255:red, 0; green, 0; blue, 0 }  ,draw opacity=1 ][fill={rgb, 255:red, 0; green, 0; blue, 0 }  ,fill opacity=1 ] (230.09,85.26) .. controls (230.05,82.88) and (231.94,80.92) .. (234.31,80.87) .. controls (236.69,80.82) and (238.65,82.71) .. (238.7,85.09) .. controls (238.75,87.46) and (236.86,89.43) .. (234.48,89.47) .. controls (232.11,89.52) and (230.14,87.63) .. (230.09,85.26) -- cycle ;

\draw  [color={rgb, 255:red, 0; green, 0; blue, 0 }  ,draw opacity=1 ][fill={rgb, 255:red, 0; green, 0; blue, 0 }  ,fill opacity=1 ] (130.61,114.28) .. controls (130.57,111.91) and (132.45,109.94) .. (134.83,109.9) .. controls (137.2,109.85) and (139.17,111.74) .. (139.22,114.11) .. controls (139.26,116.49) and (137.38,118.45) .. (135,118.5) .. controls (132.62,118.55) and (130.66,116.66) .. (130.61,114.28) -- cycle ;
 
\draw  [color={rgb, 255:red, 0; green, 0; blue, 0 }  ,draw opacity=1 ][fill={rgb, 255:red, 0; green, 0; blue, 0 }  ,fill opacity=1 ] (217.59,64.3) .. controls (217.55,61.93) and (219.43,59.96) .. (221.81,59.92) .. controls (224.19,59.87) and (226.15,61.76) .. (226.2,64.13) .. controls (226.24,66.51) and (224.36,68.47) .. (221.98,68.52) .. controls (219.6,68.57) and (217.64,66.68) .. (217.59,64.3) -- cycle ;

\draw  [color={rgb, 255:red, 0; green, 0; blue, 0 }  ,draw opacity=1 ][fill={rgb, 255:red, 0; green, 0; blue, 0 }  ,fill opacity=1 ] (215.61,123.28) .. controls (215.57,120.91) and (217.45,118.94) .. (219.83,118.9) .. controls (222.2,118.85) and (224.17,120.74) .. (224.22,123.11) .. controls (224.26,125.49) and (222.38,127.45) .. (220,127.5) .. controls (217.62,127.55) and (215.66,125.66) .. (215.61,123.28) -- cycle ;

\draw  [color={rgb, 255:red, 0; green, 0; blue, 0 }  ,draw opacity=1 ][fill={rgb, 255:red, 0; green, 0; blue, 0 }  ,fill opacity=1 ] (319.59,79.3) .. controls (319.55,76.93) and (321.43,74.96) .. (323.81,74.92) .. controls (326.19,74.87) and (328.15,76.76) .. (328.2,79.13) .. controls (328.24,81.51) and (326.36,83.47) .. (323.98,83.52) .. controls (321.6,83.57) and (319.64,81.68) .. (319.59,79.3) -- cycle ;

\draw  [color={rgb, 255:red, 0; green, 0; blue, 0 }  ,draw opacity=1 ][fill={rgb, 255:red, 0; green, 0; blue, 0 }  ,fill opacity=1 ] (412.59,79.3) .. controls (412.55,76.93) and (414.43,74.96) .. (416.81,74.92) .. controls (419.19,74.87) and (421.15,76.76) .. (421.2,79.13) .. controls (421.24,81.51) and (419.36,83.47) .. (416.98,83.52) .. controls (414.6,83.57) and (412.64,81.68) .. (412.59,79.3) -- cycle ;

\draw [color={rgb, 255:red, 208; green, 2; blue, 27 }  ,draw opacity=1 ][line width=1.5]    (133.91,65.2) -- (221.91,64.2) ;

\draw [color={rgb, 255:red, 208; green, 2; blue, 27 }  ,draw opacity=1 ][line width=1.5]    (221.91,64.2) -- (234.4,85.17) ;

\draw [color={rgb, 255:red, 208; green, 2; blue, 27 }  ,draw opacity=1 ][line width=1.5]    (234.4,85.17) -- (323.91,79.2) ;

\draw [color={rgb, 255:red, 126; green, 211; blue, 33 }  ,draw opacity=1 ][line width=1.5]    (134.91,114.2) -- (221.89,64.22) ;

\draw [color={rgb, 255:red, 126; green, 211; blue, 33 }  ,draw opacity=1 ][line width=1.5]    (221.89,64.22) .. controls (239,59.92) and (242.5,83.92) .. (234.4,85.17) ;

\draw [color={rgb, 255:red, 126; green, 211; blue, 33 }  ,draw opacity=1 ][line width=1.5]    (234.4,85.17) .. controls (275,104.92) and (308.33,92) .. (323.89,79.22) ;
 
\draw [color={rgb, 255:red, 74; green, 144; blue, 226 }  ,draw opacity=1 ][line width=1.5]    (133.91,65.2) -- (134.91,114.2) ;

\draw [color={rgb, 255:red, 74; green, 144; blue, 226 }  ,draw opacity=1 ][line width=1.5]    (134.91,114.2) -- (219.91,123.2) ;

\draw [color={rgb, 255:red, 74; green, 144; blue, 226 }  ,draw opacity=1 ][line width=1.5]    (219.91,123.2) -- (416.91,79.2) ;

\draw (100,171.4) node [anchor=north west][inner sep=0.75pt]    {$V{\textstyle _{1}}\setminus V(\mathcal{T})$};

\draw (189,171.4) node [anchor=north west][inner sep=0.75pt]    {$V{\textstyle _{2}}\setminus V(\mathcal{T})$};

\draw (296,171.4) node [anchor=north west][inner sep=0.75pt]    {$V{\textstyle _{i}}\setminus V(\mathcal{T})$};

\draw (388,171.4) node [anchor=north west][inner sep=0.75pt]    {$V{\textstyle _{j}}\setminus V(\mathcal{T})$};

\draw (110,55.4) node [anchor=north west][inner sep=0.75pt]    {${\textstyle u_{1}}$};

\draw (109,106.4) node [anchor=north west][inner sep=0.75pt]    {${\textstyle u_{2}}$};

\draw (230,51.4) node [anchor=north west][inner sep=0.75pt]    {${\textstyle v_{2}}$};

\draw (210.5,76.9) node [anchor=north west][inner sep=0.75pt]    {${\textstyle v_{3}}$};

\draw (200,122.4) node [anchor=north west][inner sep=0.75pt]    {$w$};

\draw (330,67.4) node [anchor=north west][inner sep=0.75pt]    {${\textstyle v_{1}}$};

\draw (425.2,68.53) node [anchor=north west][inner sep=0.75pt]    {$z$};

\end{tikzpicture}

    \caption{$T=\{u_1v_1v_2v_3,u_2v_1v_2v_3,u_1u_2wz\}$}
    \label{fig:T_4(1)}
\end{figure}

By Claim~\ref{fact}, we may take $v'_1\in V_j\setminus V(\mathcal{T}), v'_2,v'_3\in V_1\setminus V(\mathcal{T})$ such that $|N(v'_1v'_2v'_3)\cap (V_2\setminus V(\mathcal{T}))|>(\alpha-28\beta-\zeta)n$. 
We can pick  $u'_1,u'_2\in V_2\cap (N(v'_1v'_2v'_3)\setminus V(\mathcal{T}))$ and $w'\in V_1\setminus V(\mathcal{T})$, 
such that $\{u'_1,u'_2,w'\}\in \mathcal{C}\setminus V(\mathcal{T})$, which follows from
\[
\binom{(\alpha-28\beta-\zeta)n}{2}\binom{\delta n-28\beta n}{1}>\gamma n^3.
\]
Take $z'\in N(u'_1u'_2w')\cap (V_i\setminus V(\mathcal{T}))$.
Let 
$T'=\{u'_1v'_1v'_2v'_3,u'_2v'_1v'_2v'_3,u'_1u'_2w'z'\}$, see Figure \ref{fig:T_4(2)}.
Then $\mathbf{i}_{\mathcal{P}}(T')=\mathbf{t}$.
This contradicts the maximality of $\mathcal{T}$.

\begin{figure}[H]
    \centering

\tikzset{every picture/.style={line width=0.75pt}}        

\begin{tikzpicture}[x=0.75pt,y=0.75pt,yscale=-1,xscale=1]

\draw  [color={rgb, 255:red, 0; green, 0; blue, 0 }  ,draw opacity=1 ][fill={rgb, 255:red, 0; green, 0; blue, 0 }  ,fill opacity=1 ] (142.61,76.28) .. controls (142.57,73.91) and (144.45,71.94) .. (146.83,71.9) .. controls (149.2,71.85) and (151.17,73.74) .. (151.22,76.11) .. controls (151.26,78.49) and (149.38,80.45) .. (147,80.5) .. controls (144.62,80.55) and (142.66,78.66) .. (142.61,76.28) -- cycle ;

\draw   (100,40) -- (170,40) -- (170,160) -- (100,160) -- cycle ;

\draw   (191,40) -- (261,40) -- (261,160) -- (191,160) -- cycle ;

\draw   (294,40) -- (364,40) -- (364,160) -- (294,160) -- cycle ;

\draw   (387,40) -- (457,40) -- (457,160) -- (387,160) -- cycle ;
 
\draw  [color={rgb, 255:red, 0; green, 0; blue, 0 }  ,draw opacity=1 ][fill={rgb, 255:red, 0; green, 0; blue, 0 }  ,fill opacity=1 ] (117.58,75.61) .. controls (117.54,73.62) and (115.89,72.03) .. (113.9,72.07) .. controls (111.91,72.11) and (110.33,73.76) .. (110.37,75.75) .. controls (110.41,77.74) and (112.05,79.32) .. (114.05,79.29) .. controls (116.04,79.25) and (117.62,77.6) .. (117.58,75.61) -- cycle ;

\draw  [color={rgb, 255:red, 0; green, 0; blue, 0 }  ,draw opacity=1 ][fill={rgb, 255:red, 0; green, 0; blue, 0 }  ,fill opacity=1 ] (217.4,86.67) .. controls (217.35,84.3) and (219.24,82.33) .. (221.61,82.28) .. controls (223.99,82.24) and (225.95,84.12) .. (226,86.5) .. controls (226.05,88.88) and (224.16,90.84) .. (221.78,90.89) .. controls (219.41,90.93) and (217.44,89.05) .. (217.4,86.67) -- cycle ;
 
\draw  [color={rgb, 255:red, 0; green, 0; blue, 0 }  ,draw opacity=1 ][fill={rgb, 255:red, 0; green, 0; blue, 0 }  ,fill opacity=1 ] (217.59,64.3) .. controls (217.55,61.93) and (219.43,59.96) .. (221.81,59.92) .. controls (224.19,59.87) and (226.15,61.76) .. (226.2,64.13) .. controls (226.24,66.51) and (224.36,68.47) .. (221.98,68.52) .. controls (219.6,68.57) and (217.64,66.68) .. (217.59,64.3) -- cycle ;

\draw  [color={rgb, 255:red, 0; green, 0; blue, 0 }  ,draw opacity=1 ][fill={rgb, 255:red, 0; green, 0; blue, 0 }  ,fill opacity=1 ] (130.61,143.28) .. controls (130.57,140.91) and (132.45,138.94) .. (134.83,138.9) .. controls (137.2,138.85) and (139.17,140.74) .. (139.22,143.11) .. controls (139.26,145.49) and (137.38,147.45) .. (135,147.5) .. controls (132.62,147.55) and (130.66,145.66) .. (130.61,143.28) -- cycle ;
 
\draw  [color={rgb, 255:red, 0; green, 0; blue, 0 }  ,draw opacity=1 ][fill={rgb, 255:red, 0; green, 0; blue, 0 }  ,fill opacity=1 ] (316.4,63.67) .. controls (316.35,61.3) and (318.24,59.33) .. (320.61,59.28) .. controls (322.99,59.24) and (324.95,61.12) .. (325,63.5) .. controls (325.05,65.88) and (323.16,67.84) .. (320.78,67.89) .. controls (318.41,67.93) and (316.44,66.05) .. (316.4,63.67) -- cycle ;
 
\draw  [color={rgb, 255:red, 0; green, 0; blue, 0 }  ,draw opacity=1 ][fill={rgb, 255:red, 0; green, 0; blue, 0 }  ,fill opacity=1 ] (417.61,95.28) .. controls (417.57,92.91) and (419.45,90.94) .. (421.83,90.9) .. controls (424.2,90.85) and (426.17,92.74) .. (426.22,95.11) .. controls (426.26,97.49) and (424.38,99.45) .. (422,99.5) .. controls (419.62,99.55) and (417.66,97.66) .. (417.61,95.28) -- cycle ;

\draw [color={rgb, 255:red, 74; green, 144; blue, 226 }  ,draw opacity=1 ][line width=1.5]    (221.98,62.52) -- (221.78,84.89) ;

\draw [color={rgb, 255:red, 74; green, 144; blue, 226 }  ,draw opacity=1 ][line width=1.5]    (221.89,64.22) -- (320.7,63.59) ;

\draw [color={rgb, 255:red, 208; green, 2; blue, 27 }  ,draw opacity=1 ][line width=1.5]    (146.91,76.2) -- (221.89,64.22) ;

\draw [color={rgb, 255:red, 126; green, 211; blue, 33 }  ,draw opacity=1 ][line width=1.5]    (146.91,76.2) .. controls (164.33,87) and (207.33,92) .. (221.78,86.89) ;

\draw [color={rgb, 255:red, 126; green, 211; blue, 33 }  ,draw opacity=1 ][line width=1.5]    (221.7,86.59) .. controls (260.33,102) and (406.05,108.07) .. (421.61,95.28) ;

\draw [color={rgb, 255:red, 208; green, 2; blue, 27 }  ,draw opacity=1 ][line width=1.5]    (146.91,76.2) -- (113.97,75.68) ;
 
\draw [color={rgb, 255:red, 74; green, 144; blue, 226 }  ,draw opacity=1 ][line width=1.5]    (134.91,143.2) -- (221.7,86.59) ;
 
\draw [color={rgb, 255:red, 208; green, 2; blue, 27 }  ,draw opacity=1 ][line width=1.5]    (221.89,64.22) -- (422,95.5) ;

\draw [color={rgb, 255:red, 126; green, 211; blue, 33 }  ,draw opacity=1 ][line width=1.5]    (113.97,75.68) .. controls (125.33,83) and (132.33,84) .. (146.91,76.2) ;

\draw (101,168.4) node [anchor=north west][inner sep=0.75pt]    {$V{\textstyle _{1} \setminus V(\mathcal{T})}$};

\draw (193,168.4) node [anchor=north west][inner sep=0.75pt]    {$V{\textstyle _{2}\setminus V(\mathcal{T})}$};

\draw (296,168.4) node [anchor=north west][inner sep=0.75pt]    {$V{\textstyle _{i}\setminus V(\mathcal{T})}$};

\draw (389,168.4) node [anchor=north west][inner sep=0.75pt]    {$V{\textstyle _{j}\setminus V(\mathcal{T})}$};

\draw (105,53.4) node [anchor=north west][inner sep=0.75pt]    {${\textstyle v'_{2}}$};

\draw (141.33,55.4) node [anchor=north west][inner sep=0.75pt]    {${\textstyle v'_{3}}$};

\draw (202,43.4) node [anchor=north west][inner sep=0.75pt]    {${\textstyle u'_{1}}$};

\draw (210,91.4) node [anchor=north west][inner sep=0.75pt]    {${\textstyle u'_{2}}$};

\draw (113.91,131.6) node [anchor=north west][inner sep=0.75pt]    {$w'$};

\draw (426.81,95.32) node [anchor=north west][inner sep=0.75pt]    {${\textstyle v'_{1}}$};

\draw (329,54.4) node [anchor=north west][inner sep=0.75pt]    {$z'$};

\end{tikzpicture}

    \caption{$T'=\{u'_1v'_1v'_2v'_3,u'_2v'_1v'_2v'_3,u'_1u'_2w'z'\}$}
    \label{fig:T_4(2)}
\end{figure}

Hence, $\mathbf{s},\mathbf{t}\in I^{\beta}(\mathcal{P})$. By Lemma \ref{transferral}, $V_1\cup V_2$ is $\left(T_4,\tfrac{\beta}{2}n, 7Ct\right)$-closed.\medskip

{\bf Case 3.} $k=3$. We
distinguish two subcases.

{\bf Subcase 3.1.} There exists a $2$-vector $\textbf{u}_i+\textbf{u}_j$ with $i\neq j$ such that for all but at most $\gamma n^2$ pairs $S$ with index vector $\textbf{u}=\textbf{u}_i+\textbf{u}_j$, we have $|N(S)\cap V_l|\ge |N(S)|-\zeta n$  with $l\in[r]\setminus \{i,j\}$.
Without loss of generality (relabelling if necessary), we assume $i=1$ and $j=2$.
By Claim \ref{fact} and the assumption in Lemma \ref{kmorethan4}, 
we have $|N(S_1)\cap V_{i_1}|\ge |N(S_1)|-\zeta n$ 
holds for all but at most $\gamma n^2$ many pairs $S_1$ with 
$\mathbf{i}_{\mathcal{P}}(S_1)=2\mathbf{u}_1$ and
$i_1\in [r]\setminus \{2,l\}$. 
Let $\mathcal{T}$ be a maximal family of pairwise vertex-disjoint triples $\{T,T',T''\}$ of copies of $T_3$ such that $2\mathbf{i}_{\mathcal{P}}(T)-\mathbf{i}_{\mathcal{P}}(T')-\mathbf{i}_{\mathcal{P}}(T'')=\textbf{u}_2-\textbf{u}_{i_1}$.

We claim that $|\mathcal{T}|\ge 2\beta n$. If $|\mathcal{T}|< 2\beta n$, then $|V(\mathcal{T})|<30\beta n$.
Choose $v_1\in V_1\setminus V(\mathcal{T}),v_2 \in V_2\setminus V(\mathcal{T})$ with $|N(v_1v_2)\cap (V_l\setminus V(\mathcal{T}))|>(\alpha-30\beta-\zeta)n$. By Claim \ref{fact} and the assumption in Lemma \ref{kmorethan4}, we have $u_1,u_2\in N(v_1v_2)\cap (V_l\setminus V(\mathcal{T}))$ such that $|N(u_1u_2)\cap (V_{i_2}\setminus V(\mathcal{T}))|\ge |N(u_1u_2)|-\zeta n-|V(\mathcal{T})|$ with $i_2\in[r]\setminus\{1,2,i_1\}$, which follows from $(\alpha-30\beta-\zeta)^2n^2>\gamma n^2$. Take $w\in N(u_1u_2)\cap (V_{i_2}\setminus V(\mathcal{T}))$.
Let $T=\{v_1v_2u_1,v_1v_2u_2,u_1u_2w\}$. Then $\textbf{i}_{\mathcal{P}}(T)=\textbf{u}_1+\textbf{u}_2+2\textbf{u}_l+\textbf{u}_{i_2}$.

By Claim \ref{fact}, choose $v'_1\in V_2\setminus V(\mathcal{T}), v'_2\in V_l\setminus V(\mathcal{T})$ with $|N(v'_1v'_2)\cap (V_1\setminus V(\mathcal{T}))|>(\alpha-30\beta-\zeta)n$.
We may pick $u'_1,u'_2\in N(v'_1v'_2)\cap (V_1\setminus V(\mathcal{T}))$ such that $|N(u'_1u'_2)\cap (V_{i_1}\setminus V(\mathcal{T}))|\ge |N(u'_1u'_2)|-\zeta n-|V(\mathcal{T})|$, which follows from $(\alpha-30\beta-\zeta)^2n^2>\gamma n^2$.
Take $w'\in N(u'_1u'_2)\cap (V_{i_1}\setminus V(\mathcal{T}))$.
Let $T'=\{v'_1v'_2u'_1,v'_1v'_2u'_2,u'_1u'_2w'\}$. Then $\textbf{i}_{\mathcal{P}}(T')=2\textbf{u}_1+\textbf{u}_2+\textbf{u}_l+\textbf{u}_{i_1}$.

Similarly, by Claim \ref{fact}, we can take $v''_1\in V_l\setminus V(\mathcal{T}),v''_2\in V_{i_2}\setminus V(\mathcal{T})$ with $|N(v''_1v''_2)\cap (V_l\setminus V(\mathcal{T}))|>(\alpha-30\beta-\zeta)n$.
Pick $u''_1,u''_2\in N(v''_1v''_2)\cap (V_l\setminus V(\mathcal{T}))$ such that $|N(u''_1u''_2)\cap (V_{i_2}\setminus V(\mathcal{T}))|\ge |N(u''_1u''_2)|-\zeta n-|V(\mathcal{T})|$, which follows from $(\alpha-30\beta-\zeta)^2n^2>\gamma n^2$.
Take $w''\in N(u''_1u''_2)\cap (V_{i_2}\setminus V(\mathcal{T}))$.
Let $T''=\{v''_1v''_2u''_1,v''_1v''_2u''_2,u''_1u''_2w''\}$. Then $\textbf{i}_{\mathcal{P}}(T'')=3\textbf{u}_l+2\textbf{u}_{i_2}$.

Hence, $2\mathbf{i}_{\mathcal{P}}(T)-\mathbf{i}_{\mathcal{P}}(T')-\mathbf{i}_{\mathcal{P}}(T'')=\textbf{u}_2-\textbf{u}_{i_1}$, contradicts the maximality of $\mathcal{T}$. Therefore, $\textbf{u}_2-\textbf{u}_{i_1}\in L^\beta (\mathcal{P})$. By Lemma \ref{transferral}, $V_2\cup V_{i_1}$ is $\left(T_3,\tfrac{\beta}{2}n, 5Ct\right)$-closed.

{\bf Subcase 3.2.} For all $2$-vector $\textbf{u}_i+\textbf{u}_j$ with $i\neq j$, we have for all but at most $\gamma n^2$ pairs $S$ with index vector $\textbf{u}_i+\textbf{u}_j$ and $i\neq j$, we have $|N(S)\cap V_i|\ge |N(S)|-\zeta n$ or $|N(S)\cap V_j|\ge |N(S)|-\zeta n$. 
Without loss of generality, we assume that for all but at most $\gamma n^2$ pairs $S$ with index vector $\textbf{u}_1+\textbf{u}_2$, we have $|N(S)\cap V_1|\ge |N(S)|-\zeta n$. 
Let $\mathcal{T}$ be a maximal family of pairwise vertex-disjoint triples $\{T,T'\}$ of copies of $T_3$ such that $\mathbf{i}_{\mathcal{P}}(T)-\mathbf{i}_{\mathcal{P}}(T')=\textbf{u}_1-\textbf{u}_2$.

We claim that $|\mathcal{T}|\ge 2\beta n$. Suppose $|\mathcal{T}|< 2\beta n$. Then $|V(\mathcal{T})|<20\beta n$.

If $r=2$, choose $v_1\in V_1\setminus V(\mathcal{T}),v_2 \in V_2\setminus V(\mathcal{T})$ with $|N(v_1v_2)\cap (V_1\setminus V(\mathcal{T}))|>(\alpha-20\beta-\zeta)n$. By Claim \ref{fact}, we can pick $u_1,u_2\in N(v_1v_2)\cap (V_1\setminus V(\mathcal{T}))$ such that $|N(u_1u_2)\cap (V_2\setminus V(\mathcal{T}))|\ge |N(u_1u_2)|-\zeta n-|V(\mathcal{T})|$, which follows from $(\alpha-20\beta-\zeta)^2n^2>\gamma n^2$.
Take $w\in N(u_1u_2)\cap (V_2\setminus V(\mathcal{T}))$.
Let $T=\{v_1v_2u_1,v_1v_2u_2,u_1u_2w\}$. Then $\textbf{i}_{\mathcal{P}}(T)=3\textbf{u}_1+2\textbf{u}_2$.
Choose $v'_1,v'_2\in V_1\setminus V(\mathcal{T})$ with $|N(v'_1v'_2)\cap (V_2\setminus V(\mathcal{T}))|>(\alpha-20\beta-\zeta)n$. By the assumption in Lemma \ref{kmorethan4}, we can take $u'_1,u'_2\in N(v'_1v'_2)\cap (V_2\setminus V(\mathcal{T}))$ such that $|N(u'_1u'_2)\cap (V_2\setminus V(\mathcal{T}))|\ge |N(u'_1u'_2)|-\zeta n-|V(\mathcal{T})|$, which follows from $(\alpha-20\beta-\zeta)^2n^2>\gamma n^2$.
Take $w'\in N(u'_1u'_2)\cap (V_2\setminus V(\mathcal{T}))$.
Let $T'=\{v'_1v'_2u'_1,v'_1v'_2u'_2,u'_1u'_2w'\}$. Then $\textbf{i}_{\mathcal{P}}(T')=2\textbf{u}_1+3\textbf{u}_2$.

If $r\ge 3$, then for all but at most $\gamma n^2$ pairs $S$ with index vector $\textbf{u}_1+\textbf{u}_3$, we have $|N(S)\cap V_3|\ge |N(S)|-\zeta n$, which follows from the assumption in Lemma \ref{kmorethan4}. 
Choose $v_1,v_2\in V_3\setminus V(\mathcal{T})$ with $|N(v_1v_2)\cap (V_1\setminus V(\mathcal{T}))|>(\alpha-20\beta-\zeta)n$. By Claim \ref{fact}, we can pick $u_1,u_2\in N(v_1v_2)\cap (V_1\setminus V(\mathcal{T}))$ such that $|N(u_1u_2)\cap (V_2\setminus V(\mathcal{T}))|\ge |N(u_1u_2)|-\zeta n-|V(\mathcal{T})|$, which follows from $(\alpha-20\beta-\zeta)^2n^2>\gamma n^2$.
Take $w\in N(u_1u_2)\cap (V_2\setminus V(\mathcal{T}))$.
Let $T=\{v_1v_2u_1,v_1v_2u_2,u_1u_2w\}$. Then $\textbf{i}_{\mathcal{P}}(T)=2\textbf{u}_1+\textbf{u}_2+2\textbf{u}_3$.
Similarly, we have for all but at most $\gamma n^2$ pairs $S$ with index vector $\textbf{u}_2+\textbf{u}_3$, we have $|N(S)\cap V_2|\ge |N(S)|-\zeta n$. 
Choose $v'_1,v'_2\in V_2\setminus V(\mathcal{T})$ with $|N(v'_1v'_2)\cap (V_3\setminus V(\mathcal{T}))|>(\alpha-20\beta-\zeta)n$. By the assumption in Lemma \ref{kmorethan4}, we can take $u'_1,u'_2\in N(v'_1v'_2)\cap (V_3\setminus V(\mathcal{T}))$ such that $|N(u'_1u'_2)\cap (V_1\setminus V(\mathcal{T}))|\ge |N(u'_1u'_2)|-\zeta n-|V(\mathcal{T})|$, which follows from $(\alpha-20\beta-\zeta)^2n^2>\gamma n^2$.
Take $w'\in N(u'_1u'_2)\cap (V_1\setminus V(\mathcal{T}))$.
Let $T'=\{v'_1v'_2u'_1,v'_1v'_2u'_2,u'_1u'_2w'\}$. Then $\textbf{i}_{\mathcal{P}}(T')=\textbf{u}_1+2\textbf{u}_2+2\textbf{u}_3$.

Wenever $r=2$ or $r\ge 3$, $\mathbf{i}_{\mathcal{P}}(T)-\mathbf{i}_{\mathcal{P}}(T')=\textbf{u}_1-\textbf{u}_2$, contradicts the maximality of $\mathcal{T}$. Therefore, $\textbf{u}_1-\textbf{u}_2\in L^\beta (\mathcal{P})$. By Lemma \ref{transferral}, $V_1\cup V_2$ is $\left(T_3,\tfrac{\beta}{2}n, 5Ct\right)$-closed.
\end{proof}

\subsubsection{Proof of Lemma~\ref{sametpart}}
This section presents an elementary proof of Lemma~\ref{sametpart}.
\begin{proof}[{\bf Proof}] 

Suppose $H$ is an $(\epsilon_k, \mathbf{x})$-complete $r$-edge-colored $k$-graph on $n$ vertices, which contains fewer than $\xi_k n^{k+1}$ rainbow tight 2-paths. We proceed by induction on $k$. Let $H^k:=H$. Let \(\phi\) be the $r$-edge-coloring of $H^k$, and $R_k$ denote the number of rainbow tight 2-paths in $H^k$. 
Additionally choose constants $c_1=c_1(k),~c_2=c_2(k),~c_3=c_3(k)$ satisfying
$$1/n\ll \epsilon_k,\xi_k \ll c_1, c_2 \ll c_3 \ll \zeta_k, \delta.$$

We first prove the statement holds for $k = 2$. 
Since $H^2$ is $(\epsilon_2, \mathbf{x})$-complete, we have $e(H^2) \ge (1 - \epsilon_2) \cdot \prod_{j = 1}^r \binom{|V_j|}{x_j}$. Recall that $\mathbf{x}=(x_1,\cdots,x_r)$ is a $k$-vector. 
Let $I = \{i : x_i \ne 0\}$. Then $1 \le |I| \le 2$.  
We write \( I = \{l_1, l_2\} \) if \( l_1 \ne l_2 \), and \( I = \{l_1\} \) otherwise, so that \( I \) is not a multiset. 
Let $G := H^2\big[\bigcup_{i \in I} V_i\big].$
Note that $E(H^2) = E(G)$, so $d_{G}(v) = d_{H^2}(v)$ for every vertex $v \in V(G)$. For convenience, we write $d(v) := d_{G}(v) = d_{H^2}(v)$.
It follows that $|V(G)| = \sum_{i \in I} |V_i|$ and 
$e(G) \ge (1 - \epsilon_2) \cdot \prod_{j \in I} \binom{|V_j|}{x_j} \ge (1 - \epsilon_2) \binom{\delta n}{2}$.  
Hence, there are at least $(1 - \sqrt{\epsilon_2}) |V_{l_1}|$ vertices $v \in V_{l_1}$ in $G$ with $d(v) \ge (1 - \sqrt{\epsilon_2}) |V_{l_2}|$, and at least $(1 - \sqrt{\epsilon_2}) |V_{l_2}|$ vertices $v \in V_{l_2}$ with $d(v) \ge (1 - \sqrt{\epsilon_2}) |V_{l_1}| $.
Let $W$ be the set of vertices $v \in V(G)$ satisfying $d(v) \ge (1 - \sqrt{\epsilon_2}) |V_{l_2}|$ or $d(v) \ge (1 - \sqrt{\epsilon_2}) |V_{l_1}|$.  
Then $|W| \ge (1 - \sqrt{\epsilon_2}) |V(G)|$.
For every vertex $v \in V(G)$ and every $i \in [r]$, let $E_i(v) = \{ e : v \in e,\ \phi(e) = i \}$ and $d_i(v) = |E_i(v)|$.  

We claim that for all but at most $c_1 n$ vertices $v \in V(G)$, there exists some $i^v \in [r]$ such that $d_{i^v}(v) \ge d(v) - (r - 1) c_2 n$.
Otherwise, there would be more than $c_1 n$ vertices $v \in V(G)$ for which, for every $i \in [r]$, we have $\sum_{j \in [r] \setminus \{i\}} d_j(v) \ge (r - 1) c_2 n$.  
Let $X$ be the set of such vertices. Then $|X| \ge c_1 n$.  
Therefore, for each vertex $v \in X$, there must exist two distinct colors $i_v$ and $j_v$ such that $d_{i_v}(v) \ge c_2 n$ and $d_{j_v}(v) \ge c_2 n$.
Since $\xi_2 \ll c_1,\, c_2$, we deduce that
\begin{equation*}
\begin{aligned}
R_2 &\ge \sum_{v \in X} d_{i_v}(v) \cdot d_{j_v}(v) 
\ge c_1 n \cdot c_2 n \cdot c_2 n 
= c_1 c_2^2 n^3 
> \xi_2 n^3.
\end{aligned}
\end{equation*}
This contradicts the assumption that 
\(
R_2 < \xi_2 n^3.
\)

Let $Y_1 = \{v \in V_{l_1} : \exists\, i \in [r],\ d_i(v) \ge d(v) - (r-1)c_2 n,\ d(v) \ge (1 - \sqrt{\epsilon_2}) |V_{l_2}|\}$,
$Y_2 = \{v \in V_{l_2} : \exists\, i \in [r],\ d_i(v) \ge d(v) - (r-1)c_2 n,\ d(v) \ge (1 - \sqrt{\epsilon_2}) |V_{l_1}|\}$,  
and define $Y = Y_1 \cup Y_2$ (note that $Y = Y_1 = Y_2$ when $l_1 = l_2$).
It follows that  
$|Y| \ge |V(G)| - c_1 n - \sqrt{\epsilon_2}n$.
For each vertex $y \in Y$, we have 
$d_{i^y}(y) \ge (1 - \sqrt{\epsilon_2}) |V_{l_2}| - (r-1)c_2 n$ if $y \in Y_1$, and  
$d_{i^y}(y) \ge (1 - \sqrt{\epsilon_2}) |V_{l_1}| - (r-1)c_2 n$ if $y \in Y_2$.
For each $i \in [r]$, let $U_i = \{y \in Y : i^y = i\}$.

\begin{claim}\label{claim:2.13}
There is $i\in[r]$ such that $\sum_{j\in [r]\setminus\{i\}} |U_j|\le (r-1) c_3 n$. 
\end{claim}

\begin{proof}
Suppose there exist two colors $i, j$ such that $|U_i| \ge c_3 n$ and $|U_j| \ge c_3 n$.
We distinguish two cases. By symmetry, assume first that $U_i, U_j \subseteq Y_1$.  
Fix a vertex $a \in U_i$ and a vertex $b \in U_j$.  
Let $N_a = \{a' \in V_{l_2} : \phi(aa') = i\}$ and $N_b = \{b' \in V_{l_2} : \phi(bb') = j\}$.  
Then  
$|N_a| \ge (1 - \sqrt{\epsilon_2}) |V_{l_2}| - (r-1)c_2 n$,  
$|N_b| \ge (1 - \sqrt{\epsilon_2}) |V_{l_2}| - (r-1)c_2 n$, and  
\[
|N_a \cap N_b| \ge |V_{l_2}| - 2\sqrt{\epsilon_2} |V_{l_2}| - 2 (r-1)c_2 n.
\]
Since $\epsilon_2,\xi_2 \ll c_2 \ll c_3 \ll \delta$, we have
\begin{equation*}
\begin{aligned}
R_2 &\ge \sum_{a \in U_i} \sum_{b \in U_j} \sum_{w \in N_a \cap N_b} 1 \\
    &\ge c_3 n \cdot c_3 n \cdot (|V_{l_2}| - 2 \sqrt{\epsilon_2} |V_{l_2}| - 2 (r-1)c_2 n) \\
    &\ge c_3 n \cdot c_3 n \cdot (\delta n - 2 \sqrt{\epsilon_2} n - 2 (r-1)c_2 n) \\
    &> \xi_2 n^3.
\end{aligned}
\end{equation*}
This contradicts the assumption that $R_2 < \xi_2 n^3$.

Now assume (again by symmetry) that $U_i \subseteq Y_1$ and $U_j \subseteq Y_2$.  
Fix a vertex $a \in U_i$, and let $N_a = \{b \in V_{l_2} : \phi(ab) = i\}$.  
Then $|N_a| \ge (1 - \sqrt{\epsilon_2}) |V_{l_2}| - (r-1)c_2 n$, and so
\[
|N_a \cap U_j| \ge |U_j| - \sqrt{\epsilon_2} |V_{l_2}| - (r-1)c_2 n \ge c_3 n - (\sqrt{\epsilon_2} n + (r-1)c_2 n).
\]
Pick a vertex $u \in N_a \cap U_j$, and 
let $N_u = \{w \in V_{l_1} : \phi(uw) = j\}$. 
Then $|N_u| \ge (1 - \sqrt{\epsilon_2}) |V_{l_1}| - (r-1)c_2 n$.  
Thus,
\begin{equation*}
\begin{aligned}
R_2 &\ge \sum_{a \in U_i} \sum_{u \in N_a \cap U_j} \sum_{w \in N_u} 1 \\
    &\geq c_3n\cdot (c_3n-\sqrt{\epsilon_2}\size{V_{l_2}} - (r-1)c_2n)\cdot ((1 -\sqrt{\epsilon_2})\size{V_{l_1}} - (r-1)c_2n)\\ 
    &\ge c_3 n \cdot (c_3 n - \sqrt{\epsilon_2} n - (r-1)c_2 n) \cdot (\delta n - \sqrt{\epsilon_2} n - (r-1)c_2 n) \\
    &> \xi_2 n^3,
\end{aligned}
\end{equation*}
again a contradiction.
Therefore, there must exist a color $i \in [r]$ such that
\[
\sum_{j \in [r] \setminus \{i\}} |U_j| \le (r - 1) c_3  n.\qedhere
\]
\end{proof}
It follows from Claim \ref{claim:2.13} and the choice $\epsilon_2\ll c_1, c_2\ll c_3\ll\zeta_2,\delta$ that the number of edges not colored $i$ in $H^2$ is at most
\begin{equation*}
\begin{aligned}
&\sum_{j \in [r] \setminus \{i\}} |U_j| \cdot n + |V(G) \setminus Y| \cdot n + (r-1)c_2n\cdot n\\
&\le (r - 1) \cdot c_3 \cdot n^2 + (c_1 n + \sqrt{\epsilon_2} n) \cdot n + (r-1)c_2n^2\\
&\le \zeta_2 \cdot e(H^2).
\end{aligned}
\end{equation*}
Hence, $H^2$ is $\zeta_2$-monochromatic.

Now let $k \geq 3$, and assume the statement holds for every integer $k'$ with $2 \le k' < k$. Since $H^k$ is $(\epsilon_k, \mathbf{x})$-complete, we have 
$e(H^k) \geq (1 - \epsilon_k) \prod_{j=1}^r \binom{\size{V_j}}{x_j}$.
Let $I^k = \{ i : x_i \ne 0 \}$. Then $1 \le \size{I^k} \le k$.  
Define $G := H^k\big[ \bigcup_{i \in I^k} V_i \big]$. 
Hence, $\size{V(G)} = \sum_{i \in I^k} \size{V_i} \ge \delta n$ and 
$e(G) \ge (1 - \epsilon_k) \cdot \prod_{j \in I^k} \binom{\size{V_j}}{x_j}$.
Without loss of generality, assume $x_1 \ne 0$ and let $\hat{\mathbf{x}} = (x_1 - 1, x_2, \dots, x_r)$. 
Then there are at least 
$(1 - \sqrt{\epsilon_k}) \binom{\size{V_1}}{x_1 - 1} \prod_{j \in I^k \setminus \{1\}} \binom{\size{V_j}}{x_j}$ 
many $(k-1)$-sets $S = \{v_1, \dots, v_{k-1}\}$ with $\mathbf{i}_{\mathcal{P}}(S) = \hat{\mathbf{x}}$ such that 
$d_G(S) \ge (1 - \sqrt{\epsilon_k}) \size{V_1}$.
For every $(k-1)$-set $S$ in $G$ and $i\in[r]$, define
\[
E_i(S) := \left\{e \in E(G) : S\<e,~\phi(e) = i \right\},
\]
and let $d_i(S) := \size{E_i(S)}$.

We claim that
for all but at most $c_1n^{k-1}$ $(k-1)$-sets $S$ in $G$ with $\mathbf{i}_{\mathcal{P}}(S) = \hat{\mathbf{x}}$, there exists some $i^S \in [r]$ such that
\[
d_{i^S}(S) \ge d_G(S) - (r-1)c_2n.
\]
Suppose for the contrary that there are more than $c_1n^{k-1}$ such $(k-1)$-sets $S$ in $G$ with $\mathbf{i}_{\mathcal{P}}(S) = \hat{\mathbf{x}}$ for which
\[
\sum_{j \in [r] \setminus \{i\}} d_j(S) > (r-1)c_2n \quad \text{for every } i \in [r].
\]
Let $\mathcal{X}^{k-1}$ denote the collection of these $(k-1)$-sets, so $\size{\mathcal{X}^{k-1}} \ge c_1n^{k-1}$. Therefore, for each $S \in \mathcal{X}^{k-1}$, there exist two distinct colors $i(S), j(S) \in [r]$ such that
\[
d_{i(S)}(S) \ge c_2n \quad \text{and} \quad d_{j(S)}(S) \ge c_2n.
\]
Hence, as $\xi_k \ll c_1,c_2$,
\begin{equation*}
\begin{aligned}
R_k &\ge \sum_{S \in \mathcal{X}^{k-1}} d_{i(S)}(S) \cdot d_{j(S)}(S) \\
&\ge c_1n^{k-1} \cdot c_2n \cdot c_2n \\
&= c_1c_2^2 \cdot n^{k+1} \\
&> \xi_k n^{k+1}.
\end{aligned}
\end{equation*}
This contradicts our assumption that
\(
R_k < \xi_k n^{k+1}. 
\)

Additionally pick constants $\epsilon_{k-1}$, $\xi_{k-1}$, and $\zeta_{k-1}$ such that
$1/n \ll \epsilon_k, \xi_k \ll  \xi_{k-1} \ll c_1,c_2  \ll \epsilon_{k-1} \ll \zeta_{k-1}\ll \zeta_{k},\delta.$
Let $\mathcal{S}^{k-1}$ be the collection of all $(k-1)$-sets $S \subseteq V(G)$ with $\mathbf{i}_{\mathcal{P}}(S) = \hat{\mathbf{x}}$, such that there exists $i^S \in [r]$ satisfying
\[
d_{i^S}(S) \geq d_G(S) - (r-1)c_2n \quad \text{and} \quad d_G(S) \geq (1 - \sqrt{\epsilon_k}) |V_1|.
\]
It follows from the aforementioned claim that
\[
|\mathcal{S}^{k-1}| \geq (1 - \sqrt{\epsilon_k}) \binom{|V_1|}{x_1 - 1} \prod_{j \in I^k \setminus \{1\}} \binom{|V_j|}{x_j} - c_1 n^{k-1}.
\]
For each $(k-1)$-set $S \in \mathcal{S}^{k-1}$, we have
\[
d_{i^S}(S) \geq (1 - \sqrt{\epsilon_k}) |V_1| - (r-1)c_2 n.
\]
Let $Q_i = \{S \in \mathcal{S}^{k-1} : i^S = i \}$ for each $i \in [r]$. Construct a $(k-1)$-graph $H^{k-1}$ with vertex set $V(H^{k-1}) = V(G)$ and edge set $E(H^{k-1}) = \mathcal{S}^{k-1}$. Define an edge-coloring $\psi$ on $H^{k-1}$ by setting $\psi(S) = i^S$ for each $S \in E(H^{k-1})$ and let $R_{k-1}$ be the number of rainbow tight 2-paths in $H^{k-1}$.
Then
\[
\begin{aligned}
d^{\hat{\mathbf{x}}}(H^{k-1}) &= \frac{|\mathcal{S}^{k-1}|}{\binom{|V_1|}{x_1 - 1} \prod_{j = 2}^{r} \binom{|V_j|}{x_j}} \\
&\geq \frac{(1 - \sqrt{\epsilon_k}) \binom{|V_1|}{x_1 - 1} \prod_{j \in I^k \setminus \{1\}} \binom{|V_j|}{x_j} - c_1 n^{k-1}}{\binom{|V_1|}{x_1 - 1} \prod_{j = 2}^{r} \binom{|V_j|}{x_j}} \\
%&> 1 - \sqrt{\epsilon_k} -  c_1 (x_1-1)!\prod_{j=2}^r x_j! \\
& > 1 - \sqrt{\epsilon_k} -  c_1 \factorial{k}\\
&> 1 - \epsilon_{k-1},
\end{aligned}
\]

\noindent where the penultimate inequality uses Stirling's formula. Then $H^{k-1}$ is $(\epsilon_{k-1}, \hat{\mathbf{x}})$-complete.

We further claim that $R_{k-1} \leq \xi_{k-1} n^k$. Suppose for contradiction that $R_{k-1} > \xi_{k-1} n^k$. Since any two $(k-1)$-edges in $H^{k-1}$ have at least $|V_1| - 2 \sqrt{\epsilon_k} |V_1| - 2 (r-1)c_2 n$ common neighbors in $G$, we have
\[
\begin{aligned}
R_k &\geq R_{k-1} \cdot \left( |V_1| - 2 \sqrt{\epsilon_k} |V_1| - 2 (r-1)c_2 n \right) \\
&\geq \xi_{k-1} n^k \cdot \left( \delta n - 2 \sqrt{\epsilon_k} n - 2 (r-1)c_2 n \right) \\
&> \xi_k n^{k+1},
\end{aligned}
\]
which contradicts the assumption that $R_k < \xi_k n^{k+1}$. 
By the choice $1/n \ll \xi_{k-1} \ll \epsilon_{k-1} \ll \zeta_{k-1}\ll \delta$ and the induction hypothesis, $H^{k-1}$ is $\zeta_{k-1}$-monochromatic. Therefore, there are at least
\begin{equation*}
  \begin{aligned}
  &(1-\zeta_{k-1}) e(H^{k-1}) \big((1-\sqrt{\epsilon_k})\size{V_1}- (r-1)c_2n\big) \\
  &\ge (1-\zeta_{k-1}) \bigg((1-\sqrt{\epsilon_k})\binom{\size{V_1}}{x_1 - 1}\prod_{j\in I^k\setminus\{1\}}\binom{\size{V_j}}{x_j} - c_1 n^{k-1}\bigg) \big((1-\sqrt{\epsilon_k})\size{V_1} - (r-1)c_2n\big) \\ 
  & \ge (1-\zeta_k)e(H^k)
 \end{aligned}   
 \end{equation*}
edges in $H^k$ that have the same color, where the last inequality follows from the choice $\epsilon_k \ll c_1,c_2 \ll\zeta_{k-1}  \ll \zeta_k$.
Hence, $H^k$ is $\zeta_k$-monochromatic.
\end{proof}

\section{Almost perfect tilings} \label{sec:tiling}
This section proves the almost cover lemma (Lemma~\ref{lem:tiling}). 
We first establish the fractional analogue (Lemma~\ref{frac_tiling_with_bound}), then apply a convenient reformulation of the Pippenger–Spencer theorem~\cite{PS} obtained by Bowtell, Kathapurkar, Morrison and Mycroft \cite{BKMM} to deduce Lemma \ref{lem:tiling}. To begin with, we introduce the following standard definitions for fractional tilings in $k$-graphs.

Let \(H\) and \(F\) be two \(k\)-graphs. Define \(\mathcal{F}(H)\) as the set of all copies \(F'\) of \(F\) in \(H\). For two distinct vertices \(u, v \in V(H)\), further define
$\mathcal{F}_u(H)$ as the set of copies \(F' \in \mathcal{F}(H)\) containing \(u\), \(\mathcal{F}_{uv}(H)\) as the set of copies \(F' \in \mathcal{F}(H)\) containing both \(u\) and \(v\).
A \emph{fractional \(F\)-tiling} of \(H\) is a function \(\omega : \mathcal{F}(H) \to [0, 1]\) satisfying
\(
\sum_{F' \in \mathcal{F}_u(H)} \omega(F') \leq 1
\)
for every $u \in V(H)$.
%This extends the integer \(F\)-tiling problem: if \(\omega(F') \in \{0,1\}\) for all \(F' \in \mathcal{F}(H)\), then \(\{F' : \omega(F') = 1\}\) corresponds to an \(F\)-tiling in \(H\).
Furthermore, we say \(\omega\) is \emph{perfect} if
\(
\sum_{F' \in \mathcal{F}_u(H)} \omega(F') = 1
\)
for every $u\in V(H)$.
For any two distinct vertices \(u, v \in V(H)\), define
$\omega(u) = \sum_{F' \in \mathcal{F}_u(H)} \omega(F')$ and
$\omega(uv) = \sum_{F' \in \mathcal{F}_{uv}(H)} \omega(F')$.
When we consider $F$-tilings in the specific case where $F=T_k$, we write $\mathcal{T}_k(H)$ in place of $\mathcal{F}(H)$.

We now characterize perfect fractional $F$-tilings in $k$-graphs. Let $H$ be a $k$-graph with vertex set $V(H) = [n]$. The \emph{characteristic vector} of $U \subseteq [n]$ is $\mathds{1}_U \in \mathbb{R}^n$, where $(\mathds{1}_U)_i=1$ if $i \in U$, $(\mathds{1}_U)_i=0$ otherwise.
The \emph{positive cone} of $S = \{\mathbf{v}_1, \dots, \mathbf{v}_t\} \subseteq \mathbb{R}^n$ is
\[
\operatorname{cone}(S) := \left\{ \textstyle\sum_{i=1}^t \lambda_i \mathbf{v}_i : \lambda_i\in \mathbb{R}, \lambda_i\ge 0~\text{for each}~i\in [t] \right\}.
\]

Let $\mathbf{1} \in \mathbb{R}^n$ denote the all-ones vector. Then $H$ has a perfect fractional $F$-tiling if and only if 
\(
\mathbf{1} \in \operatorname{cone}\left( \left\{ \mathds{1}_{V(F')} : F' \in \mathcal{F}(H) \right\} \right),
\)
where the tiling weights $\omega(F')$ correspond to the coefficients $\lambda_i$ for each copy $F'\in \mathcal{F}(H)$ of $F$.

Our key application of Farkas' lemma for linear systems provides a useful consequence of non-existence of tiling: $H$ contains no perfect fractional $F$-tiling, namely that there exists some vector $\mathbf{a} \in \mathbb{R}^n$ satisfying
\(
\mathbf{a} \cdot \mathds{1}_{V(F')} \geq 0 \mbox{ for every } F' \in \mathcal{F}(H),
\)
while simultaneously satisfying
\(
\mathbf{a} \cdot \mathbf{1} < 0.
\)

\begin{lemma}[Farkas' lemma, \cite{LY}] \label{lem:farkas}
    If $X \subseteq \mathbb{R}^n$ is finite and ${\bf y} \in \mathbb{R}^n \setminus \operatorname{cone}(X)$, then there exists some ${\bf a} \in \mathbb{R}^n$ such that ${\bf a}\cdot {\bf x} \geq 0$ for all ${\bf x} \in X$ and ${\bf a} \cdot {\bf y} <0$.
\end{lemma}

The following additional definitions are required in our proof.
Let $H$ be a $k$-graph with vertex set $V(H)=[n]$. Let $U = \{u_1, u_2, \dots, u_{\ell}\}$ and $W = \{w_1, w_2, \dots, w_{\ell}\}$ be two ordered $\ell$-subsets of $V(H)$ satisfying $u_1 < u_2 < \cdots < u_{\ell}$ and $w_1 < w_2 < \cdots < w_{\ell}$. We say $U$ {\em dominates} $W$ if $w_i \leq u_i$ for every $i \in [\ell]$.

\subsection{$B$-avoiding fractional tilings}

Given a graph \(B\) with \(V(B) = V(H)\), a copy \(F' \in \mathcal{F}(H)\) is \emph{\(B\)-avoiding} if \(E(B)\) contains no edge \(uv\) with \(u, v \in V(F')\). A fractional $F$-tiling $\omega$ is \emph{$B$-avoiding} if $\omega(F')>0$ implies $F'$ is $B$-avoiding; equivalently, \(\omega(uv) = 0\) for every \(uv \in E(B)\).
Our next lemma shows that under identical minimum codegree conditions, every $k$-graph on $n$ vertices either is
$\gamma$-extremal or admits a perfect $B$-avoiding fractional tiling. Our proof primarily draws on the approach used by Bowtell, Kathapurkar, Morrison and Mycroft~\cite{BKMM}.

\begin{lemma}\label{frac_tiling} 
Suppose $1/n \ll \varepsilon, \alpha \ll \gamma$ and let $V=\{v_1,v_2,\ldots,v_n\}$ be an $n$-vertex set ordered with $v_1<v_2<\dots<v_n$. For any $k$-graph $H = (V, E)$ with $\delta(H) \geq \left(\frac{2}{2k-1} - \alpha\right)n$ and a graph $B$ on $V$ with maximum degree $\Delta(B) \leq \varepsilon n$, either $H$ admits a perfect $B$-avoiding fractional $T_k$-tiling or $H$ is $\gamma$-extremal.
\end{lemma}

We postpone the proof of Lemma~\ref{frac_tiling} to Section~\ref{sub2}. Using Lemma~\ref{frac_tiling}, we show there exists an alternative fractional 
$T_k$-tiling whose weight is bounded on every pair of vertices,
where the proof strategy directly adapts the approach of Lemma 4.3 in \cite{BKMM}. Nevertheless, to ensure expository completeness and clarity, we provide a self-contained proof below.
\begin{lemma} \label{frac_tiling_with_bound} 
Suppose that $1/n \ll \eps, \alpha \ll \gamma$. Let $H$ be a $k$-graph on $n$ vertices with $\delta(H) \geq 2n/(2k-1)-\alpha n$. 
If $H$ is not $\gamma$-extremal, then $H$ admits a perfect fractional $T_k$-tiling $\omega$ in which $\omega(uv) < 1/(\eps n)$ for all $u, v \in V(H)$.
\end{lemma}

\begin{proof}
By Lemma~\ref{frac_tiling} and the assumption that $H$ is not $\gamma$-extremal, there must exist a perfect fractional $T_k$-tiling. Furthermore, we let $W = \min_{\omega}   \psi(\omega)$, where the minimum is over all perfect fractional $T_k$-tilings $\omega$ in $H$, and $\psi(\omega) = \max_{uv \in \binom{V(H)}{2}} \omega(uv)$. Fix a perfect fractional $T_k$-tilings $\omega$ achieving $\psi(\omega) = W$. 
Then we may choose $\mu \in (0, W)$ such that for every pair $uv$, either $\omega(uv) = W$ or $\omega(uv) < W - \mu$.

Suppose for contradiction that $W \geq 1/(\varepsilon n)$. Define a graph $B$ on $V(H)$ where 
$E(B) = \{ uv : \omega(uv) = W \}$. For any vertex $u$,
\[
W \cdot d_B(u) \leq \sum_{\substack{v \in V(H) \\ v \neq u}} \omega(uv) 
= (2k-2) \omega(u) = 2k-2.
\]
Thus, $\Delta(B) \leq (2k-2)/W \leq (2k-2)\varepsilon n$.

By Lemma~\ref{frac_tiling}, there exists a perfect $B$-avoiding fractional $T_k$-tiling $\omega'$. Then $\omega'(uv) = 0$ for all $uv \in E(B)$. For each $T\in \mathcal{T}_k(H)$, we define a new fractional $T_k$-tiling $\omega''$ satisfying
$\omega''(T) = (1-\mu) \omega(T) + \mu \omega'(T)$.
Hence, for every $u\in V(H)$ we have
\[
\omega''(u) = (1-\mu)\omega(u) + \mu\omega'(u) = (1-\mu) + \mu = 1,
\]
so $\omega''$ is a perfect fractional $T_k$-tiling in $H$. For any pair $uv$, if $uv \in E(B)$,
then $\omega''(uv) = (1-\mu)W < W$.
Otherwise, 
\[
\omega''(uv)=(1-\mu)\omega(uv)+\mu\omega'(uv) < (1-\mu)(W - \mu) + \mu = W - \mu(W - \mu) < W,
\]
which follows from that $\mu(W - \mu) > 0$.
Thus $\psi(\omega'') < W$, contradicting the minimality of $W$. We conclude $W < 1/(\varepsilon n)$.
\end{proof}

Theorem~\ref{lem:tiling} follows by combining Lemma~\ref{frac_tiling_with_bound} with a result of Bowtell, Kathapurkar, Morrison and Mycroft \cite{BKMM}, which restates a special case of Pippenger and Spencer's theorem~\cite{PS}.

\begin{lemma}[\cite{BKMM}] \label{frac_to_almost}
Fix $k \geq 2$ and suppose that $1/n \ll \eps \ll \eta, 1/k$, and let $H$ be a $k$-graph on $n$ vertices. If $H$ admits a fractional matching $\omega$ in which $\omega(u) \geq 1-\eps$ for every $u \in V(H)$ and $\omega(uv) \leq \eps$ for all $u, v \in V(H)$, then $H$ admits a matching of size at least $(1-\eta) n/k$.
\end{lemma}

We now complete the proof of Lemma~\ref{lem:tiling}.  

\begin{proof}[\bf Proof of Lemma~\ref{lem:tiling}]
Fix a constant $\eps$ satisfying $1/n \ll \eps \ll \eta\ll \gamma$. Suppose that $H$ is not $\gamma$-extremal, then Lemma~\ref{frac_tiling_with_bound} guarantees a perfect fractional $T_k$-tiling $\omega$ in $H$ satisfying $\omega(uv) < 1/(\eps n)$ for all distinct $u, v \in V(H)$. 

Define a $(2k-1)$-graph $H'$ on $V(H)$ where hyperedges are all sets $S \in \binom{V(H)}{2k-1}$ that support $T_k$ in $H$. Assign weights $\omega'(S) = \sum_{T \in\mathcal{T}_k(H[S])} \omega(T)$ for each such $S$. This forms a perfect fractional matching $\omega'$ in $H'$ satisfying:
\[
\sum_{\substack{S \in E(H') \\ u,v\in S}} \omega'(S) = \omega(uv) < \frac{1}{\eps n}.
\]
By Lemma~\ref{frac_to_almost}, $H'$ contains a matching $M = \{S_1,\dots,S_r\}$ of size $|M| \geq (1-\eta)n/(2k-1)$. For each $S_i \in M$, select a copy $T^i$ of $T_k$ in $H$ with $V(T^i) = S_i$. The collection $\{T^1,\dots,T^r\}$ then covers at least $(1-\eta)n$ vertices of $H$, as required.
\end{proof}

\subsection{Proof of Lemma~\ref{frac_tiling}}\label{sub2}
In this section, we present a proof of Lemma~\ref{frac_tiling}.

\begin{proof}
It suffices to prove the lemma when $(2k-1) \mid n$ (we defer the reduction to this case until the end of the proof). Assume $(2k-1) \mid n$ and define
\(
\mathcal{T}_{k,B} = \{T \in \mathcal{T}_k(H) : T \text{ is $B$-avoiding}\}.
\)
Then $H$ admits a perfect $B$-avoiding fractional $T_k$-tiling if and only if 
\(
\bm{1} \in \operatorname{cone}(\{\mathds{1}_{V(T)} : T \in \mathcal{T}_{k,B}\}).
\)

We proceed by contradiction. Suppose $H$ is not $\gamma$-extremal and lacks a perfect $B$-avoiding fractional $T_k$-tiling. Then 
$\bm{1} \notin \operatorname{cone}(\{\mathds{1}_{V(T)} : T \in \mathcal{T}_{k,B}\})$.
By Lemma~\ref{lem:farkas}, there exists $\bm{a} = (a_1, \dots, a_n) \in \mathbb{R}^n$ satisfying $\bm{a} \cdot \bm{1} < 0$ and $\bm{a} \cdot \mathds{1}_{V(T)} \ge 0$ for all $T \in \mathcal{T}_{k,B}$.
Without loss of generality, we assume $a_1 \leq \cdots \leq a_n$.

Recall our assumption that $1/n\ll \eps,\alpha\ll \gamma$. 
Take a minimal constant $\beta$ such that $\beta \geq \alpha + (2k-1)\varepsilon$ and $\beta n$ is an integer. 
We partition $V(H)$ into three disjoint families $\mathcal{V}_1, \mathcal{V}_2, \mathcal{V}_3$ of $(2k-1)$-subsets as follows: 

\begin{itemize}
    \item $\mathcal{V}_1 = \biggl\{ \bigl\{ 
        v_i,\ 
        v_{\beta n + i},\ 
        \ldots,\ 
        v_{(k-2)\beta n + i},
            v_{\frac{2k-3}{2k-1}n+\beta n+ i},
            \ldots, 
            v_{\nkterm{2k-3}{k} + i}\bigl\} : i \in [\beta n] \biggr\}$
    
    \item $\mathcal{V}_2 = \biggl\{ \bigl\{ 
        v_{\nknegterm{2k-3}{(2k-4)k+1} + i},\ 
        v_{\nknegterm{2k-3}{(2k-5)k+2} + i},\ 
        \ldots,\ 
        v_{\nknegterm{2k-3}{k-2} + i}, \\
        \hspace*{3.5cm}v_{\nkterm{2k-3}{(k+1)} + i}
        \bigr\} : i \in [(k-1)\beta n] \biggr\}$
    
    \item $\mathcal{V}_3 = \biggl\{ \bigl\{ 
        v_{(k-1)\beta n + i},\ 
        v_{\nk - \beta n + i},\ 
        \ldots,\ 
        v_{\nknegterm{2k-4}{(2k-5)k+1} + i},
            v_{\nkterm{2k-3}{2k} + i}, \\
        \hspace*{3.5cm}
            v_{\nkterm{2k-2}{k} + i}
            \bigr\} : i \in [\nk - k\beta n] \biggr\}$
\end{itemize}

In doing this, our goal is to show there exist $T^1,T^2,T^3 \in \mathcal{T}_{k, B}$ such that for each $i \in [3]$, every $S \in \mathcal{V}_i$ dominates $T^i$. In this case, we have $\mathbf{a} \cdot \mathds{1}_{V(T^i)} \leq \mathbf{a} \cdot \mathds{1}_S$, and it follows from $\sum_{i=1}^3 \sum_{S \in \mathcal{V}_i} \mathds{1}_S={\bf 1}$ that
\begin{equation*}
\begin{aligned}
0&>\mathbf{a}\cdot \mathbf{1}=\mathbf{a}\cdot \bigg(\sum_{i=1}^3 \sum_{S \in \mathcal{V}_i} \mathds{1}_S\bigg)\\
&\ge \beta n(\mathbf{a}\cdot \mathds{1}_{V(T^1)})+(k-1)\beta n(\mathbf{a}\cdot \mathds{1}_{V(T^2)})+\bigg(\frac{n}{2k-1}-k\beta n\bigg)(\mathbf{a}\cdot \mathds{1}_{V(T^3)})\\
&\ge 0,    
\end{aligned}
\end{equation*}
a contradiction.
Hence, $H$ is either $\gamma$-extremal or admits a perfect $B$-avoiding fractional $T_k$-tiling.

Our next goal is to find $T^1$, $T^2$ and $T^3$ satisfying the aforementioned conditions. Note that $T^i \in \mathcal{T}_{k,B}$ means $V(T^i)$ induces no edges in $B$ for each $i\in [3]$. Hence, $V(T^i)$ would form a complete graph in the complement graph $\bar{B}$.

We begin by establishing the existence of $T^1$. Let $G=\bar{B}$.
Fix $v_1$ to be the first vertex of $T^1$.
Sequentially select vertices $\{v_{i_2}, \dots, v_{i_{k-1}}\}$ with indices $2 \leq i_2 < \cdots < i_{k-1} \leq \bconst n$ such that $\{v_1, v_{i_2}, \dots, v_{i_{k-1}}\}$ forms a clique in $\Gind$. This selection is feasible as each vertex has at most $\eps n$ non-neighbors in $\Gind$ and $\bconst \geq \alpha + (2k-1)\econst$.
We now sequentially select vertices $v_s$ and $v_t$ such that 
$s, t \leq \left( \frac{2k-3}{2k-1} + \alpha + (2k-1)\varepsilon \right)n$, $v_s, v_t \in N(v_1 v_{i_2} \dots v_{i_{k-1}})$ and $\{v_1, v_{i_2}, \dots, v_{i_{k-1}}, v_s, v_t\}$ induces a $(k+1)$-clique in $G$. The selection of $v_s$ and $v_t$ is valid by the same reason as the choice of vertices $\{v_{i_2}, \dots, v_{i_{k-1}}\}$, combined with the fact that \(d_H(v_1v_{i_2}\dots v_{i_{k-1}}) \geq \delta(H) \geq \frac{2n}{2k-1} - \alpha n.\). Then we select vertices $v_{j_1}, \dots, v_{j_{k-3}}$ satisfying $j_1 \leq \cdots \leq j_{k-3} \leq \frac{2k-3}{2k-1}n + \beta n$.
Similarly, there exists $m \leq \frac{2k-3}{2k-1}n + \beta n$ with $v_m \in N(v_s v_t v_{j_1} \dots v_{j_{k-3}})$ such that  
$\{v_1, v_{i_2},\dots, v_{i_{k-1}}, v_s, v_t, v_{j_1},\dots,v_{j_{k-3}}, v_m\}$ forms a $(2k-1)$-clique in $G$.
Thus, we obtain a $B$-avoiding $T^1 \in \mathcal{T}_{k,B}$ that is indeed dominated by every set in $\mathcal{V}_1$. As illustrated in Figure~\ref{fig:T1}, distinct edges are represented using differently colored curves.

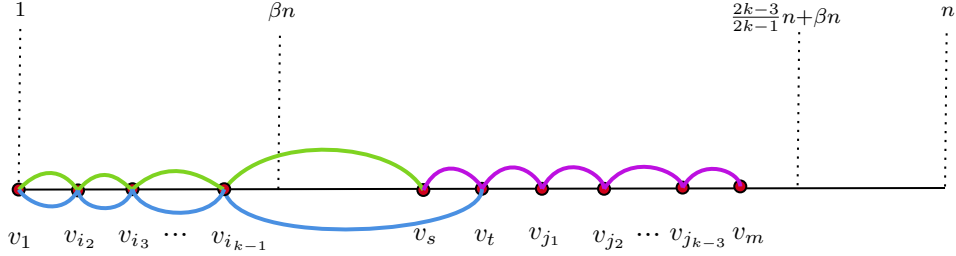
\begin{figure}[htbp]
\centering

\tikzset{every picture/.style={line width=0.75pt}} %set default line width to 0.75pt        

\begin{tikzpicture}[x=0.75pt,y=0.75pt,yscale=-1,xscale=1]
%uncomment if require: \path (0,475); %set diagram left start at 0, and has height of 475

%Straight Lines [id:da30437460358455615] 
\draw    (47.8,341.6) -- (510.2,340.8) ;
%Shape: Circle [id:dp8347945251208291] 
\draw  [fill={rgb, 255:red, 208; green, 2; blue, 27 }  ,fill opacity=1 ] (44.8,341.6) .. controls (44.8,339.94) and (46.14,338.6) .. (47.8,338.6) .. controls (49.46,338.6) and (50.8,339.94) .. (50.8,341.6) .. controls (50.8,343.26) and (49.46,344.6) .. (47.8,344.6) .. controls (46.14,344.6) and (44.8,343.26) .. (44.8,341.6) -- cycle ;
%Shape: Circle [id:dp6893894468171657] 
\draw  [fill={rgb, 255:red, 208; green, 2; blue, 27 }  ,fill opacity=1 ] (74.4,342) .. controls (74.4,340.34) and (75.74,339) .. (77.4,339) .. controls (79.06,339) and (80.4,340.34) .. (80.4,342) .. controls (80.4,343.66) and (79.06,345) .. (77.4,345) .. controls (75.74,345) and (74.4,343.66) .. (74.4,342) -- cycle ;
%Shape: Circle [id:dp4459992234885566] 
\draw  [fill={rgb, 255:red, 208; green, 2; blue, 27 }  ,fill opacity=1 ] (101.6,341.4) .. controls (101.6,339.74) and (102.94,338.4) .. (104.6,338.4) .. controls (106.26,338.4) and (107.6,339.74) .. (107.6,341.4) .. controls (107.6,343.06) and (106.26,344.4) .. (104.6,344.4) .. controls (102.94,344.4) and (101.6,343.06) .. (101.6,341.4) -- cycle ;
%Shape: Circle [id:dp5569055530121779] 
\draw  [fill={rgb, 255:red, 208; green, 2; blue, 27 }  ,fill opacity=1 ] (246.8,341.8) .. controls (246.8,340.14) and (248.14,338.8) .. (249.8,338.8) .. controls (251.46,338.8) and (252.8,340.14) .. (252.8,341.8) .. controls (252.8,343.46) and (251.46,344.8) .. (249.8,344.8) .. controls (248.14,344.8) and (246.8,343.46) .. (246.8,341.8) -- cycle ;
%Shape: Circle [id:dp7081739330611756] 
\draw  [fill={rgb, 255:red, 208; green, 2; blue, 27 }  ,fill opacity=1 ] (276,341.2) .. controls (276,339.54) and (277.34,338.2) .. (279,338.2) .. controls (280.66,338.2) and (282,339.54) .. (282,341.2) .. controls (282,342.86) and (280.66,344.2) .. (279,344.2) .. controls (277.34,344.2) and (276,342.86) .. (276,341.2) -- cycle ;
%Shape: Circle [id:dp9968363393870333] 
\draw  [fill={rgb, 255:red, 208; green, 2; blue, 27 }  ,fill opacity=1 ] (306,341.2) .. controls (306,339.54) and (307.34,338.2) .. (309,338.2) .. controls (310.66,338.2) and (312,339.54) .. (312,341.2) .. controls (312,342.86) and (310.66,344.2) .. (309,344.2) .. controls (307.34,344.2) and (306,342.86) .. (306,341.2) -- cycle ;
%Shape: Circle [id:dp15271066133710876] 
\draw  [fill={rgb, 255:red, 208; green, 2; blue, 27 }  ,fill opacity=1 ] (337,341.2) .. controls (337,339.54) and (338.34,338.2) .. (340,338.2) .. controls (341.66,338.2) and (343,339.54) .. (343,341.2) .. controls (343,342.86) and (341.66,344.2) .. (340,344.2) .. controls (338.34,344.2) and (337,342.86) .. (337,341.2) -- cycle ;
%Straight Lines [id:da7643123075942782] 
\draw [line width=0.75]  [dash pattern={on 0.84pt off 2.51pt}]  (48.6,260.8) -- (47.8,338.6) ;
%Straight Lines [id:da750984067151441] 
\draw  [dash pattern={on 0.84pt off 2.51pt}]  (178.2,264.2) -- (177.4,342) ;
%Straight Lines [id:da5098548997301683] 
\draw  [dash pattern={on 0.84pt off 2.51pt}]  (437.4,262.4) -- (436.6,340.2) ;
%Straight Lines [id:da37235614066449696] 
\draw  [dash pattern={on 0.84pt off 2.51pt}]  (511,263) -- (510.2,340.8) ;
%Curve Lines [id:da5588190984685553] 
\draw [color={rgb, 255:red, 126; green, 211; blue, 33 }  ,draw opacity=1 ][line width=1.5]    (47.8,341.6) .. controls (66.2,322.4) and (77.4,342.4) .. (77.4,342) ;
%Curve Lines [id:da8219890635092131] 
\draw [color={rgb, 255:red, 126; green, 211; blue, 33 }  ,draw opacity=1 ][line width=1.5]    (77.4,342) .. controls (91.8,324.8) and (102.2,341.6) .. (104.6,341.4) ;
%Curve Lines [id:da8850142574929689] 
\draw [color={rgb, 255:red, 126; green, 211; blue, 33 }  ,draw opacity=1 ][line width=1.5]    (150.4,341.4) .. controls (171.8,312.8) and (235.8,317) .. (249.8,341.8) ;
%Curve Lines [id:da9312902311594206] 
\draw [color={rgb, 255:red, 74; green, 144; blue, 226 }  ,draw opacity=1 ][line width=1.5]    (47.8,341.6) .. controls (52.6,349) and (69.4,356) .. (77.4,342) ;
%Curve Lines [id:da019946919502994542] 
\draw [color={rgb, 255:red, 74; green, 144; blue, 226 }  ,draw opacity=1 ][line width=1.5]    (77.4,342) .. controls (79,352.8) and (100.6,355.2) .. (104.6,341.4) ;
%Curve Lines [id:da38056071806084646] 
\draw [color={rgb, 255:red, 74; green, 144; blue, 226 }  ,draw opacity=1 ][line width=1.5]    (150.4,341.4) .. controls (155.6,371.77) and (279.8,364.8) .. (279,341.2) ;
%Curve Lines [id:da12117107377619063] 
\draw [color={rgb, 255:red, 189; green, 16; blue, 224 }  ,draw opacity=1 ][line width=1.5]    (249.8,341.8) .. controls (260.6,320) and (278.6,337.37) .. (279,341.2) ;
%Curve Lines [id:da3969328243101333] 
\draw [color={rgb, 255:red, 189; green, 16; blue, 224 }  ,draw opacity=1 ][line width=1.5]    (279,341.2) .. controls (292.6,320) and (308.2,335.77) .. (309,341.2) ;
%Curve Lines [id:da9151734516210779] 
\draw [color={rgb, 255:red, 189; green, 16; blue, 224 }  ,draw opacity=1 ][line width=1.5]    (309,341.2) .. controls (319.8,320.8) and (340.2,334.17) .. (340,341.2) ;
%Shape: Circle [id:dp874124158550148] 
\draw  [fill={rgb, 255:red, 208; green, 2; blue, 27 }  ,fill opacity=1 ] (147.4,341.4) .. controls (147.4,339.74) and (148.74,338.4) .. (150.4,338.4) .. controls (152.06,338.4) and (153.4,339.74) .. (153.4,341.4) .. controls (153.4,343.06) and (152.06,344.4) .. (150.4,344.4) .. controls (148.74,344.4) and (147.4,343.06) .. (147.4,341.4) -- cycle ;
%Curve Lines [id:da6252250157567926] 
\draw [color={rgb, 255:red, 126; green, 211; blue, 33 }  ,draw opacity=1 ][line width=1.5]    (104.6,341.4) .. controls (127,320.8) and (148,341.6) .. (150.4,341.4) ;
%Curve Lines [id:da3804603790056297] 
\draw [color={rgb, 255:red, 74; green, 144; blue, 226 }  ,draw opacity=1 ][line width=1.5]    (104.6,341.4) .. controls (108.6,356) and (144.6,358.4) .. (150.4,341.4) ;
%Shape: Circle [id:dp4307918833138282] 
\draw  [fill={rgb, 255:red, 208; green, 2; blue, 27 }  ,fill opacity=1 ] (376.2,340.6) .. controls (376.2,338.94) and (377.54,337.6) .. (379.2,337.6) .. controls (380.86,337.6) and (382.2,338.94) .. (382.2,340.6) .. controls (382.2,342.26) and (380.86,343.6) .. (379.2,343.6) .. controls (377.54,343.6) and (376.2,342.26) .. (376.2,340.6) -- cycle ;
%Curve Lines [id:da437909643774573] 
\draw [color={rgb, 255:red, 189; green, 16; blue, 224 }  ,draw opacity=1 ][line width=1.5]    (340,341.2) .. controls (353.4,320.8) and (379.4,333.57) .. (379.2,340.6) ;
%Shape: Circle [id:dp45581224645031504] 
\draw  [fill={rgb, 255:red, 208; green, 2; blue, 27 }  ,fill opacity=1 ] (405,340) .. controls (405,338.34) and (406.34,337) .. (408,337) .. controls (409.66,337) and (411,338.34) .. (411,340) .. controls (411,341.66) and (409.66,343) .. (408,343) .. controls (406.34,343) and (405,341.66) .. (405,340) -- cycle ;
%Curve Lines [id:da8339383880079239] 
\draw [color={rgb, 255:red, 189; green, 16; blue, 224 }  ,draw opacity=1 ][line width=1.5]    (379.2,340.6) .. controls (388.6,324) and (408.2,332.97) .. (408,340) ;

% Text Node
\draw (170.8,246) node [anchor=north west][inner sep=0.75pt]  [font=\normalsize] [align=left] {$\displaystyle _{\beta n}$};
% Text Node
\draw (44.4,246) node [anchor=north west][inner sep=0.75pt]  [font=\normalsize] [align=left] {$\displaystyle _{1}$};
% Text Node
\draw (401.2,246) node [anchor=north west][inner sep=0.75pt]  [font=\normalsize] [align=left] {$\displaystyle _{\frac{2k-3}{2k-1} n+\beta n}$};
% Text Node
\draw (506.8,248) node [anchor=north west][inner sep=0.75pt]  [font=\normalsize] [align=left] {$\displaystyle _{n}$};
% Text Node
\draw (40.8,361.6) node [anchor=north west][inner sep=0.75pt]  [font=\footnotesize] [align=left] {$\displaystyle v_{1}$};
% Text Node
\draw (68.8,360.6) node [anchor=north west][inner sep=0.75pt]  [font=\footnotesize] [align=left] {$\displaystyle v_{i}{}_{_{2}}$};
% Text Node
\draw (96,360.6) node [anchor=north west][inner sep=0.75pt]  [font=\footnotesize] [align=left] {$\displaystyle v_{i}{}_{_{3}}$};
% Text Node
\draw (243.2,359.8) node [anchor=north west][inner sep=0.75pt]  [font=\footnotesize] [align=left] {$\displaystyle v_{s}$};
% Text Node
\draw (273.6,360.2) node [anchor=north west][inner sep=0.75pt]  [font=\footnotesize] [align=left] {$\displaystyle v_{t}$};
% Text Node
\draw (300.8,359) node [anchor=north west][inner sep=0.75pt]  [font=\footnotesize] [align=left] {$\displaystyle v_{j_{1}}$};
% Text Node
\draw (332.8,360.6) node [anchor=north west][inner sep=0.75pt]  [font=\footnotesize] [align=left] {$\displaystyle v_{j_{2}}$};
% Text Node
\draw (142.4,361.4) node [anchor=north west][inner sep=0.75pt]  [font=\footnotesize] [align=left] {$\displaystyle v_{i}{}_{_{k-1}}$};
% Text Node
\draw (118,362) node [anchor=north west][inner sep=0.75pt]   [align=left] {...};
% Text Node
\draw (371.2,359.8) node [anchor=north west][inner sep=0.75pt]  [font=\footnotesize] [align=left] {$\displaystyle v_{j_{k-3}}$};
% Text Node
\draw (354,362) node [anchor=north west][inner sep=0.75pt]   [align=left] {...};
% Text Node
\draw (402.4,359.4) node [anchor=north west][inner sep=0.75pt]  [font=\footnotesize] [align=left] {$\displaystyle v_{m}$};

\end{tikzpicture}

\caption{One possible configuration of $T^1$}
\label{fig:T1}
\end{figure}

Next, we establish the existence of $T^2$. Define 
\[
S = \left\{v_1, \dots, v_{\frac{2k-3}{2k-1}n - ((2k-4)k+1)\beta n }\right\}
\quad \text{and} \quad 
S' = \left\{v_1, \dots, v_{\frac{2k-3}{2k-1}n}\right\}.
\]
We claim that there exists a $(k-1)$-clique $v_{x_1},v_{x_2},\dots,v_{x_{k-1}}$ in $G[S]$ with $|N(v_{x_1}\dots v_{x_{k-1}}) \cap S| > (k+1)\varepsilon n$. Suppose not, then since there are at most $(k-2)\varepsilon n \binom{|S|}{k-2}$ non-clique $(k-1)$-sets in $G[S]$, it follows that
\begin{equation*}
\begin{aligned}
e(H[S']) & \leq (\varepsilon n)(k-2)\binom{|S|}{k-2}|S| + \binom{|S|}{k-1}(k+1)\varepsilon n + |S'\setminus S|\binom{|S'|}{k-1} \\
& \leq \gamma \binom{(2k-3)n/(2k-1)}{k},
\end{aligned}
\end{equation*}
and so $S'$ witnesses that $H$ is $\gamma$-extremal, a contrary to our earlier assumption. In this case, let $v_{x_1},v_{x_2},\dots,v_{x_{k-1}}$ be a $(k-1)$-clique in $G[S]$ with $|N(v_{x_1}\dots v_{x_{k-1}}) \cap S| > (k+1)\varepsilon n$. We sequentially select distinct $v_z,v_w \in S$ such that:
$\{v_{x_1},\dots, v_{x_{k-1}},v_z\}$ and $\{v_{x_1},\dots, v_{x_{k-1}},v_w\}$ are hyperedges in $H$ and $v_{x_1},\dots,v_{x_{k-1}},v_z,v_w$ forms a $(k+1)$-clique in $G$.
Next, choose $v_{y_1},\dots,v_{y_{k-3}} \in S$. We can pick an integer $s \leq \frac{2k-3}{2k-1}n + \beta n$ satisfying $\{v_z,v_w,v_{y_1},\dots, v_{y_{k-3}},v_s\}$ is a hyperedge in $H$ and $v_{x_1},\dots,v_{x_{k-1}},v_z,v_w,v_{y_1},\dots,v_{y_{k-3}},v_s$ forms a $(2k-1)$-clique in $G$.
Thus, we obtain a $B$-avoiding $T^2 \in \mathcal{T}_{k,B}$
whose vertex set is $\{v_{x_1},\dots,v_{x_{k-1}},v_{z},v_{w},v_{y_1},\dots,v_{y_{k-3}},v_s\}$ and edge set is 
\[\{v_{x_1}\dots v_{x_{k-1}}v_{z},v_{x_1}\dots v_{x_{k-1}}v_{w},v_{z}v_{w}v_{y_1}\dots v_{y_{k-3}}v_s\}.\] 
It is dominated by every set in $\mathcal{V}_2$, see Figure~\ref{fig:T2}.

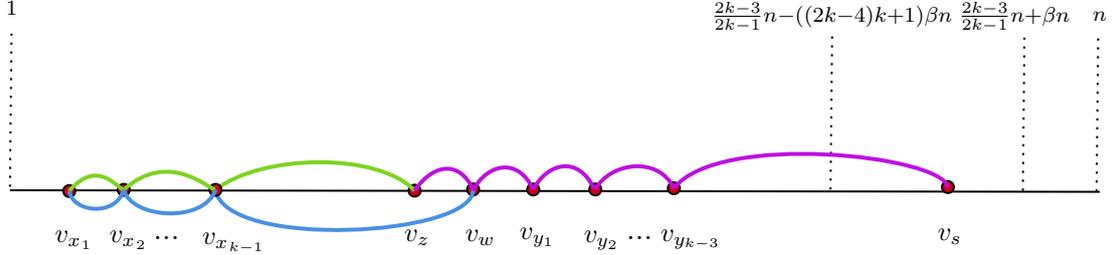
\begin{figure}[htbp]
\centering
\tikzset{every picture/.style={line width=0.75pt}} %set default line width to 0.75pt        

\begin{tikzpicture}[x=0.75pt,y=0.75pt,yscale=-1,xscale=1]
%uncomment if require: \path (0,475); %set diagram left start at 0, and has height of 475

%Straight Lines [id:da30437460358455615] 
\draw    (79,270.4) -- (623,271.2) ;
%Shape: Circle [id:dp6893894468171657] 
\draw  [fill={rgb, 255:red, 208; green, 2; blue, 27 }  ,fill opacity=1 ] (105.6,270.8) .. controls (105.6,269.14) and (106.94,267.8) .. (108.6,267.8) .. controls (110.26,267.8) and (111.6,269.14) .. (111.6,270.8) .. controls (111.6,272.46) and (110.26,273.8) .. (108.6,273.8) .. controls (106.94,273.8) and (105.6,272.46) .. (105.6,270.8) -- cycle ;
%Shape: Circle [id:dp4459992234885566] 
\draw  [fill={rgb, 255:red, 208; green, 2; blue, 27 }  ,fill opacity=1 ] (132.8,270.2) .. controls (132.8,268.54) and (134.14,267.2) .. (135.8,267.2) .. controls (137.46,267.2) and (138.8,268.54) .. (138.8,270.2) .. controls (138.8,271.86) and (137.46,273.2) .. (135.8,273.2) .. controls (134.14,273.2) and (132.8,271.86) .. (132.8,270.2) -- cycle ;
%Shape: Circle [id:dp5569055530121779] 
\draw  [fill={rgb, 255:red, 208; green, 2; blue, 27 }  ,fill opacity=1 ] (278,270.6) .. controls (278,268.94) and (279.34,267.6) .. (281,267.6) .. controls (282.66,267.6) and (284,268.94) .. (284,270.6) .. controls (284,272.26) and (282.66,273.6) .. (281,273.6) .. controls (279.34,273.6) and (278,272.26) .. (278,270.6) -- cycle ;
%Shape: Circle [id:dp7081739330611756] 
\draw  [fill={rgb, 255:red, 208; green, 2; blue, 27 }  ,fill opacity=1 ] (307.2,270) .. controls (307.2,268.34) and (308.54,267) .. (310.2,267) .. controls (311.86,267) and (313.2,268.34) .. (313.2,270) .. controls (313.2,271.66) and (311.86,273) .. (310.2,273) .. controls (308.54,273) and (307.2,271.66) .. (307.2,270) -- cycle ;
%Shape: Circle [id:dp9968363393870333] 
\draw  [fill={rgb, 255:red, 208; green, 2; blue, 27 }  ,fill opacity=1 ] (337.2,270) .. controls (337.2,268.34) and (338.54,267) .. (340.2,267) .. controls (341.86,267) and (343.2,268.34) .. (343.2,270) .. controls (343.2,271.66) and (341.86,273) .. (340.2,273) .. controls (338.54,273) and (337.2,271.66) .. (337.2,270) -- cycle ;
%Shape: Circle [id:dp15271066133710876] 
\draw  [fill={rgb, 255:red, 208; green, 2; blue, 27 }  ,fill opacity=1 ] (368.2,270) .. controls (368.2,268.34) and (369.54,267) .. (371.2,267) .. controls (372.86,267) and (374.2,268.34) .. (374.2,270) .. controls (374.2,271.66) and (372.86,273) .. (371.2,273) .. controls (369.54,273) and (368.2,271.66) .. (368.2,270) -- cycle ;
%Straight Lines [id:da7643123075942782] 
\draw [line width=0.75]  [dash pattern={on 0.84pt off 2.51pt}]  (79.8,191.6) -- (79,269.4) ;
%Straight Lines [id:da750984067151441] 
\draw  [dash pattern={on 0.84pt off 2.51pt}]  (489.4,193) -- (488.6,270.8) ;
%Straight Lines [id:da5098548997301683] 
\draw  [dash pattern={on 0.84pt off 2.51pt}]  (585.4,193.6) -- (584.6,271.4) ;
%Straight Lines [id:da37235614066449696] 
\draw  [dash pattern={on 0.84pt off 2.51pt}]  (622.2,194) -- (621.4,271.8) ;
%Curve Lines [id:da8219890635092131] 
\draw [color={rgb, 255:red, 126; green, 211; blue, 33 }  ,draw opacity=1 ][line width=1.5]    (108.6,270.8) .. controls (123,253.6) and (133.4,270.4) .. (135.8,270.2) ;
%Curve Lines [id:da8850142574929689] 
\draw [color={rgb, 255:red, 126; green, 211; blue, 33 }  ,draw opacity=1 ][line width=1.5]    (181.6,270.2) .. controls (204.6,252) and (263,251.2) .. (281,270.6) ;
%Curve Lines [id:da019946919502994542] 
\draw [color={rgb, 255:red, 74; green, 144; blue, 226 }  ,draw opacity=1 ][line width=1.5]    (108.6,270.8) .. controls (110.2,281.6) and (131.8,284) .. (135.8,270.2) ;
%Curve Lines [id:da38056071806084646] 
\draw [color={rgb, 255:red, 74; green, 144; blue, 226 }  ,draw opacity=1 ][line width=1.5]    (181.6,270.2) .. controls (186.8,300.57) and (311,293.6) .. (310.2,270) ;
%Shape: Circle [id:dp874124158550148] 
\draw  [fill={rgb, 255:red, 208; green, 2; blue, 27 }  ,fill opacity=1 ] (178.6,270.2) .. controls (178.6,268.54) and (179.94,267.2) .. (181.6,267.2) .. controls (183.26,267.2) and (184.6,268.54) .. (184.6,270.2) .. controls (184.6,271.86) and (183.26,273.2) .. (181.6,273.2) .. controls (179.94,273.2) and (178.6,271.86) .. (178.6,270.2) -- cycle ;
%Curve Lines [id:da6252250157567926] 
\draw [color={rgb, 255:red, 126; green, 211; blue, 33 }  ,draw opacity=1 ][line width=1.5]    (135.8,270.2) .. controls (158.2,249.6) and (179.2,270.4) .. (181.6,270.2) ;
%Curve Lines [id:da3804603790056297] 
\draw [color={rgb, 255:red, 74; green, 144; blue, 226 }  ,draw opacity=1 ][line width=1.5]    (135.8,270.2) .. controls (139.8,284.8) and (175.8,287.2) .. (181.6,270.2) ;
%Shape: Circle [id:dp4307918833138282] 
\draw  [fill={rgb, 255:red, 208; green, 2; blue, 27 }  ,fill opacity=1 ] (407.4,269.4) .. controls (407.4,267.74) and (408.74,266.4) .. (410.4,266.4) .. controls (412.06,266.4) and (413.4,267.74) .. (413.4,269.4) .. controls (413.4,271.06) and (412.06,272.4) .. (410.4,272.4) .. controls (408.74,272.4) and (407.4,271.06) .. (407.4,269.4) -- cycle ;
%Shape: Circle [id:dp45581224645031504] 
\draw  [fill={rgb, 255:red, 208; green, 2; blue, 27 }  ,fill opacity=1 ] (544.2,268.8) .. controls (544.2,267.14) and (545.54,265.8) .. (547.2,265.8) .. controls (548.86,265.8) and (550.2,267.14) .. (550.2,268.8) .. controls (550.2,270.46) and (548.86,271.8) .. (547.2,271.8) .. controls (545.54,271.8) and (544.2,270.46) .. (544.2,268.8) -- cycle ;
%Curve Lines [id:da36489153103426486] 
\draw [color={rgb, 255:red, 189; green, 16; blue, 224 }  ,draw opacity=1 ][line width=1.5]    (281,270.6) .. controls (291,256) and (303.8,256) .. (310.2,270) ;
%Curve Lines [id:da7939682289883297] 
\draw [color={rgb, 255:red, 189; green, 16; blue, 224 }  ,draw opacity=1 ][line width=1.5]    (310,270.6) .. controls (317.4,256.8) and (332.6,253.6) .. (340.2,270) ;
%Curve Lines [id:da979333216177757] 
\draw [color={rgb, 255:red, 189; green, 16; blue, 224 }  ,draw opacity=1 ][line width=1.5]    (340.2,270) .. controls (347.6,256.2) and (362.8,253) .. (370.4,269.4) ;
%Curve Lines [id:da728086706538586] 
\draw [color={rgb, 255:red, 189; green, 16; blue, 224 }  ,draw opacity=1 ][line width=1.5]    (371.2,270) .. controls (378.6,256.2) and (402.8,253) .. (410.4,269.4) ;
%Curve Lines [id:da48605846373505] 
\draw [color={rgb, 255:red, 189; green, 16; blue, 224 }  ,draw opacity=1 ][line width=1.5]    (410.4,269.4) .. controls (432.6,240.8) and (543.8,252.8) .. (547.2,268.8) ;

% Text Node
\draw (75.6,174) node [anchor=north west][inner sep=0.75pt]  [font=\normalsize] [align=left] {$\displaystyle _{1}$};
% Text Node
\draw (551,174) node [anchor=north west][inner sep=0.75pt]  [font=\normalsize] [align=left] {$\displaystyle _{\frac{2k-3}{2k-1} n+\beta n}$};
% Text Node
\draw (618,179) node [anchor=north west][inner sep=0.75pt]  [font=\normalsize] [align=left] {$\displaystyle _{n}$};
% Text Node
\draw (100,289.4) node [anchor=north west][inner sep=0.75pt]  [font=\footnotesize] [align=left] {$\displaystyle v_{x}{}_{_{1}}$};
% Text Node
\draw (127.2,289.4) node [anchor=north west][inner sep=0.75pt]  [font=\footnotesize] [align=left] {$\displaystyle v_{x}{}_{_{2}}$};
% Text Node
\draw (274.4,288.6) node [anchor=north west][inner sep=0.75pt]  [font=\footnotesize] [align=left] {$\displaystyle v_{z}$};
% Text Node
\draw (304.8,289) node [anchor=north west][inner sep=0.75pt]  [font=\footnotesize] [align=left] {$\displaystyle v_{w}$};
% Text Node
\draw (332,287.8) node [anchor=north west][inner sep=0.75pt]  [font=\footnotesize] [align=left] {$\displaystyle v_{y_{1}}$};
% Text Node
\draw (364,289.4) node [anchor=north west][inner sep=0.75pt]  [font=\footnotesize] [align=left] {$\displaystyle v_{y_{2}}$};
% Text Node
\draw (173.6,290.2) node [anchor=north west][inner sep=0.75pt]  [font=\footnotesize] [align=left] {$\displaystyle v_{x}{}_{_{k-1}}$};
% Text Node
\draw (149.2,292) node [anchor=north west][inner sep=0.75pt]   [align=left] {...};
% Text Node
\draw (402.4,288.6) node [anchor=north west][inner sep=0.75pt]  [font=\footnotesize] [align=left] {$\displaystyle v_{y_{k-3}}$};
% Text Node
\draw (385.2,292) node [anchor=north west][inner sep=0.75pt]   [align=left] {...};
% Text Node
\draw (540.6,288.2) node [anchor=north west][inner sep=0.75pt]  [font=\footnotesize] [align=left] {$\displaystyle v_{s}$};
% Text Node
\draw (427.6,174) node [anchor=north west][inner sep=0.75pt]  [font=\normalsize] [align=left] {$\displaystyle _{\frac{2k-3}{2k-1} n-(( 2k-4) k+1) \beta n}$};

\end{tikzpicture}

\caption{One possible configuration of $T^2$}
\label{fig:T2}
\end{figure}

Finally, we establish the existence of $T_3$. Define:
\[
\begin{aligned}
U &= \left\{v_i : i \leq (2k-1)\varepsilon n\right\}, \\
W &= \left\{v_i : i \leq  \frac{k-1}{2k-1}n - ((k-2)k+1)\beta n  \right\}, \\
W' &= \left\{v_i :  \frac{k-1}{2k-1}n - ((k-2)k+1)\beta n  \!+\! 1 \leq i \leq \frac{2k-2}{2k-1}n \right\}.
\end{aligned}
\]
Observe that $G[U]$ contains a $k$-clique $K$. Let $V(K) = \{v_{x_1},v_{x_2},\dots,v_{x_k}\}$. Next we consider two cases.

\textbf{Case 1.} Suppose there exist $v_{x_1},v_{x_2},\dots,v_{x_{k-1}}\in V(K)$ with $|N(v_{x_1}v_{x_2}\dots v_{x_{k-1}})\cap W|\geq k\varepsilon n$. 
Then we can choose $v_z \in W$ such that $\{v_{x_1},v_{x_2},\dots,v_{x_{k-1}}\} \cup \{v_z\} \in E(H)$ and $v_z$ is adjacent to all vertices $\{v_{x_1},v_{x_2},\dots,v_{x_{k-1}}\}$ in $G$. 
Now choose $v_i\in N(v_{x_1}\dots v_{x_{k-1}})$ such that $i\le \frac{2k-3}{2k-1}n + \beta n$ and $v_{x_1},\dots,v_{x_{k-1}},v_z,v_i$ forms a $(k+1)$-clique in $G$. 
Select $k-3$ vertices $v_{i_1},\dots ,v_{i_{k-3}}$ with indices $i_1\le i_2\le \cdots \le i_{k-3}\le \frac{k-1}{2k-1}n - ((k-2)k+1)\beta n$. 
Finally, we choose $v_j$ with $j \leq \frac{2k-3}{2k-1}n + \beta n$ such that $\{v_z, v_{i}, v_{i_1},\dots,v_{i_{k-3}}, v_j\} \in E(H)$ (i.e., $v_j \in N(v_z v_{i} v_{i_1}\dots v_{i_{k-3}})$), and $v_{x_1},\dots,v_{x_{k-1}},v_{i}, v_z, v_{i_1},\dots,v_{i_{k-3}}, v_j$ forms a $(2k-1)$-clique in $G$.  
Thus, we obtain a $B$-avoiding $T^3 \in \mathcal{T}_{k,B}$
whose vertex set is 
\[
\{v_{x_1},\dots,v_{x_{k-1}},v_z, v_{i}, v_{i_1},\dots,v_{i_{k-3}}, v_j\}
\]
and edge set is 
\(
\{v_{x_1}\dots v_{x_{k-1}}v_{z},v_{x_1}\dots v_{x_{k-1}}v_i,v_zv_{i}v_{i_1}\dots v_{i_{k-3}}v_j\}.
\)
It is dominated by every set in $\mathcal{V}_3$, see Figure~\ref{fig:T3-1}.
\begin{figure}
\centering

\tikzset{every picture/.style={line width=0.75pt}} %set default line width to 0.75pt        

\begin{tikzpicture}[x=0.75pt,y=0.75pt,yscale=-1,xscale=1]
%uncomment if require: \path (0,475); %set diagram left start at 0, and has height of 475

%Straight Lines [id:da30437460358455615] 
\draw    (79,270.4) -- (623,271.2) ;
%Shape: Circle [id:dp6893894468171657] 
\draw  [fill={rgb, 255:red, 208; green, 2; blue, 27 }  ,fill opacity=1 ] (105.6,270.8) .. controls (105.6,269.14) and (106.94,267.8) .. (108.6,267.8) .. controls (110.26,267.8) and (111.6,269.14) .. (111.6,270.8) .. controls (111.6,272.46) and (110.26,273.8) .. (108.6,273.8) .. controls (106.94,273.8) and (105.6,272.46) .. (105.6,270.8) -- cycle ;
%Shape: Circle [id:dp4459992234885566] 
\draw  [fill={rgb, 255:red, 208; green, 2; blue, 27 }  ,fill opacity=1 ] (132.8,270.2) .. controls (132.8,268.54) and (134.14,267.2) .. (135.8,267.2) .. controls (137.46,267.2) and (138.8,268.54) .. (138.8,270.2) .. controls (138.8,271.86) and (137.46,273.2) .. (135.8,273.2) .. controls (134.14,273.2) and (132.8,271.86) .. (132.8,270.2) -- cycle ;
%Shape: Circle [id:dp7081739330611756] 
\draw  [fill={rgb, 255:red, 208; green, 2; blue, 27 }  ,fill opacity=1 ] (307.2,270) .. controls (307.2,268.34) and (308.54,267) .. (310.2,267) .. controls (311.86,267) and (313.2,268.34) .. (313.2,270) .. controls (313.2,271.66) and (311.86,273) .. (310.2,273) .. controls (308.54,273) and (307.2,271.66) .. (307.2,270) -- cycle ;
%Shape: Circle [id:dp9968363393870333] 
\draw  [fill={rgb, 255:red, 208; green, 2; blue, 27 }  ,fill opacity=1 ] (337.2,270) .. controls (337.2,268.34) and (338.54,267) .. (340.2,267) .. controls (341.86,267) and (343.2,268.34) .. (343.2,270) .. controls (343.2,271.66) and (341.86,273) .. (340.2,273) .. controls (338.54,273) and (337.2,271.66) .. (337.2,270) -- cycle ;
%Shape: Circle [id:dp15271066133710876] 
\draw  [fill={rgb, 255:red, 208; green, 2; blue, 27 }  ,fill opacity=1 ] (368.2,270) .. controls (368.2,268.34) and (369.54,267) .. (371.2,267) .. controls (372.86,267) and (374.2,268.34) .. (374.2,270) .. controls (374.2,271.66) and (372.86,273) .. (371.2,273) .. controls (369.54,273) and (368.2,271.66) .. (368.2,270) -- cycle ;
%Straight Lines [id:da7643123075942782] 
\draw [line width=0.75]  [dash pattern={on 0.84pt off 2.51pt}]  (79.8,191.6) -- (79,269.4) ;
%Straight Lines [id:da750984067151441] 
\draw  [dash pattern={on 0.84pt off 2.51pt}]  (427.8,193) -- (427,270.8) ;
%Straight Lines [id:da5098548997301683] 
\draw  [dash pattern={on 0.84pt off 2.51pt}]  (588.6,193.6) -- (587.8,271.4) ;
%Straight Lines [id:da37235614066449696] 
\draw  [dash pattern={on 0.84pt off 2.51pt}]  (622.2,194) -- (621.4,271.8) ;
%Curve Lines [id:da8219890635092131] 
\draw [color={rgb, 255:red, 126; green, 211; blue, 33 }  ,draw opacity=1 ][line width=1.5]    (108.6,270.8) .. controls (123,253.6) and (133.4,270.4) .. (135.8,270.2) ;
%Curve Lines [id:da019946919502994542] 
\draw [color={rgb, 255:red, 74; green, 144; blue, 226 }  ,draw opacity=1 ][line width=1.5]    (181.6,270.2) .. controls (188.09,288) and (502.49,297.6) .. (547.2,269.8) ;
%Shape: Circle [id:dp874124158550148] 
\draw  [fill={rgb, 255:red, 208; green, 2; blue, 27 }  ,fill opacity=1 ] (178.6,270.2) .. controls (178.6,268.54) and (179.94,267.2) .. (181.6,267.2) .. controls (183.26,267.2) and (184.6,268.54) .. (184.6,270.2) .. controls (184.6,271.86) and (183.26,273.2) .. (181.6,273.2) .. controls (179.94,273.2) and (178.6,271.86) .. (178.6,270.2) -- cycle ;
%Curve Lines [id:da6252250157567926] 
\draw [color={rgb, 255:red, 126; green, 211; blue, 33 }  ,draw opacity=1 ][line width=1.5]    (135.8,270.2) .. controls (158.2,249.6) and (179.2,270.4) .. (181.6,270.2) ;
%Curve Lines [id:da3804603790056297] 
\draw [color={rgb, 255:red, 74; green, 144; blue, 226 }  ,draw opacity=1 ][line width=1.5]    (135.8,270.2) .. controls (139.8,284.8) and (175.8,287.2) .. (181.6,270.2) ;
%Shape: Circle [id:dp4307918833138282] 
\draw  [fill={rgb, 255:red, 208; green, 2; blue, 27 }  ,fill opacity=1 ] (407.4,269.4) .. controls (407.4,267.74) and (408.74,266.4) .. (410.4,266.4) .. controls (412.06,266.4) and (413.4,267.74) .. (413.4,269.4) .. controls (413.4,271.06) and (412.06,272.4) .. (410.4,272.4) .. controls (408.74,272.4) and (407.4,271.06) .. (407.4,269.4) -- cycle ;
%Shape: Circle [id:dp45581224645031504] 
\draw  [fill={rgb, 255:red, 208; green, 2; blue, 27 }  ,fill opacity=1 ] (544.2,269.8) .. controls (544.2,268.14) and (545.54,266.8) .. (547.2,266.8) .. controls (548.86,266.8) and (550.2,268.14) .. (550.2,269.8) .. controls (550.2,271.46) and (548.86,272.8) .. (547.2,272.8) .. controls (545.54,272.8) and (544.2,271.46) .. (544.2,269.8) -- cycle ;
%Curve Lines [id:da7939682289883297] 
\draw [color={rgb, 255:red, 189; green, 16; blue, 224 }  ,draw opacity=1 ][line width=1.5]    (310,270.6) .. controls (317.4,256.8) and (332.6,253.6) .. (340.2,270) ;
%Curve Lines [id:da979333216177757] 
\draw [color={rgb, 255:red, 189; green, 16; blue, 224 }  ,draw opacity=1 ][line width=1.5]    (340.2,270) .. controls (347.6,256.2) and (362.8,253) .. (370.4,269.4) ;
%Curve Lines [id:da728086706538586] 
\draw [color={rgb, 255:red, 189; green, 16; blue, 224 }  ,draw opacity=1 ][line width=1.5]    (371.2,270) .. controls (378.6,256.2) and (402.8,253) .. (410.4,269.4) ;
%Curve Lines [id:da48605846373505] 
\draw [color={rgb, 255:red, 189; green, 16; blue, 224 }  ,draw opacity=1 ][line width=1.5]    (410.4,269.4) .. controls (424.89,250.4) and (543.8,253.8) .. (547.2,269.8) ;
%Curve Lines [id:da8844280761336328] 
\draw [color={rgb, 255:red, 126; green, 211; blue, 33 }  ,draw opacity=1 ][line width=1.5]    (181.6,270.2) .. controls (200.6,253.6) and (282.49,254.4) .. (310.2,270) ;
%Straight Lines [id:da4036908465423492] 
\draw [line width=0.75]  [dash pattern={on 0.84pt off 2.51pt}]  (231.8,192.8) -- (231,270.6) ;
%Straight Lines [id:da5402716814156074] 
\draw  [dash pattern={on 0.84pt off 2.51pt}]  (82.49,217.6) -- (148.09,217.6) ;
%Straight Lines [id:da17348194791533733] 
\draw  [dash pattern={on 0.84pt off 2.51pt}]  (158.49,217.6) -- (231.29,217.6) ;
%Straight Lines [id:da6238995268398184] 
\draw  [dash pattern={on 0.84pt off 2.51pt}]  (80.89,159.2) -- (224.09,159.2) ;
%Straight Lines [id:da4721932179423445] 
\draw  [dash pattern={on 0.84pt off 2.51pt}]  (240.89,160) -- (431,160) ;
%Shape: Circle [id:dp13347207676633965] 
\draw  [fill={rgb, 255:red, 208; green, 2; blue, 27 }  ,fill opacity=1 ] (564.2,269.8) .. controls (564.2,268.14) and (565.54,266.8) .. (567.2,266.8) .. controls (568.86,266.8) and (570.2,268.14) .. (570.2,269.8) .. controls (570.2,271.46) and (568.86,272.8) .. (567.2,272.8) .. controls (565.54,272.8) and (564.2,271.46) .. (564.2,269.8) -- cycle ;
%Curve Lines [id:da12473502865610109] 
\draw [color={rgb, 255:red, 189; green, 16; blue, 224 }  ,draw opacity=1 ][line width=1.5]    (547.2,269.8) .. controls (552.89,260) and (560.89,260.8) .. (567.2,269.8) ;
%Curve Lines [id:da5985882545208698] 
\draw [color={rgb, 255:red, 74; green, 144; blue, 226 }  ,draw opacity=1 ][line width=1.5]    (108.6,270.8) .. controls (113.4,284.59) and (131,282.99) .. (135.8,270.2) ;

% Text Node
\draw (75.6,174) node [anchor=north west][inner sep=0.75pt]  [font=\normalsize] [align=left] {$\displaystyle _{1}$};
% Text Node
\draw (552.4,174) node [anchor=north west][inner sep=0.75pt]  [font=\normalsize] [align=left] {$\displaystyle _{\frac{2k-3}{2k-1} n+\beta n}$};
% Text Node
\draw (618,179) node [anchor=north west][inner sep=0.75pt]  [font=\normalsize] [align=left] {$\displaystyle _{n}$};
% Text Node
\draw (100,289.4) node [anchor=north west][inner sep=0.75pt]  [font=\footnotesize] [align=left] {$\displaystyle v_{x}{}_{_{1}}$};
% Text Node
\draw (127.2,289.4) node [anchor=north west][inner sep=0.75pt]  [font=\footnotesize] [align=left] {$\displaystyle v_{x}{}_{_{2}}$};
% Text Node
\draw (304.8,289) node [anchor=north west][inner sep=0.75pt]  [font=\footnotesize] [align=left] {$\displaystyle v_{z}$};
% Text Node
\draw (332,287.8) node [anchor=north west][inner sep=0.75pt]  [font=\footnotesize] [align=left] {$\displaystyle v_{i_{1}}$};
% Text Node
\draw (364,289.4) node [anchor=north west][inner sep=0.75pt]  [font=\footnotesize] [align=left] {$\displaystyle v_{i_{2}}$};
% Text Node
\draw (173.6,290.2) node [anchor=north west][inner sep=0.75pt]  [font=\footnotesize] [align=left] {$\displaystyle v_{x}{}_{_{k-1}}$};
% Text Node
\draw (149.2,294) node [anchor=north west][inner sep=0.75pt]   [align=left] {...};
% Text Node
\draw (402.4,288.6) node [anchor=north west][inner sep=0.75pt]  [font=\footnotesize] [align=left] {$\displaystyle v_{i_{k-3}}$};
% Text Node
\draw (385.2,294) node [anchor=north west][inner sep=0.75pt]   [align=left] {...};
% Text Node
\draw (540.6,288.2) node [anchor=north west][inner sep=0.75pt]  [font=\footnotesize] [align=left] {$\displaystyle v_{i}$};
% Text Node
\draw (366,174) node [anchor=north west][inner sep=0.75pt]  [font=\normalsize] [align=left] {$\displaystyle _{\frac{k-1}{2k-1} n-(( k-2) k+1) \beta n}$};
% Text Node
\draw (210,174) node [anchor=north west][inner sep=0.75pt]  [font=\normalsize] [align=left] {$\displaystyle _{( 2k-1) \varepsilon n}$};
% Text Node
\draw (146,210) node [anchor=north west][inner sep=0.75pt]  [font=\large] [align=left] {$\displaystyle _{U}$};
% Text Node
\draw (224.4,151) node [anchor=north west][inner sep=0.75pt]  [font=\large] [align=left] {$\displaystyle _{W}$};
% Text Node
\draw (561.4,289) node [anchor=north west][inner sep=0.75pt]  [font=\footnotesize] [align=left] {$\displaystyle v_{j}$};

\end{tikzpicture}

\caption{One possible configuration of $T^3$ under Case 1}
\label{fig:T3-1}
\end{figure}
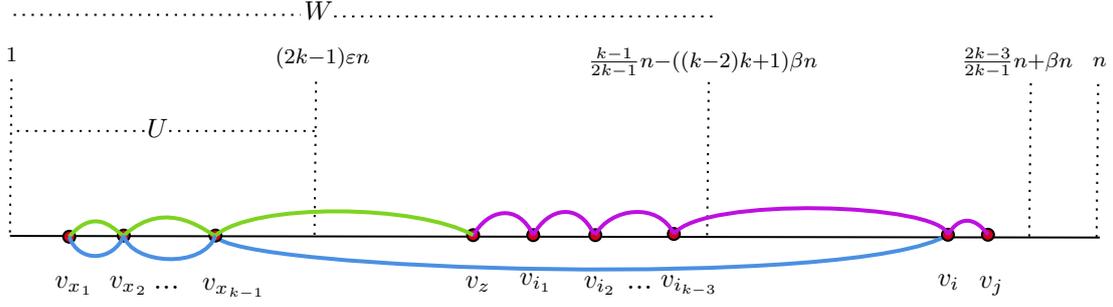

\textbf{Case 2.} For all ($k-1$)-subset $\{v_{x_{i_1}},v_{x_{i_2}},\dots,v_{x_{i_{k-1}}}\}$ of $V(K)$, we have 
\[
|N(v_{x_{i_1}}v_{x_{i_2}}\dots v_{x_{i_{k-1}}})\cap W|< k\varepsilon n.
\]
Then as $\delta(H) \geq \frac{2}{2k-1}n - \alpha n$, we have
\[
\begin{aligned}
|N(v_{x_{i_1}}v_{x_{i_2}}\dots v_{x_{i_{k-1}}})\cap W'| &\geq \delta(H) - k\varepsilon n - \frac{n}{2k-1}\\ 
&\geq \frac{n}{2k-1} - \alpha n - k\varepsilon n.
\end{aligned}
\]
Since $|W'| = \frac{k-1}{2k-1}n + ((k-2)k+1)\beta n$, we may assume that  
\[
\left|N(v_{x_1}\dots v_{x_{k-1}}) \cap N(v_{x_2}\dots v_{x_k}) \cap W'\right| \geq (k+1)\varepsilon n.
\]  
Choose $v_w \in W'$ such that $v_w \in N(v_{x_1}\dots v_{x_{k-1}}) \cap N(v_{x_2}\dots v_{x_k})$ and $v_w$ is adjacent to each of $v_{x_1},\dots,v_{x_k}$ in $G$. Select $k-3$ vertices $v_{i_1},\dots ,v_{i_{k-3}}$ with indices $i_1\le i_2\le \cdots \le i_{k-3}\le  \frac{k-1}{2k-1}n - (k(k-2)+1)\beta n$.  
Choose $v_i$ with $i \leq \frac{2k-3}{2k-1}n + \alpha n + (2k-1)\varepsilon n$ such that $v_i \in N(v_{x_1} v_{x_k} v_{i_1}\dots v_{i_{k-3}})$ and $v_{x_1},\dots,v_{x_k}, v_w, v_{i_1},\dots,v_{i_{k-3}}, v_i$ is a $(2k-1)$-clique in $G$.  
Thus, we obtain a $B$-avoiding $T^3 \in \mathcal{T}_{k,B}$
whose vertex set is $\{v_{x_1},\dots,v_{x_k}, v_w, v_{i_1},\dots,v_{i_{k-3}}, v_i\}$ and edge set is 
\[
\{v_{x_1}\dots v_{x_{k-1}}v_w,v_{x_2}\dots v_{x_{k-1}}v_{x_k}v_w,v_{x_1}v_{x_k}v_{i_1}\dots v_{i_{k-3}}v_i\}.
\] 
It is dominated by every set in $\mathcal{V}_3$, see Figure~\ref{fig:T3-2}.
\\[4pt]
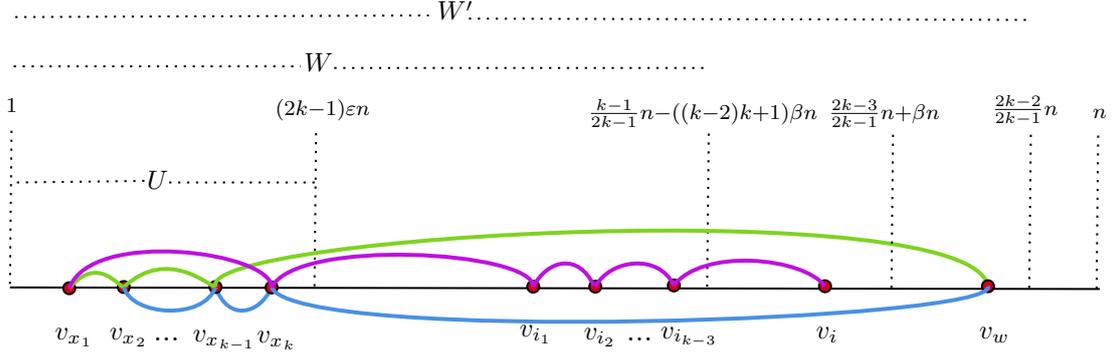
\begin{figure}
\centering

\tikzset{every picture/.style={line width=0.75pt}} %set default line width to 0.75pt        

\begin{tikzpicture}[x=0.75pt,y=0.75pt,yscale=-1,xscale=1]
%uncomment if require: \path (0,475); %set diagram left start at 0, and has height of 475

%Straight Lines [id:da30437460358455615] 
\draw    (79,270.4) -- (623,271.2) ;
%Shape: Circle [id:dp6893894468171657] 
\draw  [fill={rgb, 255:red, 208; green, 2; blue, 27 }  ,fill opacity=1 ] (105.6,270.8) .. controls (105.6,269.14) and (106.94,267.8) .. (108.6,267.8) .. controls (110.26,267.8) and (111.6,269.14) .. (111.6,270.8) .. controls (111.6,272.46) and (110.26,273.8) .. (108.6,273.8) .. controls (106.94,273.8) and (105.6,272.46) .. (105.6,270.8) -- cycle ;
%Shape: Circle [id:dp4459992234885566] 
\draw  [fill={rgb, 255:red, 208; green, 2; blue, 27 }  ,fill opacity=1 ] (132.8,270.2) .. controls (132.8,268.54) and (134.14,267.2) .. (135.8,267.2) .. controls (137.46,267.2) and (138.8,268.54) .. (138.8,270.2) .. controls (138.8,271.86) and (137.46,273.2) .. (135.8,273.2) .. controls (134.14,273.2) and (132.8,271.86) .. (132.8,270.2) -- cycle ;
%Shape: Circle [id:dp9968363393870333] 
\draw  [fill={rgb, 255:red, 208; green, 2; blue, 27 }  ,fill opacity=1 ] (337.2,270) .. controls (337.2,268.34) and (338.54,267) .. (340.2,267) .. controls (341.86,267) and (343.2,268.34) .. (343.2,270) .. controls (343.2,271.66) and (341.86,273) .. (340.2,273) .. controls (338.54,273) and (337.2,271.66) .. (337.2,270) -- cycle ;
%Shape: Circle [id:dp15271066133710876] 
\draw  [fill={rgb, 255:red, 208; green, 2; blue, 27 }  ,fill opacity=1 ] (368.2,270) .. controls (368.2,268.34) and (369.54,267) .. (371.2,267) .. controls (372.86,267) and (374.2,268.34) .. (374.2,270) .. controls (374.2,271.66) and (372.86,273) .. (371.2,273) .. controls (369.54,273) and (368.2,271.66) .. (368.2,270) -- cycle ;
%Straight Lines [id:da7643123075942782] 
\draw [line width=0.75]  [dash pattern={on 0.84pt off 2.51pt}]  (79.8,191.6) -- (79,269.4) ;
%Straight Lines [id:da750984067151441] 
\draw  [dash pattern={on 0.84pt off 2.51pt}]  (427.8,193) -- (427,270.8) ;
%Straight Lines [id:da5098548997301683] 
\draw  [dash pattern={on 0.84pt off 2.51pt}]  (588.6,193.6) -- (587.8,271.4) ;
%Straight Lines [id:da37235614066449696] 
\draw  [dash pattern={on 0.84pt off 2.51pt}]  (622.2,194) -- (621.4,271.8) ;
%Curve Lines [id:da8219890635092131] 
\draw [color={rgb, 255:red, 126; green, 211; blue, 33 }  ,draw opacity=1 ][line width=1.5]    (108.6,270.8) .. controls (123,253.6) and (133.4,270.4) .. (135.8,270.2) ;
%Curve Lines [id:da019946919502994542] 
\draw [color={rgb, 255:red, 74; green, 144; blue, 226 }  ,draw opacity=1 ][line width=1.5]    (181.6,270.2) .. controls (188.09,288) and (203.67,283.83) .. (209.6,270.4) ;
%Shape: Circle [id:dp874124158550148] 
\draw  [fill={rgb, 255:red, 208; green, 2; blue, 27 }  ,fill opacity=1 ] (178.6,270.2) .. controls (178.6,268.54) and (179.94,267.2) .. (181.6,267.2) .. controls (183.26,267.2) and (184.6,268.54) .. (184.6,270.2) .. controls (184.6,271.86) and (183.26,273.2) .. (181.6,273.2) .. controls (179.94,273.2) and (178.6,271.86) .. (178.6,270.2) -- cycle ;
%Curve Lines [id:da6252250157567926] 
\draw [color={rgb, 255:red, 126; green, 211; blue, 33 }  ,draw opacity=1 ][line width=1.5]    (135.8,270.2) .. controls (158.2,249.6) and (179.2,270.4) .. (181.6,270.2) ;
%Curve Lines [id:da3804603790056297] 
\draw [color={rgb, 255:red, 74; green, 144; blue, 226 }  ,draw opacity=1 ][line width=1.5]    (135.8,270.2) .. controls (139.8,284.8) and (175.8,287.2) .. (181.6,270.2) ;
%Shape: Circle [id:dp4307918833138282] 
\draw  [fill={rgb, 255:red, 208; green, 2; blue, 27 }  ,fill opacity=1 ] (407.4,269.4) .. controls (407.4,267.74) and (408.74,266.4) .. (410.4,266.4) .. controls (412.06,266.4) and (413.4,267.74) .. (413.4,269.4) .. controls (413.4,271.06) and (412.06,272.4) .. (410.4,272.4) .. controls (408.74,272.4) and (407.4,271.06) .. (407.4,269.4) -- cycle ;
%Shape: Circle [id:dp45581224645031504] 
\draw  [fill={rgb, 255:red, 208; green, 2; blue, 27 }  ,fill opacity=1 ] (482.87,269.8) .. controls (482.87,268.14) and (484.21,266.8) .. (485.87,266.8) .. controls (487.52,266.8) and (488.87,268.14) .. (488.87,269.8) .. controls (488.87,271.46) and (487.52,272.8) .. (485.87,272.8) .. controls (484.21,272.8) and (482.87,271.46) .. (482.87,269.8) -- cycle ;
%Curve Lines [id:da979333216177757] 
\draw [color={rgb, 255:red, 189; green, 16; blue, 224 }  ,draw opacity=1 ][line width=1.5]    (340.2,270) .. controls (347.6,256.2) and (362.8,253) .. (370.4,269.4) ;
%Curve Lines [id:da728086706538586] 
\draw [color={rgb, 255:red, 189; green, 16; blue, 224 }  ,draw opacity=1 ][line width=1.5]    (371.2,270) .. controls (378.6,256.2) and (402.8,253) .. (410.4,269.4) ;
%Curve Lines [id:da48605846373505] 
\draw [color={rgb, 255:red, 189; green, 16; blue, 224 }  ,draw opacity=1 ][line width=1.5]    (410.4,269.4) .. controls (424.89,250.4) and (475,255.16) .. (485.87,269.8) ;
%Curve Lines [id:da8844280761336328] 
\draw [color={rgb, 255:red, 126; green, 211; blue, 33 }  ,draw opacity=1 ][line width=1.5]    (181.6,270.2) .. controls (171.67,243.83) and (559,222.5) .. (567.2,269.8) ;
%Straight Lines [id:da4036908465423492] 
\draw [line width=0.75]  [dash pattern={on 0.84pt off 2.51pt}]  (231.8,192.8) -- (231,270.6) ;
%Straight Lines [id:da5402716814156074] 
\draw  [dash pattern={on 0.84pt off 2.51pt}]  (82.49,217.6) -- (148.09,216.8) ;
%Straight Lines [id:da17348194791533733] 
\draw  [dash pattern={on 0.84pt off 2.51pt}]  (158.49,217.6) -- (231.29,217.6) ;
%Straight Lines [id:da6238995268398184] 
\draw  [dash pattern={on 0.84pt off 2.51pt}]  (80.89,159.2) -- (224.09,159.2) ;
%Straight Lines [id:da4721932179423445] 
\draw  [dash pattern={on 0.84pt off 2.51pt}]  (240.89,160) -- (428,160) ;
%Straight Lines [id:da7662854867637965] 
\draw  [dash pattern={on 0.84pt off 2.51pt}]  (81.69,134.4) -- (288.09,133.6) ;
%Straight Lines [id:da7428409034600381] 
\draw  [dash pattern={on 0.84pt off 2.51pt}]  (308.09,135.2) -- (589,135.83) ;
%Shape: Circle [id:dp13347207676633965] 
\draw  [fill={rgb, 255:red, 208; green, 2; blue, 27 }  ,fill opacity=1 ] (564.2,269.8) .. controls (564.2,268.14) and (565.54,266.8) .. (567.2,266.8) .. controls (568.86,266.8) and (570.2,268.14) .. (570.2,269.8) .. controls (570.2,271.46) and (568.86,272.8) .. (567.2,272.8) .. controls (565.54,272.8) and (564.2,271.46) .. (564.2,269.8) -- cycle ;
%Straight Lines [id:da05509557239699847] 
\draw  [dash pattern={on 0.84pt off 2.51pt}]  (519.8,193.6) -- (519,271.4) ;
%Shape: Circle [id:dp4596821997978151] 
\draw  [fill={rgb, 255:red, 208; green, 2; blue, 27 }  ,fill opacity=1 ] (206.6,270.4) .. controls (206.6,268.74) and (207.94,267.4) .. (209.6,267.4) .. controls (211.26,267.4) and (212.6,268.74) .. (212.6,270.4) .. controls (212.6,272.06) and (211.26,273.4) .. (209.6,273.4) .. controls (207.94,273.4) and (206.6,272.06) .. (206.6,270.4) -- cycle ;
%Curve Lines [id:da3695923917306434] 
\draw [color={rgb, 255:red, 74; green, 144; blue, 226 }  ,draw opacity=1 ][line width=1.5]    (209.6,270.4) .. controls (222.33,299.16) and (573,287.16) .. (567.2,269.8) ;
%Curve Lines [id:da07455512401735476] 
\draw [color={rgb, 255:red, 189; green, 16; blue, 224 }  ,draw opacity=1 ][line width=1.5]    (108.6,270.8) .. controls (119.67,240.5) and (209.67,252.5) .. (209.6,270.4) ;
%Curve Lines [id:da3623032927006635] 
\draw [color={rgb, 255:red, 189; green, 16; blue, 224 }  ,draw opacity=1 ][line width=1.5]    (209.6,270.4) .. controls (222.33,245.16) and (340.27,252.1) .. (340.2,270) ;

% Text Node
\draw (75.6,174) node [anchor=north west][inner sep=0.75pt]  [font=\normalsize] [align=left] {$\displaystyle _{1}$};
% Text Node
\draw (567.6,172) node [anchor=north west][inner sep=0.75pt]  [font=\normalsize] [align=left] {$\displaystyle _{\frac{2k-2}{2k-1} n}$};
% Text Node
\draw (618,179) node [anchor=north west][inner sep=0.75pt]  [font=\normalsize] [align=left] {$\displaystyle _{n}$};
% Text Node
\draw (100,289.4) node [anchor=north west][inner sep=0.75pt]  [font=\footnotesize] [align=left] {$\displaystyle v_{x}{}_{_{1}}$};
% Text Node
\draw (127.2,289.4) node [anchor=north west][inner sep=0.75pt]  [font=\footnotesize] [align=left] {$\displaystyle v_{x}{}_{_{2}}$};
% Text Node
\draw (332,287.8) node [anchor=north west][inner sep=0.75pt]  [font=\footnotesize] [align=left] {$\displaystyle v_{i_{1}}$};
% Text Node
\draw (364,289.4) node [anchor=north west][inner sep=0.75pt]  [font=\footnotesize] [align=left] {$\displaystyle v_{i_{2}}$};
% Text Node
\draw (168.8,290.2) node [anchor=north west][inner sep=0.75pt]  [font=\footnotesize] [align=left] {$\displaystyle v_{x}{}_{_{k-1}}$};
% Text Node
\draw (149.2,294) node [anchor=north west][inner sep=0.75pt]   [align=left] {...};
% Text Node
\draw (402.4,288.6) node [anchor=north west][inner sep=0.75pt]  [font=\footnotesize] [align=left] {$\displaystyle v_{i_{k-3}}$};
% Text Node
\draw (385.2,294) node [anchor=north west][inner sep=0.75pt]   [align=left] {...};
% Text Node
\draw (479.93,288.87) node [anchor=north west][inner sep=0.75pt]  [font=\footnotesize] [align=left] {$\displaystyle v_{i}$};
% Text Node
\draw (366,174) node [anchor=north west][inner sep=0.75pt]  [font=\normalsize] [align=left] {$\displaystyle _{\frac{k-1}{2k-1} n-(( k-2) k+1) \beta n}$};
% Text Node
\draw (210,174) node [anchor=north west][inner sep=0.75pt]  [font=\normalsize] [align=left] {$\displaystyle _{( 2k-1) \varepsilon n}$};
% Text Node
\draw (146,210) node [anchor=north west][inner sep=0.75pt]  [font=\large] [align=left] {$\displaystyle _{U}$};
% Text Node
\draw (224.4,151) node [anchor=north west][inner sep=0.75pt]  [font=\large] [align=left] {$\displaystyle _{W}$};
% Text Node
\draw (290.8,125) node [anchor=north west][inner sep=0.75pt]  [font=\large] [align=left] {$\displaystyle _{W'}$};
% Text Node
\draw (561.4,289) node [anchor=north west][inner sep=0.75pt]  [font=\footnotesize] [align=left] {$\displaystyle v_{w}$};
% Text Node
\draw (486,174) node [anchor=north west][inner sep=0.75pt]  [font=\normalsize] [align=left] {$\displaystyle _{\frac{2k-3}{2k-1} n+\beta n}$};
% Text Node
\draw (200.8,291) node [anchor=north west][inner sep=0.75pt]  [font=\footnotesize] [align=left] {$\displaystyle v_{x}{}_{_{k}}$};

\end{tikzpicture}

\caption{One possible configuration of $T^3$ under Case 2}
\label{fig:T3-2}
\end{figure}

\noindent
\textbf{Divisibility reduction.} We show it suffices to prove the lemma when $(2k-1) \mid n$. To see this, form augmented blow-ups $H'$ and $B'$ of $H$ and $B$ in the following way. Both $H'$ and $B'$ have vertex set $V'$, which is a set formed from $V$ by replacing each vertex $u \in V$ by $2k-1$ copies $u^1, \dots, u^{2k-1}$. So $|V'| = (2k-1)n$. 
The edges of $H'$ are the $k$-sets $\{u_1^{i_1},u_2^{i_2},\dots, u_k^{i_k}\}$ for which $\{u_1,u_2,\dots, u_k\} \in E(H)$ and $i_1,i_2,\dots,i_k \in [2k-1]$, and also all $k$-sets of the form $\{u_1^{i_1},u_1^{i_2},u_2^{i_3},\dots, u_{k-1}^{i_{k}}\}$ with $u_{1}, u_{2},\dots, u_{k-1} \in V$ (we may have $u_i=u_j$ for $i,j\in [k-1]$) and $i_1,i_2,\dots, i_k \in [2k-1]$ with $i_1 \neq i_2$. 
The edges of $B'$ are all the pairs $u^iv^j$ with $uv \in E(B)$ and $i,j \in [2k-1]$, and also all pairs $u^iu^j$ with $u \in V$ and $i, j \in [2k-1]$ with $i \neq j$. This definition implies that 
\[
\delta(H') \geq (2k-1)\delta(H) \geq (2/(2k-1)-\alpha) ((2k-1)n)
\] 
and 
\[
\Delta(B') = (2k-1)\Delta(B) + (2k-2) \leq 2\eps ((2k-1)n).
\] 
So, the lemma in the case that $(2k-1)\mid n$ implies that either $H'$ admits a perfect $B'$-avoiding fractional $T_k$-tiling $\omega'$ or $H'$ is $\gamma$-extremal.

If $H'$ admits a perfect $B'$-avoiding fractional $T_k$-tiling $\omega'$, our choice of $B'$ implies that for each $u \in V$ and $i \neq j$ there is no $T\in \mathcal{T}_k(H')$ with $\omega'(T) > 0$ which contains both $u_i$ and $u_j$. In particular, every $B'$-avoiding copy $T$ of $T_k$ in $H'$ corresponds directly to a copy $T^*$ of $T_k$ in $H$. 
Setting $\omega(T^*)$ to be the average of $\omega'(T)$ over every copy $T$ of $T_k$ in $H'$ that corresponds to $T^*$ therefore gives a perfect $B$-avoiding fractional $T_k$-tiling $\omega$ in $H$. On the other hand, if $H'$ is $\gamma$-extremal then there is a set $S' \subseteq V'$ of size $(2k-3)n$ with $d(H'[S']) \leq \gamma$. Let $S^*$ be the set of all vertices $u \in V$ for which $u_i \in S'$ for some $i \in [2k-1]$. We then have $|S'|/(2k-1) \leq |S^*| \leq |S'|$ and $e(H[S^*]) \leq e(H'[S'])$, so 
$$d(H[S^*]) = \frac{e(H[S^*])}{\binom{|S^*|}{k}} \leq \frac{e(H'[S'])}{\frac{1}{(2k-1)^k}\binom{|S'|}{k}} = (2k-1)^k d(H'[S']) \leq (2k-1)^k\gamma.$$ 
Since $|S^*| \geq (2k-3)n/(2k-1)$, by an averaging argument there is a set $S \subseteq S^*$ of size $\lfloor (2k-3)n/(2k-1) \rfloor$ with $d(H[S]) \leq d(H[S^*])$, and $S$ witnesses that $H$ is $(2k-1)^k\gamma$-extremal. 
\end{proof}

\section{The extremal case} \label{sec:extremal}
In this section, we focus on establishing Lemma~\ref{lem:extremal}, which asserts that Theorem~\ref{thm:main} holds in the extremal scenario. To achieve this, we complete the required $T_k$-tiling by locating a perfect matching in a suitably dense auxiliary $(2k-1)$-partite $(2k-1)$-graph. This step relies on a result by Daykin and H\"aggkvist~\cite{DH}.

\begin{theorem}[Daykin and H\"aggkvist~\cite{DH}] \label{5partitematching}
Let $J$ be a $k$-partite $k$-graph whose vertex classes each have size $n$. If every vertex of $J$ is contained in at least $(k-1)n^{k-1}/k$ edges of $J$, then $J$ admits a perfect matching.
\end{theorem}

\begin{corollary} \label{auxiliarymatching}
Suppose $k\in\N$ and choose $1/n \ll \beta \ll \frac{1}{k}$. Let $A$ and $B$ be disjoint sets of size $(2k-3)n$ and $2n$ respectively, and let $J$ be a $(2k-1)$-graph on~$(2k-1)n$ vertices with vertex set $V := A \cup B$ such that every edge in $J$ contains $2k-3$ vertices of $A$ and $2$ vertices of $B$. If $e(J) \geq (1-\beta) \binom{(2k-3)n}{2k-3}\binom{2n}{2}$ and every vertex of $J$ is contained in at least $n^{2k-2}/(2k-1)^{(k+1)^2}$ edges of $J$, then $J$ admits a perfect matching.  
\end{corollary}

We are now ready to present the proof of Lemma~\ref{lem:extremal}. Our approach builds on the work of Bowtell, Kathapurkar, Morrison and Mycroft~\cite{BKMM}, who proved the lemma for the case $k=3$. Here, we adapt and generalize their method to arbitrary $k$-graphs.

\begin{proof}[\bf Proof of Lemma~\ref{lem:extremal}]
Suppose that $1/n \ll \gamma \ll \gamma' \ll \beta \ll \tfrac{1}{k}$ and that $(2k-1) \mid n$. Let $H$ be a $k$-graph on $n$ vertices with minimum codegree $\delta(H) \ge \frac{2n}{2k-1}$, and assume that $H$ is $\gamma$-extremal. That is, there exists a subset $S \subseteq V(H)$ of size $|S| = \frac{2k-3}{2k-1}n$ such that the density $d(H[S]) \le \gamma$.
We say a $(k-1)$-set $x_1 \dots x_{k-1}\subseteq S$ is \emph{bad} if $|N(x_1 \dots x_{k-1}) \cap S| > \sqrt{\gamma}n$, and \emph{good} if $|N(x_1 \dots x_{k-1}) \cap S| \le \sqrt{\gamma}n$. Sets not entirely contained in $S$ are considered neither good nor bad.
There are at most $\sqrt{\gamma} n^{k-1}$ bad $(k-1)$-sets. Otherwise, they would contribute more than $\frac{1}{k} \cdot \sqrt{\gamma} n^{k-1} \cdot \sqrt{\gamma} n = \frac{\gamma}{k} n^k \ge \gamma \binom{n}{k}$ edges in $H[S]$, which contradicts the assumption $d(H[S]) \le \gamma$.

Let $X$ be the set of vertices in $V(H) \setminus S$ that appear in fewer than $n^{k-1}/(2k-1)^k$ edges containing exactly $k-1$ vertices from $S$. Since each good $(k-1)$-set has at least $\frac{2n}{2k-1} - \sqrt{\gamma}n$ neighbors in $V(H) \setminus S$, and the total number of bad $(k-1)$-sets is at most $\sqrt{\gamma} n^{k-1}$, the total number of $k$-sets with exactly $k-1$ vertices in $S$ that are not edges is at most
$\sqrt{\gamma} n^{k-1} \cdot \frac{2n}{2k-1} + \binom{\frac{2k-3}{2k-1}n}{k-1} \cdot \sqrt{\gamma}n \le \sqrt{\gamma} n^k$.
Each vertex in $X$ is missing at least $\binom{\frac{2k-3}{2k-1}n}{k-1} - \frac{n^{k-1}}{(2k-1)^k} \ge \frac{n^{k-1}}{(k-1)^k}$ such $k$-sets. Therefore, we conclude that $|X| \le (k-1)^k \sqrt{\gamma} n$.
Define $A := S \cup X$ and $B := V(H) \setminus A$. Then we have
$\frac{2k-3}{2k-1}n \le |A| = \frac{2k-3}{2k-1}n + |X| \le \left( \frac{2k-3}{2k-1} + (k-1)^k \sqrt{\gamma} \right)n$
and
$\left( \frac{2}{2k-1} - (k-1)^k \sqrt{\gamma} \right)n \le |B| = \frac{2n}{2k-1} - |X| \le \frac{2n}{2k-1}$.
Moreover, every good $(k-1)$-set lies entirely in $A$ and has at least $\frac{2n}{2k-1} - ((k-1)^k + 1)\sqrt{\gamma}n$ neighbors in $B$.

We claim that there exists a matching $M$ in $H[A]$ of size $|X|$, where each edge contains a good $(k-1)$-set.
To prove this, consider a largest such matching $M$ in $H[A]$ with $|M| \le |X|$. Suppose for contradiction that $|M| < |X|$. Then we have $|V(M)| \le k|M| < k|X| \le k(k-1)^k \sqrt{\gamma} n$.
This allows us to choose $k$ pairwise disjoint good $(k-1)$-sets $Q_1, \dots, Q_k$ disjoint from $V(M)$.
% It follows from $\delta(H)\ge \frac{2n}{2k-1}$ that each $Q_i$ has at least $|X|$ neighbors in $A$.
By the minimum codegree condition, each $Q_i$ has at least $|X|$ neighbors in $A$. Since $M$ is maximal, all these neighbors must already be covered by $M$.
However, $\sum_{i=1}^{k} |N(Q_i) \cap A| \ge k|X| > |V(M)|$. So by the pigeonhole principle, there exist distinct vertices $x, y \in V(M)$ and distinct indices  $i, j\in [k]$ such that $x \in N(Q_i)$ and $y \in N(Q_j)$ and $x, y$ lie in the same edge $e \in M$.
Replacing $e$ with the edges $Q_i \cup \{x\}$ and $Q_j \cup \{y\}$ gives a larger matching, contradicting the maximality of $M$.
Therefore, $|M| = |X|$, as claimed.

We now use the matching $M$ to construct a collection of $|X|$ pairwise vertex-disjoint copies of $T_k$ in $H$, where each copy includes $2k - 2$ vertices from $A$ and one vertex from $B$.
To do this, consider an edge $u_1 \dots u_{k-1} w\in M$, where $u_1 \dots u_{k-1}$ form a good $(k - 1)$-set. Choose vertices $z_1, \dots, z_{k - 2}\in A \setminus V(M)$ that have not been used in any previously selected copy of $T_k$.
If the $(k-1)$-set $z_1\dots z_{k - 2}w$ has at least $\gamma' n$ neighbors in $B$, then the two $(k - 1)$-sets $u_1 \dots u_{k - 1}$ and $wz_1 \dots z_{k - 2}$ share at least 
$\gamma' n - ((k - 1)^k + 1) \sqrt{\gamma} n \ge (2k - 1) |X| + |V(M)|$ 
common neighbors in $B$. Therefore, we can select a previously unused vertex $y \in (B \setminus V(M))\cap N(u_1 \dots u_{k - 1})\cap N(w z_1 \dots z_{k - 2})$, thus forming the desired copy of $T_k$ with edges $u_1 \dots u_{k - 1} y$, $u_1 \dots u_{k - 1} w$, and $z_1 \dots z_{k - 2} w y$.
If not, then the $(k - 1)$-set $z_1 \dots z_{k - 2} w$ must have at least $2n / (2k - 1) - \gamma' n$ neighbors in $A$. In this scenario, we can choose $k-1$ vertices $a_1, \dots, a_{k - 1}\in (A \setminus V(M))\cap N(z_1 \dots z_{k - 2} w)$, which have not been used yet, and altogether form a good $(k - 1)$-set. Then we pick a previously unused vertex $c \in B\cap N(a_1 \dots a_{k - 1})$.
This yields a copy of $T_k$ as required with edges $z_1 \dots z_{k - 2} w a_1$, $z_1 \dots z_{k - 2} w a_2$, and $a_1 \dots a_{k - 1} c$.

After constructing $|X|$ vertex-disjoint copies of $T_k$, we remove them from $H$, leaving two subsets $A' \subseteq A$ and $B' \subseteq B$ with
$|A'| = |A| - (2k - 2)|X| = \frac{2k - 3}{2k - 1}n - (2k - 3)|X| = \frac{2k - 3}{2k - 1}n_1$ and 
$|B'| = |B| - |X| = \frac{2n}{2k - 1} - 2|X| = \frac{2n_1}{2k - 1}$, 
where $n_1 := n - (2k - 1)|X| = |A' \cup B'|$. In particular, 
$n - (2k - 1)(k - 1)^k \sqrt{\gamma} n \le n_1 \le n$ and $(2k - 1) \mid n_1$.
Note that the ratio of the sizes satisfies $|A'| : |B'| = (2k - 3) : 2$. Let $H' := H[A' \cup B']$ be the subgraph induced on $n_1$ vertices. Then we have:

\begin{enumerate}[(i)]
    \item At most $\gamma' n_1^{k - 1}$ many $(k - 1)$-sets in $A'$ are not good. Indeed, each such set is either among the at most $\sqrt{\gamma} n^{k - 1}$ bad sets, or contains a vertex from $X$, which has size at most $(k - 1)^k \sqrt{\gamma} n$.\label{i0}

    \item Every good $(k - 1)$-set in $A'$ has at least $\frac{2n_1}{2k - 1} - \gamma' n_1$ neighbors in $B'$, since each good $(k - 1)$-set had at least $\frac{2n}{2k - 1} - ((k - 1)^k + 1)\sqrt{\gamma} n$ neighbors in $B$.\label{i1}

    \item $\delta(H') \ge \frac{2n}{2k - 1} - (2k - 1)|X| \ge \frac{2n_1}{2k - 1} - \gamma' n_1$, since we removed at most $(2k - 1)|X|$ vertices from $H$.\label{i2}

    \item Each vertex in $B'$ lies in at least $\frac{n_1^{k - 1}}{2(2k - 1)^k}$ edges with exactly $k - 1$ vertices from $A'$, because it was not in $X$ and hence belonged to at least $\frac{n^{k - 1}}{(2k - 1)^k}$ such edges in $H$.\label{i3}
\end{enumerate}

Let $J$ be the $(2k - 1)$-graph on $V(H')$ whose edges are all subsets $S \subseteq V(H')$ with $|S \cap A'| = 2k - 3$ and $|S \cap B'| = 2$, such that $H'[S]$ contains a copy of $T_k$, denoted as $T_k^S$. We now show that $J$ satisfies the following properties:
\begin{enumerate}[(P1)]
    \item $e(J) \ge (1 - \beta) \binom{\frac{(2k - 3)n_1}{2k - 1}}{2k - 3} \binom{\frac{2n_1}{2k - 1}}{2}$. \label{item_p1}
    
    \item Every vertex in $H'$ lies in at least $\frac{1}{(2k - 1)^{(k + 1)^2}} \left( \frac{n_1}{2k - 1} \right)^{2k - 2}$ edges of $J$. \label{item_p2}
\end{enumerate}
Using properties (P\ref{item_p1}) and (P\ref{item_p2}) we may apply Corollary~\ref{auxiliarymatching} to obtain a perfect matching in $J$. The corresponding collection of copies $T_k^S$, one for each edge $S \in J$, forms a perfect $T_k$-tiling in $H'$. Together with the previously selected $|X|$ copies of $T_k$ in $H$, this yields a perfect $T_k$-tiling of $H$, completing the proof.

It remains to verify properties (P\ref{item_p1}) and (P\ref{item_p2}). We begin with (P\ref{item_p1}).
By (i), there are at most $\gamma' n_1^{2k - 3}$ subsets $Y = \{x_1, \dots, x_{2k - 3}\} \subseteq A'$ such that not all $(k - 1)$-sets in $\binom{Y}{k - 1}$ are good.
Now fix a set $Y \subseteq A'$ of size $2k - 3$ in which every $(k - 1)$-set is good. By (ii), the number of vertices $y \in B'$ that fail to lie in $\bigcap_{R \in \binom{Y}{k - 1}} N_{H'}(R)$ is at most $\binom{2k - 3}{k - 1} \gamma' n_1$. Therefore, there are at most $\binom{2k - 3}{k - 1} \gamma' n_1 \cdot |B'|$ pairs $\{y_1, y_2\} \subseteq B'$ such that at least one $y_i\notin\bigcap_{R \in \binom{Y}{k - 1}} N_{H'}(R)$ with $i\in[2]$. 
Thus, the total number of $(2k - 1)$-sets $S = \{x_1, \dots, x_{2k - 3}, y_1, y_2\}$ with $x_i \in A'$, $y_j \in B'$ for each $i\in[2k-3]$ and each $j\in[2]$ that do not belong to $E(J)$ is at most
$\gamma' n_1^{2k - 3} \cdot \binom{|B'|}{2} + \binom{|A'|}{2k - 3} \cdot \binom{2k - 3}{k - 1} \gamma' n_1 \cdot |B'| \le 2 \gamma' n_1^{2k - 1}$.
For every remaining $(2k - 1)$-set $S = \{x_1, \dots, x_{2k - 3}, y_1, y_2\}$ with $x_i \in A'$, $y_j \in B'$ for each $i\in[2k-3]$ and each $j\in[2]$, we have $S \in E(J)$, since $H'$ contains a copy $T_k^S$ of $T_k$ with edges $x_1 \dots x_{k - 1} y_1$, $x_2 \dots x_k y_1$, and $x_1 x_k x_{k + 1} \dots x_{2k - 3} y_2$.
Therefore,
$e(J) \ge \binom{\frac{(2k - 3) n_1}{2k - 1}}{2k - 3} \cdot \binom{\frac{2n_1}{2k - 1}}{2} - 2 \gamma' n_1^{2k - 1}$,
which proves property (P\ref{item_p1}).

To establish property (P\ref{item_p2}), consider any vertex $b \in B'$. By property~(\ref{i3}), there are at least $\frac{n_1^{k - 2}}{2(2k - 1)^k}$ distinct $(k - 2)$-sets $a_1 \dots a_{k - 2} \subseteq A'$ such that $|N(a_1 \dots a_{k-2} b) \cap A'| \geq \frac{n_1}{2(2k-1)^k}$.
For each such set, we can proceed to construct a copy of $T_k$ in $H'$ containing $b$ and exactly $2k - 3$ vertices in $A'$. To do so, select a $(k - 1)$-set $c_1 \dots c_{k - 1} \subseteq N(a_1 \dots a_{k - 2} b) \cap A'$ such that $c_1 \dots c_{k - 1}$ forms a good $(k - 1)$-set. Then, since it is good, property~(\ref{i1}) ensures that the set $N(c_1 \dots c_{k - 1}) \cap B'$ has size at least $\frac{2n_1}{(2k - 1)^2} - \gamma'n_1 \geq \frac{n_1}{(2k - 1)^2}$. Picking a vertex $b_1$ from this set gives a copy of $T_k$ in $H'$ with edges $a_1 \dots a_{k - 2} b c_1$, $a_1 \dots a_{k - 2} b c_2$, and $c_1 \dots c_{k - 1} b_1$.
Moreover, from property~(\ref{i0}), the number of good $(k - 1)$-sets among the vertices in $N(a_1 \dots a_{k - 2} b) \cap A'$ is at least
$\binom{\frac{n_1}{2(2k - 1)^k}}{k - 1} - \gamma' n_1^{k - 1} \ge \frac{n_1^{k - 1}}{(2k - 1)^{k^2}}$.
Therefore, each of the $\frac{n_1^{k - 2}}{2(2k - 1)^k}$ choices for $a_1, \dots, a_{k - 2}$ yields at least $\frac{n_1^{k - 1}}{(2k - 1)^{k^2}} \cdot \frac{n_1}{(2k - 1)^2} = \frac{n_1^{k}}{(2k - 1)^{k^2 + 2}}$ copies of $T_k$ that contain $b$.
Multiplying over all such initial sets, we obtain a total of at least
$\frac{n_1^{2k - 2}}{(2k - 1)^{(k + 1)^2}}$ copies of $T_k$ in $H'$ that include $b$ and exactly $2k - 3$ vertices from $A'$.
Since each $(2k - 1)$-subset of $V(H')$ can support at most $(2k - 1)! \le (2k - 1)^{2k - 2}$ labeled copies  of $T_k$, it follows that $b$ is contained in at least
$\frac{1}{(2k - 1)^{(k + 1)^2}} \left( \frac{n_1}{2k - 1} \right)^{2k - 2}$ edges of $J$, as claimed in (P\ref{item_p2}).

Now fix a vertex $a \in A'$. Suppose that there are more than $\frac{n_1^{k - 3}}{2(2k - 1)^k}$ distinct $(k - 3)$-sets $a_1 \dots a_{k - 3} \subseteq A'$ for which at least $\frac{n_1}{2(2k - 1)^k}$ vertices $b \in B'$ satisfy $|N(a_1 \dots a_{k - 3} a b) \cap A'| \ge \frac{n_1}{2(2k - 1)^k}$.
For such a set $a_1 \dots a_{k - 3}$ and vertex $b \in B'$, we select a good $(k - 1)$-set $c_1 \dots c_{k - 1} \subseteq N(a_1 \dots a_{k - 3} a b) \cap A'$, and then choose $b_1 \in N(c_1 \dots c_{k - 1}) \cap B'$.
This yields a copy of $T_k$ with edges $a_1 \dots a_{k - 3} a b c_1$, $a_1 \dots a_{k - 3} a b c_2$, and $c_1 \dots c_{k - 1} b_1$, which contains $a$ and exactly two vertices in $B'$. By exactly
the same argument as we used for each $b \in B'$ in the previous paragraph, we deduce that property (P\ref{item_p2}) also holds for $a$.

Now assume that for at most $\frac{n_1^{k - 3}}{2(2k - 1)^k}$ many $(k - 3)$-sets $a_1 \dots a_{k - 3} \subseteq A'$, there exist more than $\frac{n_1}{2(2k - 1)^k}$ vertices $b \in B'$ such that $|N(a_1 \dots a_{k - 3} a b) \cap A'| \ge \frac{n_1}{2(2k - 1)^k}$. Equivalently, for more than 
$\binom{\frac{(2k - 3)n_1}{2k - 1}}{k - 3} - \frac{n_1^{k - 3}}{2(2k - 1)^k}$
many $(k - 3)$-sets $a_1 \dots a_{k - 3} \subseteq A'$, there are at least 
$\frac{2n_1}{2k - 1} - \frac{n_1}{2(2k - 1)^k}$
vertices $b \in B'$ for which 
$|N(a_1 \dots a_{k - 3} a b) \cap A'| < \frac{n_1}{2(2k - 1)^k}.$
Fix such a $(k - 3)$-set $a_1 \dots a_{k - 3}$, and pick a good $(k - 1)$-set $c_1 \dots c_{k - 1} \subseteq A'$. Then choose a pair $b_1 b_2 \subseteq B'$ satisfying $a_1 \dots a_{k - 3} a b_1 b_2 \in E(H')$ and $b_1 b_2 \subseteq N(c_1 \dots c_{k - 1})$.
This produces a copy of $T_k$ with edges $c_1 \dots c_{k - 1} b_1$, $c_1 \dots c_{k - 1} b_2$, and $b_1 b_2 a a_1 \dots a_{k - 3}$.
From property~(\ref{i0}), we know there are at least
$\binom{\frac{n_1}{2(2k - 1)^k}}{k - 1} - \gamma' n_1^{k - 1} \ge \frac{n_1^{k - 1}}{(2k - 1)^{k - 1}}$
choices for the good $(k - 1)$-set $c_1 \dots c_{k - 1}$.
Now, using our assumption together with properties~(\ref{i1}) and~(\ref{i2}), the number of suitable pairs $b_1 b_2$ in $B'$ is at least
\begin{equation*}
\begin{aligned}
 &\left( \frac{2n_1}{2k - 1} - \frac{n_1}{2(2k - 1)^k} \right) \left( \frac{2n_1}{2k - 1} - \gamma' n_1 - \frac{n_1}{2(2k - 1)^k} \right) + \binom{\frac{2n_1}{2k - 1} - \gamma' n_1}{2} - \binom{\frac{2n_1}{2k - 1}}{2} \\
 & > \frac{n_1^2}{(2k - 1)^2}.
\end{aligned}
\end{equation*}
Hence, this guarantees sufficiently many constructions of $T_k$ containing $a$ and exactly two vertices from $B'$. As in the previous case, this shows that property~(P\ref{item_p2}) holds for $a$, completing the argument.
\end{proof}

We end this section with a proof of Corollary \ref{auxiliarymatching}.

\begin{proof}[\bf{Proof of Corollary \ref{auxiliarymatching}}]
Let $G$ denote the $(2k-1)$-graph on the vertex set $V=A\cup B$ whose edges are all $(2k-1)$-sets in $V$ which contains $2k-3$ vertices from $A$ and $2$ vertices from $B$. Thus, $e(G) = \binom{(2k-3)n}{2k-3} \binom{2n}{2}$. Note that $J$ is a spanning subgraph of $G$. 
Define $X \subseteq V$ to be the set of vertices $x$ satisfying $d_G(x) - d_J(x) \ge \beta^{1/2}n^{2k-2}$. Given that $e(J) \ge (1 - \beta)\binom{(2k-3)n}{2k-3} \binom{2n}{2}$, it follows that $|X| \le (2k - 3)^{2k - 3} \beta^{1/2} n$.
%Note that for any subset $S\subseteq V(J)$ of size $s$, we have  $d_J(v)-d_{J\setminus S}(v)\le (2k - 1)^{2k - 3} sn^{2k - 3}$ for every vertex $v\in V(J)\setminus S$.
Moreover, combining the assumption that every vertex in $J$ lies in at least $n^{2k - 2}/(2k - 1)^{(k + 1)^2}$ edges and the fact that $|X|\le (2k - 3)^{2k - 3} \beta^{1/2} n$ for $\beta\ll \tfrac{1}{k}$, we may greedily select vertex-disjoint edges in $J$ to cover all vertices in $X$. This yields a matching $M_1$ in $J$ of size at most $|X|$.

Let $V_1 := V \setminus V(M_1)$, and define $J_1 := J[V_1]$ and $G_1 := G[V_1]$. Then for each $v \in V_1$ we have 
\begin{equation*}
    \begin{aligned}
     d_{J_1}(v) &\geq d_J(v)- (2k-1)^{2k-3}|V(M_1)|n^{2k-3} \\ &\geq d_J(v) - (2k-1)^{2k-3}k(2k-3)^{2k-3}\beta^{1/2}n^{2k-2}   
    \end{aligned}           
\end{equation*}
and $d_{G_1}(v) \leq d_{G}(v)$. As $v \notin X$ and $\beta\ll \tfrac{1}{k}
$, it follows that $d_{G_1}(v) - d_{J_1}(v) \le \beta^{1/3}n^{2k-2}$. Now 
arbitrarily partition $A'=A\cap V_1$ into $2k-3$ parts of equal size, and $B'=B\cap V_1$ into two parts of equal size. Let $J_2 \subseteq J_1$ and $G_2 \subseteq G_1$ be the resulting $(2k-1)$-partite subgraphs from the $2k-1$ parts, where each part has size $n_2:=n-|M_1|$. Moreover for every $v \in V_1$ we have $d_{G_2}(v) - d_{J_2}(v) \leq d_{G_1}(v) - d_{J_1}(v) \leq \beta^{1/3}n^{2k-2}$ and $d_{G_2}(v) = n_2^{2k-2}$, so $ d_{J_2}(v) \geq n_2^{2k-2}-\beta^{1/3}n^{2k-2} \geq (2k-2)n_2^{2k-2}/(2k-1)$. By Theorem~\ref{5partitematching} $J_2$ admits a perfect matching $M_2$. Hence, $M_1\cup M_2$ is a perfect matching in~$J$.
\end{proof}

\section{Perfect rainbow tilings} \label{sec:rainbow} We conclude with a concise derivation of Theorem~\ref{thm:rainbow} using the following key concepts.

The \emph{tiling threshold} $\delta_F$ is the infimum of $\delta > 0$ such that for any $\mu > 0$, there exists $n_0$ where for $n \geq n_0$ with $\nu (F) \mid n$, every $k$-graph $H$ on $n$ vertices with $\delta(H) \geq (\delta + \mu)n$ has a perfect $F$-tiling.    
The \emph{rainbow tiling threshold} $\delta^r_F$ is defined similarly for families $\mathcal{H} = \{H_1, \dots, 
    H_{e(F)n/\nu (F)}\}$ requiring $\delta(H_i) \geq (\delta + \mu)n$ for each $i$.
    
For $k$-graphs $H_1, H_2$ on common vertex set $V$, a \emph{color covering homomorphism} from $F$ to $(H_1, H_2)$ is an injection $\phi: V(F) \to V$ where for some $e \in E(F)$, $\phi(e) \in E(H_1)$ and $\phi(e') \in E(H_2)$ for all $e' \neq e$ in $E(F)$.
The \emph{color covering threshold} $\delta^c_F$ requires that every such pair with $\delta(H_1), \delta(H_2) \geq (\delta + \mu)n$ admits such a homomorphism.

The derivation uses Lang's fundamental equivalence:

\begin{theorem}[\cite{Lang}, Theorem 5.11]\label{thm:lang}
For all $k$-graph $F$, we have $\delta^r_F = \max\{\delta^c_F, \delta_F\}$.
\end{theorem}

\begin{proof}[{\bf Proof of Theorem~\ref{thm:rainbow}}]
First we establish $\delta_{T_k}^\mathrm{c} = 0$. Consider $k$-graphs $H_1, H_2$ on $n \geq 2k-1$ vertices with $\delta(H_i) \geq k$. Fix an edge $e = v_1\cdots v_k \in E(H_1)$. By minimum codegree condition, there exists $v_{k+1}$ such that $v_1\cdots v_{k-1}v_{k+1} \in E(H_2)$ and $v_{k+2}, \dots, v_{2k-1}$ such that $v_kv_{k+1}\cdots v_{2k-1} \in E(H_2)$.
Then $\{v_1\cdots v_k, v_1\cdots v_{k-1}v_{k+1}, v_kv_{k+1}\cdots v_{2k-1}\}$ forms a copy of $T_k$ where $e$ is covered by $H_1$ and all other edges covered by $H_2$
providing the required colour covering homomorphism.

By Theorem~\ref{thm:main} and the extremal construction $H_\mathrm{ext}$, we have $\delta_{T_k} = \frac{2}{2k-1}$. Thus by Theorem~\ref{thm:lang}, 
\[
\delta^r_{T_k} = \max\left\{0, \frac{2}{2k-1}\right\} = \frac{2}{2k-1}.
\]
The definition of $\delta^r_{T_k}$ then implies Theorem~\ref{thm:rainbow}.
\end{proof}

\section{Concluding remarks}
In this paper, for all $k \geq 3$, we determine the optimal minimum codegree threshold for perfect $T_k$-tilings in $k$-graphs. Our proof uses the lattice-based absorption method, as is usual, but develops a unified and effective approach to build transferrals for all uniformities. %Combining with Bowtell, Kathapurkar, Morrison and Mycroft's result, the optimal minimum codegree threshold for perfect $T_k$-tilings is completely obtained. 
To the best of our knowledge, for general $k$, our work establishes the first optimal minimum codegree condition guaranteeing perfect $F$-tilings for non-$k$-partite $k$-uniform hypergraphs $F$. Also, we believe that our work can provide strong support for resolving perfect $F$-tilings in the case when $F$ is a non-$k$-partite $k$-graph.

Additionally, we establish an asymptotically tight minimum codegree threshold for the rainbow variant of the problem.%染色对其他图有启发？blowup的方法比较technical找T_k，用blowup可以找各种形状的T_k，兴许对以后的推广图类也许有帮助，为我们提供了不同的嵌入类型。
We conjecture that the bound in Theorem~\ref{thm:rainbow} can be strengthened to $\delta(H_i) \geq \frac{2}{2k-1}n + c$ for some absolute constant $c$. However, our method inherits a limitation from Lang's theorem~\cite{Lang}, which requires asymptotic minimum degree bounds rather than exact ones.

\appendix
\section{Proof of Lemma \ref{bip-temp}}\label{lem3.2}
Adapt with Nenadov's proof for Lemma 2.3 in \cite{NP}, the proof of Lemma \ref{bip-temp} is based on ideas of Montgomery \cite{Montgo} and relies on the existence of ``robust'' sparse bipartite graphs given by the following lemma.
\begin{lemma}
\emph{ \cite[Lemma 2.8]{Montgo}; \cite[Lemma 2.3]{NP}}\label{robust}
Let~$\beta>0$ be given.  There exists $m_0$ such that the following holds for every $m\geq m_0$. There exists a bipartite graph $B_m$ with vertex classes $X_m\cup Y_m$ and $Z_m$ and maximum degree $\Delta(B_m)\leq40$, such that $|X_m|=m+\beta m$, $|Y_m| =2m$ and $|Z_m| = 3m$, and for every subset $X_m'\subseteq X_m$ with $|X_m'|=m$, the induced graph $B_m[X_m'\cup Y_m, Z_m]$ contains a perfect matching.
\end{lemma}
With this in mind, we are ready to prove Lemma \ref{bip-temp}.
From the assumption that for every $S \in \binom{V(H)}{s}$ there are $\gamma n$ vertex-disjoint $(F,t)$-absorbers, we conclude that for every vertex $v \in V(H)$ there is a family of at least $\gamma n/s$ copies of $F$ which contain $v$ and are otherwise vertex-disjoint. Indeed, given $v$ and copies $F_1, \dots, F_k$ of $F$ containing $v$ that are otherwise vertex-disjoint, take any set $S$ of size $s$ which contains $v$ and is otherwise disjoint from $\bigcup_i F_i$. Then as long as $k < \gamma n/(s-1)$, there exists an $(F,t)$-absorber $A_S$ disjoint from $\bigcup_i F_i$. From the perfect $F$-tiling of $S \cup A_S$, take $F_{k+1}$ to be the copy of $F$ which contains $v$. Continuing in a similar fashion, we obtain a family $\mathcal{F}_v = \{F_i - \{v\}\}_i$ which has size at least $\gamma n/s$, and all sets in this family are pairwise vertex-disjoint.

Choose a subset $X \subseteq V(H)$ by including each vertex of $H$ with probability $q = \gamma/(500 s^2 t)$. The parameter $q$ is chosen such that the calculations work and for now it is enough to remember that $q$ is a sufficiently small constant. A simple application of Chernoff's inequality and a union bound show that with positive probability $nq/2 \leq |X| \leq 2nq$ and for each vertex $v \in V(G)$, at least $q^{s-1} |\mathcal{F}_v|/2$ sets from $\mathcal{F}_v$ are contained in $X$. Therefore, there exists one such $X$ for which these properties hold. Let us denote the family of sets from $\mathcal{F}_v$ completely contained in $X$ by $\mathcal{F}'_v$.

Set $\beta = q^{s-1}\gamma/(4s)$ and $m = |X|/(1 + \beta)$. Note that $m = \Omega_{q,s,\gamma}(n)$. Let $B_m$ be a graph given by Lemma \ref{robust}. Choose disjoint subsets $Y, Z \subseteq V(H) \setminus X$ of size $|Y| = 2m$ and $|Z| = 3m(s-1)$ and arbitrarily partition $Z$ into subsets $\mathcal{Z} = \{Z_i\}_{i \in [3m]}$ of size $s-1$. Take any injective mapping $\varphi_1: X_m \cup Y_m \to X \cup Y$ such that $\varphi_1(X_m) = X$, and any injective $\varphi_2: Z_m \to Z$. We claim that there exists a family $\{A_e\}_{e \in B_m}$ of pairwise disjoint $(\le s^2t)$-subsets of $V(G) \setminus (X \cup Y \cup Z)$ such that for each $e = \{w_1, w_2\} \in B_m$, where $w_1 \in X_m \cup Y_m$ and $w_2 \in Z_m$, the set $A_e$ is $(F, t)$-absorber.

Such a family can be chosen greedily. Suppose we have already found desired subsets for all the edges in some $E_0 \subset B_m$. These sets, together with $X \cup Y \cup Z$, occupy at most
\begin{equation*}
\begin{aligned}
|X| + |Y| + |Z| + s^2t|E_0| & < 4m + 3m(s-1) + s^2t \cdot 40|Z_m| \\ 
&\leq 4sm + 120s^2tm \leq 124s^2tm \\ 
& < 250s^2tnq \leq \gamma n/2
\end{aligned}
\end{equation*}
vertices in $H$. Choose arbitrary $e = \{w_1, w_2\} \in B_m \setminus E_0$. As there are $\gamma n$ vertex-disjoint $(F, t)$-absorbers, there are at least $\gamma n/2$ ones which do not contain any of the previously used vertices. Pick any and proceed.

We claim that
\[
A = X \cup Y \cup Z \cup \left( \bigcup_{e \in B_m} A_e \right)
\]
has the $(F,\xi)$-absorbing property for $\xi = \beta/(s-1)$. Consider some subset $R \subseteq V(H) \setminus A$ such that $|R| + |A| \in s\mathbb{Z}$ and $|R| \leq \xi m$. As
\[
|\mathcal{F}'_v| > q^{s-1} \gamma n/(2s) = 2\beta n = 2\xi n(s-1) > 2|R|(s-1)
\]
we can greedily choose a subset $A_v \in \mathcal{F}'_v$ for each $v \in R$ such that all these sets are pairwise disjoint (recall each set in $\mathcal{F}'_v$ is of size $s-1$ and forms a copy of $F$ with $v$). This takes care of vertices from $R$ and uses exactly $|R|(s-1) \leq \beta m$ vertices from $X$. Note that $|A| + |R| \in s\mathbb{Z}$ implies that $|R| + \beta m \in s\mathbb{Z}$, so $|X| - |R|(s-1) - m \in s\mathbb{Z}$, thus we can cover the remaining vertices from $X$ with vertex-disjoint copies of $F$ such that there are exactly $m$ vertices remaining. Again, $|\mathcal{F}'_v| > 2\beta n > 2\beta m$ implies that such copies of $F$ can be found in a greedy manner.

Let $X'$ denote the remaining vertices from $X$ and set $X'_{m} = \varphi_1^{-1}(X')$. By Lemma \ref{robust} there exists a perfect matching $M$ in $B_m$ between $X'_{m} \cup Y_m$ and $Z_m$. For each edge $e = \{w_1, w_2\} \in M$ take an perfect $F$-tiling in $H[\varphi_1(w_1) \cup \varphi_2(w_2) \cup A_e]$ and for each $e \in B_m \setminus M$ take an perfect $F$-tiling in $H[A_e]$. All together, this gives an perfect $F$-tiling of $H[A \cup R]$.

\section{Proof of Lemma \ref{close}}\label{lem3.3}
For positive integers $s, t$ with $s\ge 3$ and $\beta>0$, let $H,F,\mathcal{P}$ and $r$ be given as in the assumption. For any $s$-subset $S\subset V(H)$ with $\textbf{i}_{\mathcal{P}}(S)$ being $(F,\beta)$-robust, we greedily construct as many pairwise disjoint $(F,t)$-absorbers for $S$ as possible. Let $\mathcal{A}=\{A_1,A_2,\ldots, A_{\ell}\}$ be a maximal family of $(F,t)$-absorbers constructed so far. Suppose to the contrary that $\ell<\frac{\beta}{k^3t}n$.
%\lfloor\frac{\be n-k}{k^2t}\rfloor$,
Then $|\bigcup_{\mathcal{A}}A_i| \le \tfrac{\beta}{k} n $ as each such $A_i$ has size at most $k^2t$. \medskip

%Note that b
By $(F,\beta)$-robustness, we can pick a copy of $F$ inside $V(H)\setminus(\bigcup_{\mathcal{A}}A_i\cup S)$ whose vertex set $T$ has the same index vector as $S$. Let $S=\{s_1,s_2,\ldots,s_s\}$ and $T=\{t_1, t_2,\dots, t_s\}$ such that $s_i$ and $t_i$ belong to the same part of $\mathcal{P}$ for each $i\in [s]$. We now greedily pick up a collection $\{S_1,S_2,\ldots, S_{s}\}$ of vertex disjoint subsets in $V(H)\setminus (\bigcup_{\mathcal{A}}A_i\cup S\cup T)$ such that each $S_i$ is an $F$-connector for $s_i,t_i$ with $|S_i|\le st-1$. Since
\[\left|\bigcup_{i=1}^\ell A_i\cup\left(\bigcup_{i=1}^{s'}S_i\right)\cup S\cup T\right|\le\beta n,\]
for any $0\leq s'\leq s$ (using that $n$ is sufficiently large), we can pick each such $S_i$ one by one because $s_i$ and $t_i$ are $(F,\beta n,t)$-reachable.
At this point, it is easy to verify that $\bigcup_{i=1}^{k}S_i\cup T$ is actually an $(F,t)$-absorber for $S$, contrary to the maximality of $\ell$.

\section{Proof of Lemma \ref{partition}}\label{lem3.4}

We shall make use of the following subtle observation in the proof of Lemma \ref{partition}.
\begin{fact}\label{fact:concatenate}
Let $H$ be a graph of $n$ vertices and let $F$ be a graph of $k$ vertices.
For two vertices $u, w\in V(G)$, if there exists $m_1$ vertices $v\in V(G)$ which are $(F, m_2, t)$-reachable to both $u$ and $w$, respectively, then $u$ and $w$ are $(F, m, 2t)$-reachable, where $m= \min\{m_1-1, m_2-kt\}$.
\end{fact}

Let $r_0 = \lceil1/\delta\rceil+1$ and  choose constants $0<\beta_{r_0+1}\ll \beta_{r_0}\ll\beta_{r_0-1} \ll \cdots \ll \beta_1=\delta$. Furthermore, fix $\beta=\beta_{r_0+1}$ and $t=2^{r_0}$.
Assume that there are two vertices that are not $(F, \beta_{r_0} n, 2^{r_0-1})$-reachable, as otherwise we just output
$\mathcal P=\{V(H)\}$ as the desired partition. % with $\be=\be_{r_0}$ and $t=2^{r_0-1}$.
Observe  also that every set of $r_0$ vertices must contain two vertices that are $(F, \beta_2n, 2)$-reachable to each other.
Indeed,
fixing some arbitrary set of vertices $S=\{v_1,\ldots,v_{r_0}\}\subset V(H)$,
by the Inclusion-Exclusion principle we have that there exists a pair $i\neq j\in [r_0]$ so that  both $v_i$ and $v_j$ are 
both $(F, \beta_1 n, 1)$-reachable to a set of at least $\delta' n$  vertices for some $\delta'\ge \tfrac{\delta}{\binom{r_0}{2}}$.
Thus, by Fact~\ref{fact:concatenate}, we deduce that $v_i$ and $v_j$ are $(F, \beta_2 n, 2)$-reachable for as  $\beta_2\ll \beta_1=\delta$.\medskip

To ease the notation, we write $\lambda_i=\beta_{r_0+2-i}$ and $t_i = 2^{r_0+1-i}$ for each $i\in [1,r_0+1]$. Therefore \[0<\lambda_1\ll \lambda_2\ll \lambda_3\cdots\ll\lambda_{r_0+1}=\beta_1.\] Let $d$ be the largest integer with $2\le d \le r_0$  such that there are $d$ vertices $v_1,\dots, v_d$ in $H$ which are pairwise \emph{not} $(F, \lambda_d n, t_d)$-reachable. Note that $d$ exists and $2\le d\le r_0-1$. Indeed, we assumed above that there are  $2$ vertices that are not $(F,\lambda_2n,t_2)$-reachable and  if $d= r_0$, then there are $r_0$ vertices in $H$ which are pairwise not $(F, \beta_{2} n, 2)$-reachable, contrary to the observation above. Let $S=\{v_1, \dots, v_d\}$ be such a set of vertices and note that  we know that $v_1, \dots, v_d$ are also pairwise not $(F, \lambda_{d+1} n, t_{d+1})$-reachable.\medskip

For a vertex $v$ we write $\tilde{N}_i(v)$ for the set
of vertices which are $(F, \lambda_i n, t_i)$-reachable to $v$. Consider $\tilde{N}_{d+1}(v_i)$ for $i\in [d]$.
Then we conclude that
\begin{description}
  \item[(i)] any vertex $v\in V(G) \setminus \{ v_1, \dots, v_d\}$ must be
in $\tilde{N}_{d+1}(v_i)$ for some $i\in [d]$.
Otherwise, $v, v_1, \dots, v_d$ are pairwise not $(F, \lambda_{d+1} n, t_{d+1})$-reachable, contradicting the maximality of $d$.
  \item[(ii)] $| \tilde{N}_{d+1}(v_i) \cap \tilde{N}_{d+1}(v_j)| \le \lambda_d n+1$ for any $i\not = j$. Otherwise Fact \ref{fact:concatenate} implies that $v_i, v_j$ are $(F, \lambda_d n, t_d)$-reachable to each other, a contradiction (using that  $\lambda_d\ll\lambda_{d+1}$ here).
\end{description}

For $i\in [d]$, let \[U_i = (\tilde{N}_{d+1}(v_i)\cup \{ v_i\})
\setminus \bigcup_{ j\not = i} \tilde{N}_{d+1}(v_j).\]
Then we claim that each $U_i$ is $(F, \lambda_{d+1}n, t_{d+1})$-closed.
Indeed otherwise, there exist $u_1, u_2\in U_i$ that are not
$(F, \lambda_{d+1}n, t_{d+1})$-reachable to each other. Then
$\{ u_1, u_2\} \cup \{ v_1, \dots, v_d\} \setminus \{v_i\}$
contradicts the maximality of $d$.\medskip

Let $U_0 = V(H) \setminus (U_1\cup\dots\cup U_d)$.
We have $|U_0| \le d^2\lambda_d n$ due to {\textbf{(ii)}} above. In order to obtain the desired reachability partition,
we now drop each vertex of $U_0$ back into some $U_i$ for $i\in [d]$ as follows. Since each $v\in U_0$ is $(F, \lambda_{r_0+1}n, t_{r_0+1})$-reachable to at least $\delta n$ vertices, it holds due to the fact that  $\lambda_d\ll \delta$, that
\[|\tilde{N}_{r_0+1}(v)\setminus U_0| \ge \delta n - |U_0|\ge \delta n - d^2\lambda_d n > d\lambda_d n.\]
Therefore there exists some $i \in [d]$ such that $v$ is $(F, \lambda_{r_0+1}n, t_{r_0+1})$-reachable to at least
$\lambda_d n+1$ vertices in $U_i$ and hence $(F, \lambda_{d+1}n, t_{d+1})$-reachable to all these vertices. Fact \ref{fact:concatenate} (using that $\lambda_d\ll\lambda_{d+1}$) implies  that  $v$ is $(F, \lambda_d n, t_d)$-reachable to every vertex in $U_i$.
Now partition $U_0$ as  $U_0=\cup_{i\in d}R_i$ where for each $i\in[d]$,  $R_i$ denotes a set of  vertices $v\in U_0$ that are $(F, \lambda_d n, t_d)$-reachable to every vertex in $U_i$. Again by Fact \ref{fact:concatenate}, for every $i\in [d]$,  every two vertices in $R_i$ are $(F, \lambda_{d-1} n, t_{d-1})$-reachable to each other. Let $\mathcal{P}=\{V_1, \dots, V_d\}$ be the resulting partition by setting $V_i=R_i\cup U_i$. Then each $V_i$ is $(F, \lambda_{d-1} n, t_{d-1})$-closed.
Also, for each $i\in[d]$, it holds that \[|V_i| \ge |U_i| \ge |\tilde{N}_{d+1}(v_i)| - d^2 \lambda_d n
\ge |\tilde{N}_{r_0+1}(v_i)| - \tfrac{\delta}{2} n
%\ge (\delta - \beta_1) n
\ge  \tfrac{\delta}{2}n.\]
This completes the proof by noting that  $\beta=\beta_{r_0+1}\le\lambda_{d-1}$ and $t=2^{r_0}\ge t_{d-1}$.

\section{Proof of Lemma \ref{transferral}}\label{fulu:transferral}
Let $H$ be an $n$-vertex $k$-graph and $\mathcal{P}=\{V_1, V_2,\ldots,V_r\}$ be a partition of $V(H)$ as in the assumption.
Suppose that there exists a vector $\textbf{v}\in L^{\beta}(\mathcal{P})$ such that $\textbf{v}=\textbf{u}_i-\textbf{u}_j$, where $i\neq j\in [r]$. Without loss of generality (relabelling if necessary), we may assume that $i = 1$ and $j = 2$. Then we may denote $\textbf{v}=\sum_{l=1}^pa_l \textbf{s}_l$, where $a_l\in \mathbb{Z}$, $\textbf{s}_l=(s_l^1,s_l^2,\ldots, s_l^r)\in I^{\beta}(\mathcal{P})$ for every $l\in[p]$.  
It suffices to show that there exists a constant $C=C(F,r)$ such that every two vertices $x\in V_1$ and $y\in V_2$ are $\left(F,\tfrac{\beta}2n, Cst\right)$-reachable. Let $x\in V_1$ and $y\in V_2$ be any two vertices. Fix some vertex set $W\subset V(H)\setminus \{x,y\}$ of size at most $ \tfrac{\beta}{2}n$. %Let a=$\sum_{l=1}^p|a_l|$.
By the assumption of $(F,\beta)$-robustness, we can pick $\sum_{l=1}^p|a_l|$ vertex-disjoint copies $F_{\textbf{s}_1}^1,\dots,F_{\textbf{s}_1}^{|a_1|},\cdots,F_{\textbf{s}_p}^1,\dots,F_{\textbf{s}_p}^{|a_p|}$ of $F$ in $V(H)\setminus (W\cup\{x,y\})$ whose corresponding vertex set $V(F_{\textbf{s}_l}^j)$ has index vector $\textbf{s}_l$ for every $l\in[p]$ and $j\in[|a_l|]$, respectively. Let $U_{\textbf{v}}=\cup_{l=1}^p \cup_{j=1}^{|a_l|}V(F_{\textbf{s}_l}^j)$ and $U_{\textbf{v}}^i=U_{\textbf{v}}\cap V_i$ for every $i\in[r]$. \medskip

Without loss of generality, assume that $V(F_{\textbf{s}_1}^1)\cap V_1\ne \emptyset$ and $V(F_{\textbf{s}_p}^1)\cap V_2\ne \emptyset$. Note that we can choose $x' \in V(F_{s_1}^1) \cap V_1$ and 
$y' \in V(F_{s_p}^1) \cap V_2$ such that, for every $i \in [r]$, 
the set $U_{\mathbf{v}}^i \setminus \{x', y'\}$ can be partitioned into two equal parts $U_{\mathbf{v}}^i(1)$ and $U_{\mathbf{v}}^i(2)$ with the property that, for every $l \in [p]$ and $j \in [|a_l|]$, there exists 
some $q \in [2]$ such that 
$V(F_{\mathbf{s}_l}^j) \cap V_i \subseteq U_{\mathbf{v}}^i(q),$
and moreover, for each $F_{\mathbf{s}_l}^j$, the same $q$ works for 
all $i \in [r]$.

Let $z_i:=|U_{\mathbf{v}}^i(1)|=|U_{\mathbf{v}}^i(2)|$, $U_{\mathbf{v}}^i(1)=\{u^i_1,\cdots,u^i_{z_i}\}$ and $U_{\mathbf{v}}^i(2)=\{v^i_1,\cdots,v^i_{z_i}\}$ for every $i \in [r]$. Since each $V_i$ is $(F,\beta n, t)$-closed for each $i\in [r]$, we greedily pick a collection $\mathbb{S}^i=\{S^i_1,S^i_2,\ldots, S^i_{z_i}\}$ of vertex disjoint subsets in $V(H)\setminus (W\cup U_{\textbf{v}}\cup\{x,y\})$ such that for each $j\in[z_i]$,  $S^i_j$ is an $F$-connector for $u^i_j,v^i_j$ with $|S^i_j|\le st-1$. Moreover, $\mathbb{S}^i$ is vertex-disjoint from all previously chosen collections. Indeed, note that as $n$ sufficiently large we have that  $|W|\le \beta n- \sum_{i=1}^r z_i st - \sum_{l=1}^p |a_l|s$ and so for any $s'\leq s-1$ we have
\[
\left|\left(\bigcup_{i=1}^r\bigcup_{j=1}^{s'}S^i_j\right)\cup W\cup U_{\textbf{v}}\cup\{x,y\}\right|\le \beta n.
\] 
Therefore, we can indeed pick the $S^i_j$ one by one since $u^i_j$ and $v^i_j$ are $(\beta n ,t)$-reachable. Similarly, we additionally choose two vertex-disjoint (from each other and all other previously chosen vertices) $F$-connectors, say $S_x$ and $S_y$, for $x,x'$ and $y,y'$, respectively. At this point, it is easy to verify that the subset $\hat S:=\bigcup_{i=1}^{r}\bigcup_{j=1}^{z_i}S^i_j\cup S_x\cup S_y\cup \big(\bigcup_{i=1}^p\bigcup_{j=1}^{|a_i|} V(F_i^j)\big)$ is actually an $F$-connector for $x,y$ with size at most $\sum_{i=1}^r z_i st + \sum_{l=1}^p |a_l|s + 2st \leq Cs^2t -1$. 
Suppose that we would like to show that there exists a perfect $F$-tiling in $H[\hat S \cup \{x\}] $ (leaving $y$ uncovered). We may assume that $V(F_{s_1}^1)\setminus{\{x'\}} \cap V_1\subseteq U_{\mathbf{v}}^1(1) $, then we can take the perfect $F$-tilings in $H[S_x\cup\{x\}]$, $H[S_y\cup\{y'\}]$ and $H[S^i_j\cup\{v^i_j\}]$ for all $i\in[r]$ and $j\in [z_i]$, as well as the copies of $F$ on $\bigcup_{i=1}^r U_{\mathbf{v}}^i(1)$. Similarly, we can show that there exists a perfect $F$-tiling in $H[\hat S \cup \{y\}]$. Hence, $x$ and $y$ are $\left(F,\tfrac{\beta}{2}n, Cst\right)$-reachable.


\begin{thebibliography}{10}

\bibitem{AY} {\sc N. Alon and R. Yuster}, {\em $H$-factors in dense graphs}, {Journal of Combinatorial Theory, Series B} {\bf 66} (1996), 269--282.

\bibitem{BKMM} {\sc C. Bowtell, A. Kathapurkar, N. Morrison, and R. Mycroft}, {\em Perfect tilings of 3-graphs with the generalised triangle}, arXiv: 2505.05606.


\bibitem{cdn} {\sc A. Czygrinow, L. DeBiasio and B. Nagle}, {\em Tiling 3-uniform hypergraphs with $K^3_4-2e$}, {Journal of Graph Theory}~{\bf 75} (2014), 124--136.

\bibitem{cn} {\sc A. Czygrinow and B. Nagle}, {\em A note on codegree problems for hypergraphs}, {Bulletin of the Institute of Combinatorics and its Applications}~{\bf 32} (2001), 63--69.

\bibitem{CH} {\sc K. Corr\'adi and A. Hajnal}, \emph{On the maximal number of independent circuits in a graph}, Acta Mathematica Academiae Scientiarum Hungaricae {\bf 14} (1963), 423--439.

\bibitem{DH} {\sc D.E.~Daykin and R.~H\"aggkvist}, {\em Degrees giving independent edges in a hypergraph}, Bulletin of the Australian Mathematical Society {\bf 23} (1981), 103--109.

\bibitem{Dirac} {\sc G.A.~Dirac}, {\it Some theorems on abstract graphs}, Proceedings of the London Mathematical Society {\bf s3--2} (1952), 69--81.


\bibitem{Edmonds} {\sc J. Edmonds}, {\em Paths, trees, and flowers},
{Canadian Journal of Mathematics} {\bf 17} (1965), 449--467.

\bibitem{GHZ} {\sc W.~Gao, J.~Han and Y.~Zhao}, {\em Codegree conditions for tiling complete k-partite k-graphs and loose cycles}, Combinatorics, Probability and Computing {\bf 28} (2019), 840--870.

\bibitem{HSz} {\sc A. Hajnal and E. Szemer\'edi}, {\em Proof of a conjecture of Erd\H{o}s}, Combinatorial Theory and its Applications, Colloquia Mathematica Societatis J\'anos Bolyai {\bf 4} (1970), 601--623.

\bibitem{han16} {\sc J. Han}, {\em Near perfect matchings in k-uniform hypergraphs II}, { SIAM J. Discret. Math.}~{\bf 30} (2016), 1453--1469.

\bibitem{han17} {\sc J. Han}, {\em Decision problem for perfect matchings in dense k-uniform hypergraphs}, { Trans. Am. Math. Soc.}~{\bf 369} (2017), 5197--5218.

\bibitem{hlsm} {\sc J. Han, A. Lo and N. Sanhueza-Matamala}, {\em Covering and tiling hypergraphs with tight cycles}, {Combinatorics, Probability and Computing}~{\bf 30} (2021), 288--329.

\bibitem{HLTZ} {\sc J. Han, A. Lo, A. Treglown and Y. Zhao}, 
{\it Exact minimum codegree threshold for $K^-_4$-factors}, 
Combinatorics, Probability and Computing {\bf 26} (2017), 856--885. 

\bibitem{HMWY} {\sc J.~Han, P.~Morris, G.~Wang and D.~Yang}, {\em A Ramsey–Tur\'{a}n theory for tilings in graphs}, Random Struct. Alg. {\bf 64} (2024),  94--124. 

\bibitem{HT} {\sc J.~Han and A.~Treglown}, {\em The complexity of perfect matchings and packings in dense hypergraphs}, Journal of Combinatorial Theory, Series B {\bf 141} (2020), 72--104. 

\bibitem{blowup} 
{\sc F.~Illingworth, R.~Lang, A.~M\"{u}yesser, O.~Parczyk, and A.~Sgueglia},
{\it Spanning spheres in Dirac hypergraphs}, arXiv:2407.06275v2.

\bibitem{HZZ} {\sc J.~Han, C.~Zang and Y.~Zhao}, {\em Minimum vertex degree conditions for tiling 3-partite 3-graphs}, Journal of Combinatorial Theory, Series A {\bf 149} (2017), 115--147. 

\bibitem{keevash15} {\sc P. Keevash and R. Mycroft}, {\it A geometric theory for hypergraph matching}, Memoirs of the American Mathematical Society {\bf 233} (2015), monograph 1098.
 
\bibitem{KH} {\sc D.G. Kirkpatrick and P. Hell},
{\em On the complexity of general graph factor problems},
{SIAM Journal on Computing} {\bf 12} (1983), 601--609. 

\bibitem{Komlos} {\sc J. Koml\'os}, {\em Tiling {T}ur\'an theorems}, Combinatorica \textbf{20} (2000), 203--218.

\bibitem{KSSz} {\sc J. Koml\'os, G. S\'ark\"ozy and E. Szemer\'edi}, {\em Proof of the Alon-Yuster conjecture}, Discrete Mathematics {\bf 235} (2001), 255--269.


\bibitem{ko_loose} {\sc D. K\"uhn and D. Osthus}, {\em Loose Hamilton cycles in 3-uniform hypergraphs of large minimum degree}, {Journal of Combinatorial Theory, Series B}~{\bf 96} (2006), 767--821.

\bibitem{ko_match} {\sc D. K\"uhn and D. Osthus}, {\em Matchings in hypergraphs of large minimum degree}, {Journal of Graph Theory}~{\bf 51} (2006), 269--280.

\bibitem{KO} {\sc D. K\"uhn and D. Osthus}, {\em Embedding large subgraphs into dense graphs}, Surveys in Combinatorics 2009, Cambridge University Press, 2009, 137--167.

\bibitem{KO2}{\sc  D.~K\"uhn and D.~Osthus}, {\em The minimum degree threshold for perfect graph packings}, {Combinatorica}~{\bf 29} (2009), 65--107.

\bibitem{Lang} {\sc R.~Lang}, {\em Tiling dense hypergraphs}, arXiv:2308.12281.

\bibitem{lm2} {\sc A. Lo and K. Markstr\"{o}m}, {\em Minimum codegree threshold for ($K^3_4-e$)-factors}, {Journal of Combinatorial Theory, Series A}~{\bf 120} (2013), 708--721.

\bibitem{LM} {\sc A. Lo and K. Markstr\"{o}m}, {\em $F$-factors in hypergraphs via absorption}, {Graphs and Combinatorics}~{\bf 31} (2015), 679--712.

\bibitem{LY} {\sc D.~G. Luenberger and Y. Ye}, {\it Linear and nonlinear programming}, fifth edition, International Series in Operations Research \& Management Science, 228, Springer, Cham, 2021.

\bibitem{Montgo} {\sc R.~Montgomery}, {\em Spanning trees in random graphs}, Adv. Math. {\bf 356} (2019), 106793.

\bibitem{Mycroft} {\sc R.~Mycroft}, {\em Packing k-partite k-uniform hypergraphs}, Journal of Combinatorial Theory, Series A {\bf 138} (2016), 60--132.


\bibitem{NP} {\sc R.~ Nenadov and Y. Pehova}, {\em On a Ramsey-Tur\'{a}n variant of the Hajnal-Szemer\'{e}di theorem}, SIAM J. Discret. Math. {\bf 34} (2020),  1001--1010.

 \bibitem{pikhurko} {\sc O. Pikhurko}, {\em Perfect matchings and $K^3_4$-tilings in hypergraphs of large codegree}, {Graphs and Combinatorics}~{\bf 24} (2008), 391--404.

\bibitem{PS} {\sc N.~Pippenger and J. Spencer}, {\em Asymptotic behavior of the chromatic index for hypergraphs}, Journal of Combinatorial Theory, Series A {\bf 51} (1989), 24--42.

\bibitem{RR} {\sc V.~R\"{o}dl and A.~Ruci\'{n}ski}, {\it Dirac-type questions for hypergraphs --- a survey (or more problems for Endre to solve)}, An Irregular Mind, Bolyai Society Mathematical Studies {\bf 21} (2010), 561--590.

\bibitem{rrss} {\sc  V.~R\"{o}dl, A.~Ruci\'{n}ski, M.~Schacht and E.~Szemer\'{e}di}, {\em A note on perfect matchings in uniform hypergraphs with large minimum collective degree}, {Commentationes Mathematicae Universitatis Carolinae}~{\bf 49} (2008), 633--636. 

\bibitem{RRS2} {\sc V.~R\"{o}dl, A.~Ruci\'{n}ski and E.~Szemer\'{e}di}, {\em A Dirac-type theorem for 3-uniform hypergraphs}, {Combinatorics, Probability
and Computing}~{\bf 15} (2006), 229--251.

\bibitem{RRS} {\sc V.~R\"{o}dl, A.~Ruci\'{n}ski and E.~Szemer\'{e}di}, {\em Perfect matchings in uniform hypergraphs with large minimum degree}, European Journal of Combinatorics {\bf 27} (2006), 1333--1349.

\bibitem{rodl09} {\sc V.~R\"{o}dl, A.~Ruci\'{n}ski and E.~Szemer\'{e}di}, {\em Perfect matchings in large uniform hypergraphs with large minimum collective egree}, J. Combin. Theory A {\bf 116} (2009), 613--636.


\bibitem{Tutte}  {\sc W.T. Tutte}, {\em The factorization of linear graphs},
{Journal of the London Mathematical Society} {\bf 22} (1947), 107--111.

\bibitem{Zhao} {\sc Y.~Zhao}, {\it Recent advances on Dirac-type problems for hypergraphs}, Recent Trends in Combinatorics, The IMA Volumes in Mathematics and its Applications {\bf 159} (2016), 145--165.
\end{thebibliography}
\end{document}